\pdfoutput=1
\documentclass[11pt]{article}
\usepackage{authblk}
\usepackage[toc,page]{appendix}
\usepackage[top=2cm, bottom=2cm, left=2cm, right=2cm]{geometry}

\usepackage{color}
\usepackage{helvet}         
\usepackage{courier}        
\usepackage{type1cm}        

\usepackage{framed}
\usepackage{tikz}
\usepackage{makeidx}         
\usepackage{graphicx}        
\usepackage{multicol}        
\usepackage[bottom]{footmisc}

\usepackage{amsmath}
\usepackage{amssymb}
\usepackage{bbold}
\usepackage{amsthm}
\usepackage{subcaption}
\usepackage{sidecap}
\usepackage{floatrow}
\usepackage{pdflscape}
\usepackage{pdflscape}
\usepackage{comment}
\usepackage[font=small]{caption}
\usepackage{enumitem}
\usepackage{esint}

\usepackage{thmtools}
\usepackage{thm-restate}

\usepackage{scalerel}
\usepackage{hyperref}

\newtheorem{theorem}{Theorem}
\newtheorem{corollary}[theorem]{Corollary}
\newtheorem{lemma}[theorem]{Lemma}

\newtheorem{proposition}[theorem]{Proposition}

\theoremstyle{remark}

\newtheorem{remark}[theorem]{\bf Remark}
\newtheorem{definition}[theorem]{\bf Definition}
\newtheorem*{speculation}{\bf Conclusion}

\numberwithin{theorem}{section}
\numberwithin{corollary}{section}
\numberwithin{definition}{section}
\numberwithin{lemma}{section}
\numberwithin{conjecture}{section}
\numberwithin{proposition}{section}
\numberwithin{question}{section}
\numberwithin{remark}{section}
\numberwithin{figure}{section}
\numberwithin{equation}{section}

\begin{document}

\title{Uniform Spanning Tree in Topological Polygons, \\
Partition Functions for SLE(8), and \\
Correlations in $c=-2$ Logarithmic CFT}
\bigskip{}
\author[1,3]{Mingchang Liu\thanks{liumc\_prob@163.com}}
\author[2]{Eveliina Peltola\thanks{eveliina.peltola@hcm.uni-bonn.de}}
\author[3]{Hao Wu\thanks{hao.wu.proba@gmail.com.}}
\affil[1]{KTH Royal Institute of Technology, Sweden}
\affil[2]{Aalto University, Finland, and University of Bonn, Germany}
\affil[3]{Tsinghua University, China}

\date{}

%
%

\global\long\def\CR{\mathrm{CR}}
\global\long\def\ST{\mathrm{ST}}
\global\long\def\SF{{}_2\mathrm{SF}}
\global\long\def\cov{\mathrm{cov}}
\global\long\def\dist{\mathrm{dist}}
\global\long\def\SLE{\mathrm{SLE}}
\global\long\def\hSLE{\mathrm{hSLE}}
\global\long\def\CLE{\mathrm{CLE}}
\global\long\def\GFF{\mathrm{GFF}}
\global\long\def\inte{\mathrm{int}}
\global\long\def\ext{\mathrm{ext}}
\global\long\def\inrad{\mathrm{inrad}}
\global\long\def\outrad{\mathrm{outrad}}
\global\long\def\dimH{\mathrm{dim}}
\global\long\def\capa{\mathrm{cap}}
\global\long\def\diam{\mathrm{diam}}
\global\long\def\free{\mathrm{free}}
\global\long\def\hF{{}_2\mathrm{F}_1}
\global\long\def\simple{\mathrm{simple}}
\global\long\def\st{\mathrm{ST}}
\global\long\def\ust{\mathrm{UST}}
\global\long\def\usf{\mathrm{USF}}
\global\long\def\Leb{\mathrm{Leb}}
\global\long\def\LP{\mathrm{LP}}

\global\long\def\eps{\epsilon}
\global\long\def\ov{\overline}
\global\long\def\U{\mathbb{U}}
\global\long\def\T{\mathbb{T}}
\global\long\def\HH{\mathbb{H}}
\global\long\def\LA{\mathcal{A}}
\global\long\def\LB{\mathcal{B}}
\global\long\def\LC{\mathcal{C}}
\global\long\def\LD{\mathcal{D}}
\global\long\def\LF{\mathcal{F}}
\global\long\def\LK{\mathcal{K}}
\global\long\def\LE{\mathcal{E}}
\global\long\def\LG{\mathcal{G}}
\global\long\def\LI{\mathcal{I}}
\global\long\def\LL{\mathcal{L}}
\global\long\def\LM{\mathcal{M}}
\global\long\def\LQ{\mathcal{Q}}
\global\long\def\LR{\mathcal{R}}
\global\long\def\LT{\mathcal{T}}
\global\long\def\LS{\mathcal{S}}
\global\long\def\LU{\mathcal{U}}
\global\long\def\LV{\mathcal{V}}
\global\long\def\LX{\mathcal{X}}
\global\long\def\PartF{\mathcal{Z}}
\global\long\def\LH{\mathcal{H}}
\global\long\def\R{\mathbb{R}}
\global\long\def\C{\mathbb{C}}
\global\long\def\N{\mathbb{N}}
\global\long\def\Z{\mathbb{Z}}
\global\long\def\E{\mathbb{E}}
\global\long\def\PP{\mathbb{P}}
\global\long\def\QQ{\mathbb{Q}}
\global\long\def\A{\mathbb{A}}
\global\long\def\one{\mathbb{1}}
\global\long\def\bn{\mathbf{n}}
\global\long\def\MR{MR}
\global\long\def\cond{\,|\,}
\global\long\def\la{\langle}
\global\long\def\ra{\rangle}
\global\long\def\tree{\Upsilon}
\global\long\def\forest{\Xi}

\global\long\def\sf{{}_2\mathrm{SF}}
\global\long\def\wr{\varrho}

\global\long\def\Im{\operatorname{Im}}
\global\long\def\Re{\operatorname{Re}}

\global\long\def\ud{\mathrm{d}}
\global\long\def\pder#1{\frac{\partial}{\partial#1}}
\global\long\def\pdder#1{\frac{\partial^{2}}{\partial#1^{2}}}
\global\long\def\der#1{\frac{\ud}{\ud#1}}

\global\long\def\bZnn{\mathbb{Z}_{\geq 0}}
\global\long\def\bZpos{\mathbb{Z}_{> 0}}

\global\long\def\Vfunc{\LG}
\global\long\def\gfunc{g^{(\rr)}}
\global\long\def\hfunc{h^{(\rr)}}

\global\long\def\wfunc{\aleph_{\scaleobj{0.75}{+}}}
\global\long\def\wfuncnobranch{\aleph}

\global\long\def\SimplexInt{\rho}
\global\long\def\CubeInt{\widetilde{\rho}}

\global\long\def\ii{\mathfrak{i}}
\global\long\def\rr{\mathfrak{r}}
\global\long\def\chamber{\mathfrak{X}}
\global\long\def\Wchamber{\mathfrak{W}}

\global\long\def\SimplexIntKappa8{\SimplexInt}

\global\long\def\nested{\boldsymbol{\underline{\Cap}}}
\global\long\def\unnested{\boldsymbol{\underline{\cap\cap}}}

\global\long\def\acycle{\vartheta}
\global\long\def\bcycle{\tilde{\acycle}}
\global\long\def\Gloop{\varrho}
\global\long\def\GGloop{\mathcal{P}}

\global\long\def\metric{\mathrm{dist}}

\global\long\def\adj#1{\mathrm{adj}(#1)}

\newcommand{\conn}{\vartheta_{\mathrm{UST}}}
\newcommand{\grconn}{\vartheta_{\mathrm{GRV}}}
\newcommand{\curveSpace}{X}
\newcommand{\Qintegral}{\mathcal{I}}

\maketitle

\begin{center}
\begin{minipage}{0.96\textwidth}
\abstract{
We find explicit SLE($8$) partition functions for the scaling limits of Peano curves
in the uniform spanning tree (UST) in topological polygons with general boundary conditions.
They are given in terms of Coulomb gas integral formulas,
which can also be expressed in terms of determinants involving $a$-periods of a hyperelliptic Riemann surface.
We also identify the crossing probabilities for the UST Peano curves as ratios of these partition functions.

The partition functions are interpreted as correlation functions in a logarithmic conformal field theory (log-CFT) of central charge $c=\! -2$.
Indeed, it is clear from our results that this theory is not a minimal model and exhibits logarithmic phenomena — the limit functions
have logarithmic asymptotic behavior, that we calculate explicitly.
General fusion rules for them could also be inferred from the explicit formulas.
The discovered algebraic structure matches the known Virasoro staggered module classification,
so in this sense, we give a direct probabilistic construction for correlation functions in a log-CFT of central charge $\! -2$ describing the UST model.
}

\bigskip{}

\noindent\textbf{Keywords:} 
(logarithmic) conformal field theory (CFT), correlation function, crossing probability, uniform spanning tree (UST), partition function, Schramm-Loewner evolution (SLE) \\ 

\noindent\textbf{MSC:} 82B20, 60J67, 60K35
\end{minipage}
\end{center}

\newpage

\setcounter{tocdepth}{2}
\tableofcontents

\newpage
\allowdisplaybreaks


\smallskip{}
\section{Introduction}

Random polymer models provide a plethora of interesting phenomena in statistical physics, probability theory,  
and related fields. For instance, many of them show features of criticality 
conjecturally related to conformal invariance. 
A particularly favorable setup for exact and rigorous results is the case of \emph{planar} models.
Indeed, one of the first examples where a conformal invariance result  for a critical model was rigorously verified was 
the planar \emph{uniform spanning tree} (UST):  Lawler, Schramm \&~Werner showed~\cite{LawlerSchrammWernerLERWUST} 
that the random curve traversing along a uniformly chosen spanning tree 
(hereafter referred to as the UST \emph{Peano curve}), 
which can also be thought of as a critical dense polymer, 
converges in the scaling limit to a \emph{Schramm-Loewner evolution} process, $\SLE_\kappa$ with $\kappa = 8$.
In fact, the UST is quite a fruitful model, as it connects to other important models in various ways. For instance, the UST can be generated by loop-erased random walks via Wilson's algorithm~\cite{WilsonUSTLERW}, 
and the ``branch-winding height functions''
of the UST have the Gaussian free field as a scaling limit~\cite{KenyonDominosGFF}. 
This observation also leads to an analogue in the continuum: $\SLE_8$ curves can be coupled with the Gaussian free field as flow lines~\cite{DubedatSLEFreefield, MillerSheffieldIG4, BerestyckiLaslierRayDimersIG}. 
However, even though 
substantial evidence for conformal invariance has been obtained for the UST model,  
there is still no clear conjecture on the full \emph{conformal field theory} (CFT) describing its scaling limit.

In (heuristic) CFT parlance, the planar UST model should be described by some CFT with central charge $c=-2$, 
which remarkably is a \emph{non-unitary} theory --- 
unlike many well-known ones such as the Liouville CFT~\cite{KRV1} 
or the minimal model for the Ising model~\cite{CHI:Ising_CFT}. 
In the non-unitary case there can even exist several theories with the same central charge and conformal weights. 
Physicists have proposed various descriptions for 
a $c=-2$ theory (e.g.,~\cite{Gurarie:log_CFT,
GK:Indecomposable_fusion_products, Kausch:Symplectic_fermions,  Pearce_Rasmussen:Solvable_critical_dense_polymers}), 
involving \emph{logarithmic} fields.
These are fields in the CFT with anomalous behavior, arising from the feature that the Virasoro dilation operator $\mathrm{L}_0$, 
which generates scalings of the physical space-time,
is not diagonalizable (in particular, the Hamiltonian of the theory is not self-adjoint). 
This implies that the space of states has a complicated structure as a representation of the Virasoro algebra, containing non-trivial Jordan blocks for $\mathrm{L}_0$. 
They, in turn, result in logarithmic divergences in the correlation functions of the theory (cf.~Section~\ref{subsec: CFT_intro}).
We will show in the present work
that the correlation functions have a clear probabilistic interpretation, which could be helpful in revealing the underlying complicated algebraic structures. 

\smallbreak

This article concerns a probabilistic model of 
UST Peano curves in (topological) polygons with various boundary conditions.
We obtain scaling limit results 
expressed in terms of explicit quantities --- determinantal expressions and Coulomb gas integrals. 
As a by-product of our results, we establish 
that any (boundary) CFT describing the UST model 
(once again, we say ``any'', as there is no consensus for what the CFT should exactly be) 
must contain fields whose correlation functions have logarithmic divergences.
Specifically, we show that the explicit scaling limit objects are \emph{conformally invariant or covariant}, 
satisfy \emph{BPZ PDEs} at level two, and we derive \emph{explicit fusion rules} in terms of asymptotic behavior (with structure constants also determined), 
that manifestly show the emergence of a \emph{logarithmic CFT} (log-CFT) structure. 
We believe that by providing (to our knowledge) the first mathematically 
rigorous probabilistic results towards a systematic 
log-CFT description of the scaling limit of the UST model, 
we also initiate the building of solid analytical foundations for such a theory, relevant in particular to random geometry and statistical physics. 

\subsection{Summary of main results}

Let us summarize more precisely the main findings of the present work. 
We will consider uniformly chosen spanning trees on subsets of the square lattice, 
focusing on the behavior of the random chordal Peano curves between the tree and its dual.
Specifically, we consider scaling limits of the UST model on polygons with an even number $2N$ of marked boundary points. Thus, we 
fix $N\ge 1$ and a \emph{polygon} $(\Omega; x_1, \ldots, x_{2N})$, that is, 
a bounded simply connected domain $\Omega\subset\C$ such that $\partial\Omega$ is locally connected,
together with distinct marked boundary points $x_1, \ldots, x_{2N}$ in counterclockwise order. 
We encode boundary conditions in \emph{link patterns} (i.e., planar/non-crossing pair partitions)
\begin{align} \label{eq: link pattern ordering}
\begin{split}
& \beta = \{ \{a_1,b_1\},  \{a_2,b_2\},\ldots , \{a_N,b_N\}\} \\
& \textnormal{with link endpoints ordered as } \; 
a_1 < a_2 < \cdots < a_N \textnormal{ and } a_r < b_r , \textnormal{ for all } 1 \leq r \leq N ,  \\
& \textnormal{and such that there are no indices } 1 \leq r , s \leq N \textnormal{ with } a_r < a_s < b_r < b_s , \end{split}
\end{align}
where $\{a_1, b_1,\ldots,  a_N, b_N  \} =  \{1,2,\ldots,2N\}$. 
For convenience, we have chosen a particular ordering of the endpoints of the links $\{a_r,b_r\}$.
We shall denote by $\LP_N \ni \beta$ the set of link patterns $\beta$ of $N$ links. 

\smallbreak

\noindent
Our main results can be summarized as follows:

\begin{itemize}[leftmargin=*]
\item (Theorem~\ref{thm::ust_general}): 
We identify the scaling limits of the UST Peano curves with any boundary condition~$\beta$. 
These are variants of the $\SLE_8$ process, whose partition functions 
(denoted $\LF_\beta$)
are determined by~$\beta$, arising naturally from the discrete holomorphic observable that we employ to derive this result. Let us remark that, 
compared to the earlier results~\cite{LawlerSchrammWernerLERWUST, DubedatEulerIntegralsCommutingSLEs, HanLiuWuUST}, 
the choice and analysis of the general observable is significantly more intricate, and the boundary conditions for it are non-trivial.

\smallbreak

\item (Theorem~\ref{thm::ust_crossing_proba}): 
We find the scaling limits of all crossing probabilities of the UST Peano curves 
as~ratios of  
$\LF_\beta$ and so-called $\SLE_8$ \emph{pure partition functions} $\PartF_\alpha$
(see~\cite{BBK:Multiple_SLEs_and_statistical_mechanics_martingales, DubedatCommutationSLE, Kozdron-Lawler:Configurational_measure_on_mutually_avoiding_SLEs, PeltolaCFTSLE} and references therein for cases with $\kappa < 8$). 
In contrast to earlier lattice-level results~\cite{KenyonWilsonBoundaryPartitionsTreesDimers}, 
we obtain a more complete description of the scaling limits, 
provide explicit formulas in terms of integrals of Coulomb gas type,
and interpret these probabilities in terms of operators in a $c=-2$ log-CFT. 
Interestingly, the Coulomb gas integrals can also be written in terms of $a$-periods of a suitable hyperelliptic Riemann surface (see Proposition~\ref{prop::two_Fs_as_determinants})
--- in particular, they have a determinantal  structure.

\smallbreak

\item (Theorems~\ref{thm::coulombgasintegral}~\&~\ref{thm::ppf}): 
We prove salient properties of the partition functions, some of them 
predicted from conformal field theory~\cite{BPZ:Infinite_conformal_symmetry_in_2D_QFT, Gurarie:log_CFT,
Pearce_Rasmussen:Solvable_critical_dense_polymers},
and others crucial for their probabilistic meaning. 
\end{itemize}

Even though the log-CFT interpretation is rather heuristic, from the formal, algebraic viewpoint 
(in terms of the representation theory of the Virasoro algebra),
we thus provide evidence that any CFT describing 
the planar UST model in the scaling limit must be a non-unitary, logarithmic CFT 
containing specific operator content (cf.~Section~\ref{subsec: CFT_intro}). 
Our results furthermore provide a \emph{probabilistic} construction for such a CFT in terms of correlation functions.

\subsection{$\SLE(8)$ observables: Coulomb gas integrals and pure partition functions}
\label{subsec: Coulomb_intro}

In order to state our scaling limit results, we first describe the objects that determine the scaling limit.
We recommend readers only interested in the scaling limit results per se to first glance at Section~\ref{subsec: UST_intro}
and then return to the present Section~\ref{subsec: Coulomb_intro}. 
However, let us point out that the results of the present section 
greatly differ from the previously considered cases of the critical Ising model ($\kappa=3$)~\cite{IzyurovObservableFree, PeltolaWuCrossingProbaIsing},
multiple LERW ($\kappa=2$)~\cite{KenyonWilsonBoundaryPartitionsTreesDimers, KarrilaUSTBranches, KarrilaKytolaPeltolaCorrelationsLERWUST}, 
critical percolation ($\kappa=6$)~\cite{DubedatEulerIntegralsCommutingSLEs}, 
and critical random-cluster models ($\kappa=16/3$)~\cite{IzyurovMultipleFKIsing, FPW22}. 
Indeed, when $\kappa < 8$, for instance the $\SLE_\kappa$ pure partition functions can be uniquely classified as solutions to a certain PDE boundary value problem~\cite{FloresKlebanPDE, PeltolaWuGlobalMultipleSLEs}, 
whereas in the present case of $\kappa=8$, no classification is known and the techniques of~\cite{FloresKlebanPDE, PeltolaWuGlobalMultipleSLEs} fail.
Furthermore, while the Kac conformal weights $h_{1,1}(\kappa)$ and $h_{1,3}(\kappa)$ for $\kappa < 8$ in Equation~\eqref{eq: Kac weights} determine unambiguously the asymptotic boundary conditions for such a PDE boundary value problem, giving rise to two \emph{distinct} Frobenius exponents,
these exponents coincide when $\kappa = 8$. 
This seemingly innocent property lies at the heart of the logarithmic phenomena in the CFT describing the scaling limit of the UST model.  

\smallbreak

For each $\beta \in \LP_N$ as in~\eqref{eq: link pattern ordering}, we define $\LF_{\beta} \colon  \chamber_{2N} \to \C$, where 
\begin{align*}
\chamber_{2N} := \big\{ \boldsymbol{x} = (x_{1},\ldots,x_{2N}) \in \R^{2N} \colon x_{1} < \cdots < x_{2N} \big\} ,
\end{align*}
to be a ``Coulomb gas integral function'' (for $\kappa=8$) 
\begin{align}
\label{eqn::coulombgasintegral}
\LF_\beta (\boldsymbol{x}) :=  \; &
\prod_{1\leq i<j\leq 2N}(x_{j}-x_{i})^{1/4} 
\landupint_{x_{a_1}}^{x_{b_1}} 
\cdots \landupint_{x_{a_N}}^{x_{b_N}}
\prod_{1\leq r<s\leq N}(u_{s}-u_{r}) 
\; \prod_{r=1}^{N}
\frac{\ud u_r}{\prod_{k=1}^{2N} (u_{r}-x_{k})^{1/2}} ,
\end{align}
where the branch of the multivalued integrand is chosen to be real and positive when
\begin{align*}
x_{a_r} < u_r < x_{a_r+1} , \quad \textnormal{ for all } r \in \{1,2,\ldots N \} ,
\end{align*}
and the integration in~\eqref{eqn::coulombgasintegral} is understood so that  
the integration variables $u_r$ avoid the ramification points $x_1, x_2, \ldots, x_{2N}$ by encircling them from the upper half-plane (see Section~\ref{sec::coulombgasintegrals_new}). 
With such a branch choice, $\LF_{\beta}$ actually takes values in $(0,\infty)$,
see Theorem~\ref{thm::coulombgasintegral}.
This fact is crucial for the probabilistic interpretation of these functions --- however, it is not obvious. 

Quite a specific feature to the present case of $\kappa=8$ is that the function~\eqref{eqn::coulombgasintegral} also equals, up to a multiplicative factor, 
an integral of the Vandermonde determinant: 
\begin{align} \label{eq: VandermondeF}
\LF_\beta (\boldsymbol{x}) =  \; & \prod_{1\leq i<j\leq 2N}(x_{j}-x_{i})^{1/4} \times 
|\det P_{\beta}(\boldsymbol{x})| , \\
\nonumber
\textnormal{where} \qquad 
P_{\beta}(\boldsymbol{x}) := \; & \Big( \landupint_{x_{a_r}}^{x_{b_r}} \frac{u^{ s-1 }\ud u}{\prod_{j=1}^{2N}(u-x_{j})^{1/2}}   \Big)_{r,s=1}^{N} , \qquad \boldsymbol{x}\in\chamber_{2N} ,
\end{align}
is a matrix which can also be written in terms of $a$-periods of a suitable hyperelliptic Riemann surface --- see Proposition~\ref{prop::two_Fs_as_determinants} and Equation~(\ref{eq::relate_A_and_P},~\ref{eq::relate_det_and_circle}). 
Such a determinantal structure can be viewed as a feature of the \emph{fermionic} nature of the UST model (CFT with $c=-2$).

Historically, these Coulomb gas integrals stem from conformal field theory~\cite{DF-multipoint_correlation_functions, DubedatEulerIntegralsCommutingSLEs, KytolaPeltolaConformalCovBoundaryCorrelation}, where they have been used as a general ansatz to find formulas for correlation functions. Specifically to our case, 
we seek correlation functions of so-called degenerate fields, 
which should satisfy a system of ``BPZ PDEs'' attributed to 
Belavin, Polyakov \&~Zamolodchikov~\cite{BPZ:Infinite_conformal_symmetry_in_2D_QFT},
\begin{align}\label{eqn::USTPDE}
\tag{\textnormal{PDE}}
\bigg[ 
4 \pdder{x_{j}}
+ \sum_{i \neq j} \Big( \frac{2}{x_{i}-x_{j}} \pder{x_{i}} 
+ \frac{1/4}{(x_{i}-x_{j})^{2}} \Big) \bigg]
\LF(x_1,\ldots,x_{2N}) =  0 , \quad \textnormal{for all } j \in \{1,2,\ldots,2N\} ,
\end{align}
and the specific covariance property 
\begin{align}\label{eqn::USTCOV}
\tag{\textnormal{COV}}
\LF(x_{1},\ldots,x_{2N})  = 
\prod_{i=1}^{2N} \varphi'(x_{i})^{-1/8} 
\times \LF(\varphi(x_{1}),\ldots,\varphi(x_{2N})) ,
\end{align}
for all M\"obius maps $\varphi$ of the upper half-plane 
$\HH := \{ z \in \C \colon \Im(z) > 0 \}$ 
such that $\varphi(x_{1}) < \cdots < \varphi(x_{2N})$.

\begin{theorem} \label{thm::coulombgasintegral}
The functions $\LF_{\beta}$ defined in~\eqref{eqn::coulombgasintegral} satisfy
the PDE system~\eqref{eqn::USTPDE}, M\"obius covariance~\eqref{eqn::USTCOV}, and the following further properties. 
\begin{itemize}
\item[\textnormal{(POS)}] 
\textnormal{\bf Positivity:} 
For each $N \ge 1$ and $\beta\in\LP_N$, we have $\LF_{\beta}(\boldsymbol{x})>0$, for all $\boldsymbol{x}\in\chamber_{2N}$.

\smallbreak

\item[\textnormal{(ASY)}] 
\textnormal{\bf Asymptotics:} 
With $\LF_{\emptyset} \equiv 1$ for the empty link pattern $\emptyset \in \LP_0$, the collection $\{\LF_{\beta} \colon \beta\in\LP_N\}$ satisfies the following recursive asymptotics property.
Fix $N \ge 1$ and $j \in \{1,2, \ldots, 2N-1 \}$. 
Then, we have\footnote{Throughout, we use the cyclic indexing convention $x_{2N+1} := x_{1}$ etc.}
\begin{align}
\label{eqn::USTASY1} 
\tag{$\LF_{\beta}$\textnormal{-ASY}$_{1,1}$}
\; & \lim_{x_j,x_{j+1}\to\xi} \frac{\LF_{\beta}(\boldsymbol{x})}{ (x_{j+1}-x_j)^{1/4} }
= \pi \, \LF_{\beta/\{j, j+1\}}(\boldsymbol{\ddot{x}}_j) , 
& \textnormal{if }\{j, j+1\}\in\beta , \\
\label{eqn::USTASY2}
\tag{$\LF_{\beta}$\textnormal{-ASY}$_{1,3}$}
\; & \lim_{x_j , x_{j+1} \to \xi} \frac{\LF_{\beta}(\boldsymbol{x})}{ (x_{j+1} - x_j)^{1/4} |\log(x_{j+1}-x_j)| } 
= \LF_{\wp_j(\beta)/\{j,j+1\}}(\boldsymbol{\ddot{x}}_j) , 
& \textnormal{if }\{j, j+1\} \not\in \beta ,
\end{align}
where
\begin{align} \label{eqn::bs_notation}
\begin{split}
\boldsymbol{x} = \; & (x_1, \ldots, x_{2N}) \in \chamber_{2N} , 
\\
\boldsymbol{\ddot{x}}_j = \; & (x_1, \ldots, x_{j-1}, x_{j+2}, \ldots, x_{2N}) \in \chamber_{2N-2} ,
\end{split}
\end{align}
and $\xi \in (x_{j-1}, x_{j+2})$, and 
where $\beta/\{j,j+1\} \in \LP_{N-1}$ denotes the link pattern obtained from $\beta$ by removing the link $\{j,j+1\}$ and relabeling the remaining indices by $1, 2, \ldots, 2N-2$, 
and $\wp_j$  
is the ``tying operation'' defined by 
\begin{align*}
\wp_j \colon \LP_N\to \LP_N , \quad
\wp_j(\beta) = 
\big(\beta\setminus(\{j,\ell_1\}, \{j+1, \ell_2\})\big)\cup \{j,j+1\}\cup \{\ell_1, \ell_2\} , 
\end{align*}  
where $\ell_1$ \textnormal{(}resp.~$\ell_2$\textnormal{)} 
is the pair of $j$ \textnormal{(}resp.~$j+1$\textnormal{)} in $\beta$ \textnormal{(}and $\{j,\ell_1\}, \{j+1, \ell_2\}, \{\ell_1, \ell_2\}$ are unordered\textnormal{)}.
\begin{align*}
\vcenter{\hbox{\includegraphics[scale=0.25]{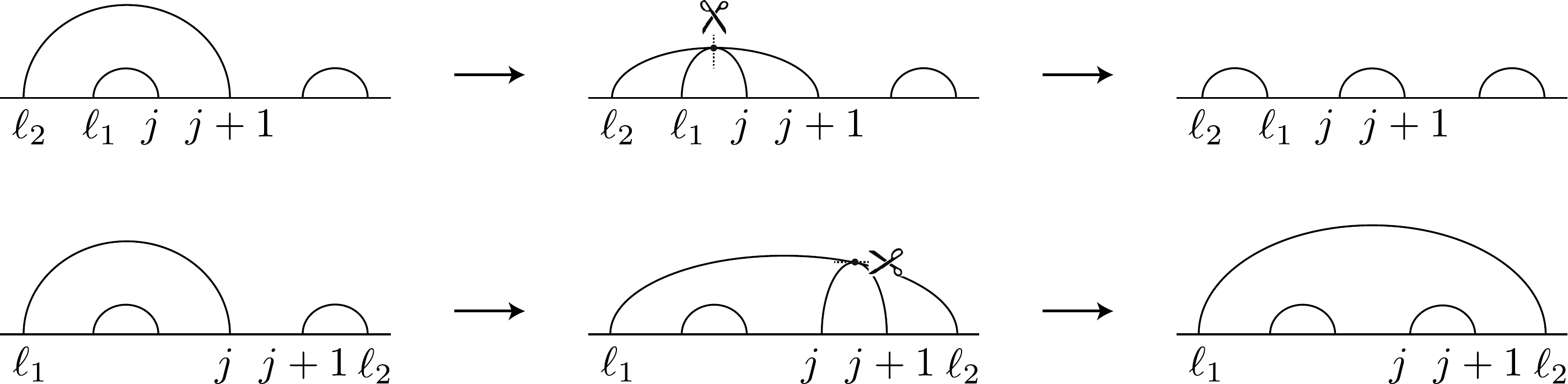}}} 
\end{align*}

\smallbreak

\item[\textnormal{(LIN)}] 
\textnormal{\bf Linear independence:} 
The functions $\{\LF_{\beta} \colon \beta\in\LP_N\}$ are linearly independent. 
\end{itemize}
\end{theorem}

In short, the proof of Theorem~\ref{thm::coulombgasintegral} comprises 
Proposition~\ref{prop: full Mobius covariance F},
Proposition~\ref{prop: PDEs F} (or Corollary~\ref{cor::coulombgasintegralPDE}), 
Proposition~\ref{prop: Asymptotics F}, 
Proposition~\ref{prop: positivity},
and the conclusion in Section~\ref{subsec::ppf_concluding}.
\begin{itemize}[leftmargin=*]
\item The BPZ PDE system~\eqref{eqn::USTPDE} can be verified by the following 
argument for the Coulomb gas integrals. The integration contours in~\eqref{eqn::coulombgasintegral} can be written as
a closed surface in a suitable homology, and the integrand after being hit by the differential operator $\mathcal{D}^{(j)}$ in~\eqref{eqn::USTPDE} 
gives an exact form. 
Therefore, Stokes' theorem can be used (with care) to argue that~\eqref{eqn::coulombgasintegral} is a solution to~\eqref{eqn::USTPDE}. We perform this argument in 
Proposition~\ref{prop: PDEs F}. The strategy is explained in detail in~\cite[Section~1]{KytolaPeltolaConformalCovBoundaryCorrelation}\footnote{The results in~\cite{KytolaPeltolaConformalCovBoundaryCorrelation} concern the case of irrational values of $\kappa$, but the same idea for this part also works for rational $\kappa$.}. 

\smallbreak

\item The M\"obius covariance~\eqref{eqn::USTCOV} is immediate for translations and scalings, but the verification of it for special conformal transformations is surprisingly troublesome.
In our special case where $\kappa=8$, the determinantal structure~\eqref{eq: VandermondeF} of $\LF_\beta$ is helpful 
(cf.~Proposition~\ref{prop: full Mobius covariance F}).

\smallbreak

\item In general, it is very difficult to establish positivity \textnormal{(POS)} for Coulomb gas type integrals\footnote{For $\kappa \leq 6$, positivity results have been established via a construction of $\SLE_\kappa$ partition functions explicitly in terms of probabilistic quantities, such as the Brownian loop measure and multiple $\SLE_\kappa$~\cite{Kozdron-Lawler:Configurational_measure_on_mutually_avoiding_SLEs, LawlerPartitionFunctionsSLE, PeltolaWuGlobalMultipleSLEs, WuHyperSLE}.}. 
Once again, in our special case the positivity is guaranteed by the determinantal structure, see Proposition~\ref{prop: positivity}. 
This property is absolutely essential for the probabilistic interpretation and usage of $\LF_\beta$ for the scaling limit results stated below in Section~\ref{subsec: UST_intro}; 
indeed, a priori $\LF_\beta$ are complicated complex valued functions expressed as iterated integrals~\eqref{eqn::coulombgasintegral}.

\smallbreak

\item Of the asymptotics properties in \textnormal{(ASY)},
the generic one,~\eqref{eqn::USTASY1}, is very easy to verify (Lemma~\ref{lem: Asymptotics1 A}), whereas the logarithmic one,~\eqref{eqn::USTASY2}, needs more detailed analysis (Lemma~\ref{lem: Asymptotics2 A}). 
These asymptotics properties are motivated by fusion rules in log-CFT, 
that in particular involve the logarithmic correction in one of the fusion channels. See Section~\ref{subsec: CFT_intro} for discussion and literature.

\smallbreak

\item Lastly, the linear independence of the functions $\LF_\beta$ is a consequence of their different asymptotic properties, as explained in the end of Section~\ref{subsec::ppf_concluding}.
\end{itemize}

Theorem~\ref{thm::coulombgasintegral} shows that 
$\{\LF_{\beta} \colon \beta\in\LP_N\}$ forms a basis for a solution space\footnote{It is not known to us what the dimension of the full solution space to the PDE system~\eqref{eqn::USTPDE} is. However, it is plausible that imposing the additional constraint~\eqref{eqn::USTCOV} and possibly a bound for the growth of the solutions analogous to~\cite[Part~1,~Eq.~(20)]{FloresKlebanPDE}, the dimension of the restricted solution space would equal the Catalan number
$\frac{1}{N+1} \binom{2N}{N} = |\LP_N|$.
}
of the PDE system~\eqref{eqn::USTPDE}. 
There is another useful basis consisting of \emph{pure partition functions} $\{\PartF_{\alpha} \colon  \alpha \in \LP_N\}$.
To explain how these two bases are related, let us recall that a \emph{meander} formed from two link patterns $\alpha, \beta\in \LP_N$ is the planar diagram obtained by placing $\alpha$ and the horizontal reflection of $\beta$ on top of each other\footnote{Some authors call these ``meandric systems.''}:
\begin{align} \label{eqn::meander_example}
\alpha \quad = \quad \vcenter{\hbox{\includegraphics[scale=0.275]{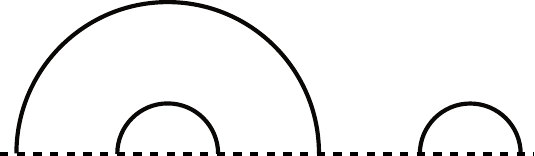}}} 
\quad  , \quad  
\beta \quad = \quad\vcenter{\hbox{\includegraphics[scale=0.275]{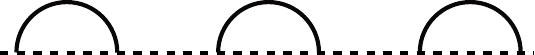}}} 
\quad\quad \Longrightarrow \quad\quad
\vcenter{\hbox{\includegraphics[scale=0.275]{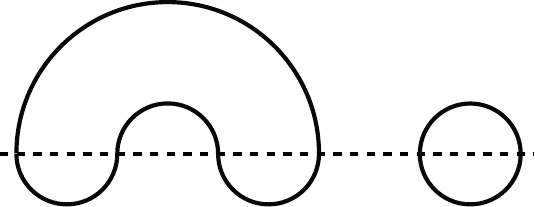}}} 
\end{align}
We define the (renormalized, symmetric) \emph{meander matrix} entries $\{\LM_{\alpha, \beta} \colon \alpha, \beta\in\LP_N\}$ as
\begin{align} \label{eqn::renormalized_meander_matrix}
\LM_{\alpha, \beta} :=
\begin{cases}
1, &\textnormal{if the meander formed from $\alpha$ and $\beta$ has one loop,}\\
0,&\textnormal{otherwise.}
\end{cases}
\end{align}
We also set $\LM_{\emptyset, \emptyset} \equiv 1$ by convention.
By~\cite[Eq.~(5.18)]{FrancescoGolinelliGuitterMeanders}, 
the 
meander matrix~\eqref{eqn::renormalized_meander_matrix} is invertible. 
We then define the pure partition functions 
$\PartF_{\alpha} \colon \chamber_{2N}\to \C$ as
\begin{align}\label{eqn::ppf_def}
\PartF_{\alpha} (\boldsymbol{x}) := \sum_{\beta \in \LP_N} \LM_{\alpha, \beta}^{-1} \, \LF_{\beta} (\boldsymbol{x})
 , \qquad \boldsymbol{x}\in\chamber_{2N} .  
\end{align}

\begin{theorem} \label{thm::ppf}
The pure partition functions defined in~\eqref{eqn::ppf_def} satisfy
the PDE system~\eqref{eqn::USTPDE}, M\"obius covariance~\eqref{eqn::USTCOV}, and the following further properties.
\begin{itemize}
\item[\textnormal{(POS)}] 
\textnormal{\bf Positivity:} 
For each $N \ge 1$ and $\beta\in\LP_N$, we have $\PartF_{\alpha}(\boldsymbol{x})>0$, for all $\boldsymbol{x}\in\chamber_{2N}$.

\smallbreak

\item[\textnormal{(ASY)}] 
\textnormal{\bf Asymptotics:} 
With $\PartF_{\emptyset} \equiv 1$ for the empty link pattern $\emptyset \in \LP_0$, the collection $\{\PartF_{\alpha} \colon \alpha\in\LP_N\}$ satisfies the following recursive asymptotics property.
Fix $N \ge 1$ and $j \in \{1,2, \ldots, 2N-1 \}$.
Then, for all $\xi \in (x_{j-1}, x_{j+2})$, using the notation~\eqref{eqn::bs_notation}, we have
\begin{align}
\label{eqn::ppf_asy1}
\tag{$\PartF_{\alpha}$\textnormal{-ASY}$_{1,1}$}
\; & \lim_{x_j , x_{j+1} \to \xi} 
\frac{\PartF_{\alpha}(\boldsymbol{x})}{(x_{j+1} - x_j)^{1/4} |\log(x_{j+1}-x_j)|} 
= \PartF_{\alpha/\{j,j+1\}}(\boldsymbol{\ddot{x}}_j),  &\quad \textnormal{if } \{j, j+1\}\in\alpha ,\\
\label{eqn::ppf_asy2}
\tag{$\PartF_{\alpha}$\textnormal{-ASY}$_{1,3}$}
\; & \lim_{x_j , x_{j+1} \to \xi} 
\frac{\PartF_{\alpha}(\boldsymbol{x})}{(x_{j+1} - x_j)^{1/4}} 
= \pi \, \PartF_{\wp_j(\alpha)/\{j,j+1\}}(\boldsymbol{\ddot{x}}_j), & \quad \textnormal{if }\{j,j+1\}\not\in\alpha .
\end{align}

\smallbreak

\item[\textnormal{(LIN)}] 
\textnormal{\bf Linear independence:} 
The functions $\{\PartF_{\alpha} \colon \alpha \in\LP_N\}$ are linearly independent. 
\end{itemize}
\end{theorem}

The proof of Theorem~\ref{thm::ppf} comprises 
Lemmas~\ref{lem::ppf_asy1} and~\ref{lem::ppf_asy2} in Section~\ref{sec::ppf}, and the conclusion 
in Section~\ref{subsec::ppf_concluding}.
The proof also uses Theorem~\ref{thm::coulombgasintegral}.
The main difficulties are to prove the asymptotic properties \textnormal{(ASY)} in Section~\ref{subsec::ppf_asy},
and the positivity \textnormal{(POS)} in Section~\ref{subsec::ppf_concluding} 
--- importantly, both use
the fact that $\PartF_{\alpha}$ are related to UST crossing probabilities. 

\smallbreak

Generalizing the covariance property~\eqref{eqn::USTCOV}, we extend the definitions of $\LF_{\beta}$ and $\PartF_{\alpha}$ to general polygons, whenever the derivatives of the associated conformal map in~\eqref{eqn::USTCOV} are defined. 
Thus, we set 
\begin{align*}
F(\Omega; x_1, \ldots, x_{2N}) 
:= \prod_{j=1}^{2N} |\varphi'(x_j)|^{-1/8} \times F(\varphi(x_1), \ldots, \varphi(x_{2N})) ,
\quad \textnormal{where} \quad
F = \LF_{\beta} \textnormal{ or } \PartF_{\alpha} ,
\end{align*}
and where $\varphi$ is any conformal map from $\Omega$ onto $\HH$ with $\varphi(x_1)<\cdots<\varphi(x_{2N})$,
assuming that the marked boundary points 
$x_1, \ldots, x_{2N}$ lie on sufficiently regular boundary segments (e.g. $C^{1+\eps}$ for some $\eps>0$). 

\smallbreak

In general, a \emph{partition function} (with $\kappa=8$) will refer to a positive smooth function $\PartF \colon \chamber_{2N} \to\R_{>0}$ satisfying the BPZ PDE system~\eqref{eqn::USTPDE} 
and M\"obius covariance~\eqref{eqn::USTCOV}. 
We can use any partition function to define a 
\emph{Loewner chain associated to} $\PartF$: in the upper half-plane $\HH$, started from $x_i \in \R$, 
and with marked points $(x_1, \ldots, x_{i-1}, x_{i+1}, \ldots, x_{2N})$,
this is the Loewner chain driven by 
the solution $W$ to the stochastic differential equations (SDEs) 
\begin{align}\label{eqn::driving_general}
\begin{cases}
\ud W_t = \sqrt{8} \, \ud B_t + 8 \, (\partial_i\log \PartF)(V_t^{1}, \ldots, V_t^{i-1}, W_t, V_t^{i+1}, \ldots, V_t^{2N}) \, \ud t, \\
\ud V_t^j =\frac{2 \, \ud t}{V_t^j-W_t},\\ 
W_0 = x_{i},\\
V_0^j=x_j, \quad j\in\{1, \ldots, i-1, i+1, \ldots, 2N\} .
\end{cases}
\end{align}
This process is well-defined up to the first time when either $x_{i-1}$ or $x_{i+1}$ is \emph{swallowed} (i.e., when the denominator in the SDE~\eqref{eqn::driving_general} blows up).
The functions $\LF_\beta$ and $\PartF_\alpha$ are examples of partition functions.

\subsection{Scaling limit results: Uniform spanning tree in polygons}
\label{subsec: UST_intro}

We now consider UST on a scaled square lattice (see the precise formulation in Section~\ref{sec::ust}). Suppose that $(\Omega^{\delta, \diamond}; x_1^{\delta, \diamond}, \ldots, x_{2N}^{\delta, \diamond})$ is a sequence of medial polygons 
on $\delta(\Z^2)^{\diamond}$. 
A common notion of convergence for such a sequence is termed after Carath\'{e}odory (see Section~\ref{subsec::ust_cvg_observable}).
However, for our purposes, the following stronger notion of convergence phrased in terms of a curve metric is relevant. 
The set $\curveSpace$ of planar oriented curves, that is, continuous mappings from $[0,1]$ to $\C$ modulo reparameterization, is a complete and separable metric space 
\begin{align}\label{eqn::metric_curvesspace}
(\curveSpace, \metric) , \qquad
\metric(\eta_1, \eta_2) 
:= \inf_{\psi_1, \psi_2} \sup_{t\in[0,1]} \big| \eta_1(\psi_1(t)) - \eta_2(\psi_2(t)) \big| ,
\end{align}
where the infimum is taken over all increasing  homeomorphisms $\psi_1, \psi_2 \colon [0,1]\to [0,1]$. 
We will assume that $(\Omega^{\delta, \diamond}; x_1^{\delta, \diamond}, \ldots, x_{2N}^{\delta, \diamond})$ converges to a polygon $(\Omega; x_1, \ldots, x_{2N})$ in the following sense: 
\begin{align} \label{eqn::polygon_cvg}
\begin{split}
\; & \textnormal{there exists} \quad C \in (0,\infty) \quad \textnormal{such that } \\ 
\; & \metric \big( (x_i^{\delta, \diamond} \, x_{i+1}^{\delta, \diamond}),(x_i \, x_{i+1}) \big)\le C \, \delta , 
\quad\textnormal{for all } i\in \{1,2, \ldots, 2N\} ,
\end{split}
\end{align}
where $(x \, y)$ denotes the counterclockwise boundary arc between $x$ and $y$. 
Note that such a convergence also implies convergence in the Carath\'{e}odory sense.

\begin{figure}[ht!]
\begin{center}
\includegraphics[width=0.45\textwidth]{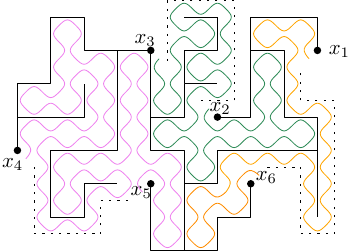}
\end{center}
\caption{\label{fig::polygon} 
For the UST 
in a polygon with six marked points on the boundary, 
with the boundary arcs $(x_1 \, x_2), (x_3 \, x_4), (x_5 \, x_6)$ wired, 
there are three Peano curves connecting $\{x_1, x_2, x_3, x_4, x_5, x_6\}$ pairwise.}
\end{figure}

Next, let $\Omega^{\delta}\subset\delta\Z^2$ be the graph on the primal lattice corresponding to $\Omega^{\delta, \diamond}$. Consider the primal polygon $(\Omega^{\delta}; x_1^{\delta}, \ldots, x_{2N}^{\delta})$ and spanning trees on it, with the following boundary conditions
(see Figure~\ref{fig::polygon}):
first, every other boundary arc is wired,
\begin{align*} 
(x_{2r-1}^{\delta} \, x_{2r}^{\delta})\textnormal{ is wired } \quad \textnormal{for all } r \in \{1,2,\ldots,N\},
\end{align*}
and second, these $N$ wired arcs are further wired together
according to a 
non-crossing partition outside of $\Omega^{\delta}$. 
Note that there is a bijection between non-crossing partitions of the $N$ wired boundary arcs 
and planar link patterns with $N$ links, illustrated in Figure~\ref{fig::8pointsE_meander}. 
Hence, we may encode the boundary condition by a label $\beta\in\LP_N$,
and we thus speak of the UST \emph{with boundary condition} (b.c.) $\beta$. 
Let $\tree_{\delta}$ be a uniformly chosen spanning tree on $\Omega^{\delta}$ 
with b.c. $\beta$. 
Then, there exist $N$ curves on the medial lattice 
$\Omega^{\delta, \diamond}$
running along the tree and connecting $\{x_1^{\delta, \diamond}, \ldots, x_{2N}^{\delta, \diamond}\}$ pairwise. 
We call them UST \emph{Peano curves}. The goal of this section is to describe the scaling limit of these Peano curves.

\smallbreak

Our first result identifies the scaling limits of the UST Peano curves with $\SLE_8$ type processes having specific partition functions, given exactly by the functions
$\LF_{\beta}$ of Theorem~\ref{thm::coulombgasintegral}, defined in~\eqref{eqn::coulombgasintegral}. 
For $N=1$, we have 
$\LF_{\vcenter{\hbox{\includegraphics[scale=0.2]{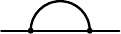}}}}(x_1,x_2)
= \pi \, (x_2-x_1)^{1/4}$
(cf.~Lemma~\ref{lem: N=1})
and the limit process in Theorem~\ref{thm::ust_general}
is the chordal $\SLE_8$ in $\Omega$ between $x_1$ and $x_2$~\cite{LawlerSchrammWernerLERWUST}. 

\begin{theorem} \label{thm::ust_general}
Fix 
a polygon $(\Omega; x_1, \ldots, x_{2N})$ whose boundary $\partial\Omega$ is  a $C^1$-Jordan curve.
Fix also a link pattern $\beta \in \LP_N$. 
Suppose that a sequence $(\Omega^{\delta, \diamond}; x_1^{\delta, \diamond}, \ldots, x_{2N}^{\delta, \diamond})$ of medial polygons converges to $(\Omega; x_1, \ldots, x_{2N})$ 
in the sense~\eqref{eqn::polygon_cvg}.
Consider the UST on the primal polygon $(\Omega^{\delta}; x_1^{\delta}, \ldots, x_{2N}^{\delta})$ with boundary condition
$\beta$.
For each $i\in\{1,2, \ldots, 2N\}$, let $\eta_i^{\delta}$ be the Peano curve started from $x_{i}^{\delta, \diamond}$. 
Let $\varphi$ 
be any conformal map from $\Omega$ onto $\HH$ such that $\varphi(x_1)<\cdots<\varphi(x_{2N})$. 
Then, $\eta_i^{\delta}$ converges weakly to the image under $\varphi^{-1}$ of the Loewner chain with driving function solving the following SDEs, up to 
the first time when $\varphi(x_{i-1})$ or $\varphi(x_{i+1})$ is swallowed\textnormal{:} 
\begin{align}\label{eqn::ust_polygon_driving_general}
\begin{cases}
\ud W_t = \sqrt{8} \, \ud B_t + 8 \, (\partial_i\log \LF_{\beta})(V_t^{1}, \ldots, V_t^{i-1}, W_t, V_t^{i+1}, \ldots, V_t^{2N}) \, \ud t, \\
\ud V_t^j =\frac{2 \, \ud t}{V_t^j-W_t},\\ 
W_0 = \varphi(x_{i}) ,\\
V_0^j=\varphi(x_j), \quad j\in\{1, \ldots, i-1, i+1, \ldots, 2N\} .
\end{cases}
\end{align}
\end{theorem}

The main issue to prove Theorem~\ref{thm::ust_general} is to identify the limit and show its uniqueness. To this end, one can use a discrete holomorphic observable, which is natural when only one or two curves are present (in that case, there is at most one free parameter after fixing three parameters via conformal invariance). 
Dub\'edat proposed a formula for the simplest case of
\begin{align} \label{eqn::unnested}
\beta = \unnested := \{\{1,2\} , \{3,4\} , \ldots, \{2N-1, 2N\}\} ,
\end{align}
in his article~\cite{DubedatEulerIntegralsCommutingSLEs} but without proof.
His formula is different from ours at first sight, but it follows from Proposition~\ref{prop::two_Fs_as_determinants} that they are actually the same.
The recent work~\cite{HanLiuWuUST} concerns the case of  
$N=2$, which is solvable by an ordinary differential equation\footnote{H.W. learned the observable in~\cite{HanLiuWuUST} from a master course delivered by S.~Smirnov in 2015, but we are not able to identify a published reference.}.
However, the general $\beta$ involving non-trivial conformal moduli is significantly more difficult.

\smallbreak

The proof of Theorem~\ref{thm::ust_general} is given in Section~\ref{sec::ust}. 
Roughly, it follows the standard strategy: 
first, we need precompactness (tightness) of the family 
$(\eta_i^{\delta})_{\delta>0}$ (from well-known arguments, cf.~Lemma~\ref{lem::Peanocurve_tight}); 
second, we construct a discrete martingale observable (Sections~\ref{subsec::holo_general}--\ref{subsec::ust_cvg_observable}); 
and third, we identify the subsequential limits $\phi_{\beta}$ through the observable (Section~\ref{subsec::ust_Peano_conv}). 
The identification step involves deriving the expansion of the observable $\phi_{\beta}(z)$ 
as $z$ approaches one of the marked points 
to a certain precision, 
and relating the expansion coefficients explicitly to the partition function $\LF_{\beta}$ 
(Lemmas~\ref{lem::conformal_general_expansion}~\&~\ref{lem::ust_general_goal_new} in Section~\ref{subsec::expansion}). 
The observable $\phi_{\beta}$ from Proposition~\ref{prop::holo_cvg} also gives the scaling limit distribution of the loop-erased random walk branch in the UST, see~\cite{LiuWuLERW}.

Interestingly enough, the scaling limit of the observable $\phi_{\beta}$ (see Proposition~\ref{prop::holo_cvg})
can be written, on the one hand, as an abelian integral (see Equation~\eqref{eq: observable_limit_detA-DetC}), 
\begin{align} \label{eq: observable another formula}
\phi_{\beta}(z) 
= \; & - \frac{\chi_{\beta}(z)}{\chi_{\beta}(x_{b_1})}
\end{align}
involving the $a$-periods discussed in Section~\ref{sec::coulombgasintegrals_new}, 
\begin{align*}
\chi_{\beta}(z) := 
\landupint_{x_1}^z \ud u_1 
\landupint_{x_{a_{2}}}^{x_{b_{2}}} \ud u_2 
\cdots \landupint_{x_{a_{N-1}}}^{x_{b_{N-1}}} \ud u_{N-1} 
\prod_{1\leq r<s\leq N-1}(u_{s}-u_{r}) 
\; \prod_{r=1}^{N-1}
\; \frac{\ud u_r}{\prod_{k=1}^{2N} (u_{r}-x_{k})^{1/2}} ,
\end{align*}
and on the other hand, as a degenerate  
Schwarz-Christoffel conformal map (see Appendix~\ref{app::SC_mappings}),
\begin{align} \label{eq: Schwarz-Christoffel formula}
\phi_{\beta}(z) 
= \landupint_{x_1}^z \frac{\prod_{\ell=1}^{N-2}(u-\mu_\ell) \, \ud u}{\prod_{j=1}^{2N}(u-x_{j})^{1/2}} 
\bigg(\landupint_{x_1}^{x_{b_1}} \frac{\prod_{\ell=1}^{N-2}(u-\mu_\ell) \, \ud u}{\prod_{j=1}^{2N}(u-x_{j})^{1/2}} \bigg)^{-1} , \qquad z \in \overline{\HH} ,
\end{align}
where the accessory parameters $\mu_1, \ldots, \mu_{N-2} \in \R$ are mapped to the tips of the slits in the image of $\phi_{\beta}$ --- see Figure~\ref{fig::slitrectangle} for an example.

\begin{figure}[ht!]
\begin{center}
\includegraphics[width=0.7\textwidth]{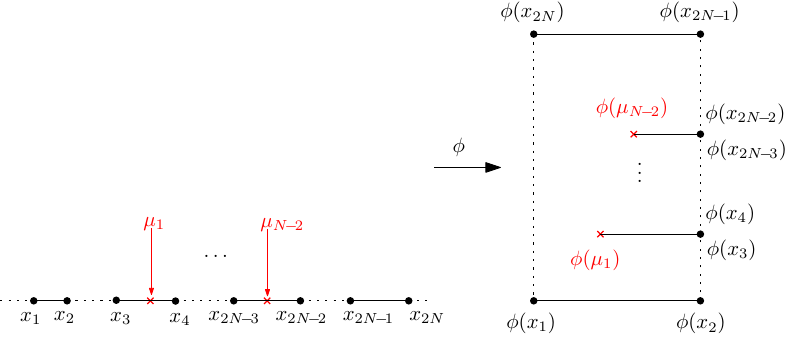}
\end{center}
\caption{\label{fig::slitrectangle}
Illustration of the Schwarz-Christoffel type conformal map $\phi_{\beta}$~\eqref{eq: Schwarz-Christoffel formula} in the case where the link pattern $\beta = \{\{1,2\}, \{3,4\}, \ldots, \{2N-1,2N\}\}$ is the simplest boundary condition. The accessory parameters $\mu_1, \ldots, \mu_{N-2} \in \R$ are mapped to the tips of the slits in the image of $\phi_{\beta}$.}
\end{figure}

\smallbreak

In our second scaling limit result, we identify the scaling limits of the crossing probabilities of the Peano curves as ratios of pure partition functions $\PartF_{\alpha}$ 
with the partition functions $\LF_{\beta}$,
the latter 
arising from the choice of the boundary condition $\beta$.
Of course, the internal crossing pattern encoded in $\alpha$ must be compatible with the b.c. $\beta$, which is exactly what the meander matrix~\eqref{eqn::renormalized_meander_matrix} ensures. 
We denote by $\PP_{\beta}^{\delta}$ the law of the UST with b.c. $\beta \in \LP_N$. 

\begin{restatable}{theorem}{ustcrossingproba}
\label{thm::ust_crossing_proba}
Assume the same setup as in Theorem~\ref{thm::ust_general}. The endpoints of the $N$ Peano curves give rise to a random planar link pattern $\conn^{\delta}$ in $\LP_N$. For any $\alpha\in\LP_N$, we have
\begin{align}\label{eqn::ust_corssing_proba}
\lim_{\delta\to 0} \PP_{\beta}^{\delta}[\conn^{\delta}=\alpha] 
= \LM_{\alpha, \beta} \, \frac{\PartF_{\alpha}(\Omega; x_1, \ldots, x_{2N})}{\LF_{\beta}(\Omega; x_1, \ldots, x_{2N})},
\end{align}
where $\LF_{\beta}$, $\LM_{\alpha,\beta}$, and $\PartF_{\alpha}$ are defined respectively in~\eqref{eqn::coulombgasintegral}, \eqref{eqn::renormalized_meander_matrix}, and~\eqref{eqn::ppf_def}. 
\end{restatable}

Let us briefly summarize the strategy for the proof of Theorem~\ref{thm::ust_crossing_proba}, 
given in detail in Sections~\ref{subsec::ust_crossingproba}--\ref{subsec::ppf_asy}.
Roughly, it requires two inputs. 
First, the scaling limit of the law of $\eta_i^{\delta}$ is given by the Loewner chain associated to $\LF_{\beta}$, which is provided by Theorem~\ref{thm::ust_general}. 
Second, the scaling limit of the probability $\PP_{\beta}^{\delta}[\conn^{\delta}=\alpha]$, denoted by $p^{\alpha}_{\beta}$, exists and is conformally invariant (Proposition~\ref{prop::proba_limit}). 
Our proof of this second fact relies on the earlier work by Kenyon \&~Wilson~\cite{KenyonWilsonBoundaryPartitionsTreesDimers}. 
Lastly, with these two inputs at hand, 
we consider the conditional law 
of $\eta_i^{\delta}$ given $\{\conn^{\delta}=\alpha\}$ as in Proposition~\ref{prop::conditionallaw_ppf}. 
It turns out that the family of such conditional laws is precompact and independent of the boundary condition $\beta$ (Lemma~\ref{lem::conditionallaw}). 
Now, for any subsequential limit $\tilde{\eta}_i$, from the preceding two inputs we conclude that the law of $\tilde{\eta}_i$ is given by the Loewner chain associated to $p^{\alpha}_{\beta} \, \LF_{\beta}$ (Proposition~\ref{prop::conditionallaw_ppf}). 
As the law of $\tilde{\eta}_i$ is independent of the boundary condition, we then conclude that $p^{\alpha}_{\beta} \, \LF_{\beta} = \PartF_{\alpha}$. 

\smallbreak

Let us also remark that, while~\cite{KenyonWilsonBoundaryPartitionsTreesDimers}  
provides a systematic method to calculate the probability $\PP_{\beta}^{\delta}[\conn^{\delta}=\alpha]$ in the discrete model
(see the summary in Section~\ref{subsec::ust_crossingproba}), 
this does not yet give the result in the scaling limit.
Indeed, by taking the limit $\delta \to 0$, we merely obtain a formula for the left-hand side of~\eqref{eqn::ust_corssing_proba}. 
However, it is far from clear why the answer from~\cite{KenyonWilsonBoundaryPartitionsTreesDimers} is the same as the right-hand side of~\eqref{eqn::ust_corssing_proba}. 
Importantly, the latter is an explicit formula for the scaling limit of the crossing probability~\eqref{eqn::ust_corssing_proba}, which can furthermore be directly related to log-CFT, as we motivate next.

\subsection{Speculation: Logarithmic CFT for UST?}
\label{subsec: CFT_intro}

In the research of critical phenomena in planar models, Polyakov's conformal invariance conjecture~\cite{Polyakov:Conformal_symmetry_of_critical_fluctuations}, later formalized by Belavin, Polyakov \&~Zamolodchikov~\cite{BPZ:Infinite_conformal_symmetry_in_2D_QFT},
has proven in the last couple of decades to be a remarkable idea that has lead to many breakthroughs in contemporary mathematics
(see, e.g., \cite{LSW:Brownian_intersection_exponents1,
SchrammICM, SmirnovConformalInvariance}).
Even before its mathematical fruition, 
from the assumption of conformal invariance, many  properties of critical planar models 
such as critical exponents and the KPZ formula were correctly derived (see, e.g.,~\cite{Nienhuis:Coulomb_gas_formulation_of_2D_phase_transitions, 
Duplantier:Conformal_fractal_geometry_and_boundary_quantum_gravity, 
Cardy:Conformal_invariance_and_surface_critical_behavior}).
It has now become customary to speak of a \emph{conformal field theory}, or briefly, CFT, 
associated to each critical lattice model even in the mathematics literature.

The best understood such a CFT description is arguably that for 
the minimal model of the critical planar Ising model (having central charge $c=1/2$), 
whose bulk correlation functions can now be claimed to have been fully constructed~\cite{CHI:Ising_CFT}. 
For other models, even though numerous results towards their conformal invariance have now been established, 
their full description in terms of conformal field theories is at its infancy. 
In particular, unfortunately but interestingly, 
CFTs pertaining to a full description of (non-local) observables in
even the simplest lattice models such as percolation and 
self-avoiding polymers~\cite{Cardy:Logarithmic_correlations, MR:Percolation_LCFT, Read_Saleur:boundary_log_CFT} ($c=0$), 
or spanning trees, critical dense polymers, and the Abelian sandpile model~\cite{Pearce_Rasmussen:Solvable_critical_dense_polymers, Ruelle:Abelian_sandpile_CFT_survey} ($c=-2$), 
or the Ising model at the presence of boundary conditions, 
are still poorly understood even in the physics literature.
Namely, whenever one wishes to find a CFT description for the boundary critical phenomena in these models, such as
for the interfaces and boundary conditions, one immediately runs outside of the realm of minimal models.
So-called \emph{logarithmic
conformal field theory} (log-CFT) has been proposed as a framework for 
the complete description of the scaling limits of such models.
These conformal field theories are non-unitary, and in particular, they lack reflection positivity.
However, as we shall see in the present article, such theories can still have a probabilistic origin.

\smallbreak

Conformal invariance features in CFT via the effect of infinitesimal local conformal transformations, 
represented in terms of the \emph{Virasoro algebra}, 
on the fields in the theory. 
Thus, the representation theory of the Virasoro algebra, or extensions thereof, plays a fundamental role in understanding CFT mathematically~\cite{Schottenloher:Mathematical_introduction_to_CFT}.
(For instance, the minimal models such as in~\cite{CHI:Ising_CFT} comprise only finitely many irreducible Virasoro representations, which makes them amenable to a complete solution.)
The Virasoro algebra is the infinite-dimensional Lie algebra spanned by $\{ \mathrm{L}_n \colon n \in \Z \} \cup \{ \mathrm{C} \}$ such that 
\begin{align*} 
[\mathrm{L}_n,\mathrm{C}] = 0 \quad \quad \textnormal{and} \quad \quad 
[\mathrm{L}_n,\mathrm{L}_m] = (n-m) \mathrm{L}_{n+m} 
+ \frac{1}{12} n(n^2-1) \delta_{n,-m} \mathrm{C} , \quad \textnormal{for } n,m \in \Z.
\end{align*}
The generator $\mathrm{L}_0$ of scalings in the Virasoro algebra 
becomes \emph{non-diagonalizable} in log-CFTs,
that is, it has non-trivial Jordan blocks (generalized eigenspaces of dimension greater than one) in the representations relevant to the theory.
The logarithmic behavior in the correlation functions stems from this property\footnote{This results from the appearance of fields transforming in \emph{non-semisimple} representations of the Virasoro algebra, making the algebraic content of log-CFTs notoriously difficult.}.
For a survey on log-CFT, see~\cite{Creutzig_Ridout:log_CFT_survey}.

\smallbreak 

Let us now very briefly describe the representations of the Virasoro algebra relevant to CFT interpretations of the present work.
Each universal highest weight module $\mathrm{V}_{c,h}$ 
(Verma module) 
is generated by a highest weight vector $v_{c,h}$ that is an eigenvector 
of the Virasoro generator $\mathrm{L}_0$ with some eigenvalue $h \in \C$, called \emph{conformal weight}, 
and where the central element $\mathrm{C}$ 
acts as a constant $c \in \C$, called the \emph{central charge}.
Of interest to us are those Verma modules that contain so-called \emph{singular vectors}, which result in correlation functions satisfying BPZ PDEs, such as~\eqref{eqn::USTPDE} for $\kappa=8$. These Verma modules have been classified by Fe{\u\i}gin and Fuchs~\cite{Feigin-Fuchs:Verma_modules_over_Virasoro_book}: 
they belong to a special series indexed by two integers 
$r,s \geq 1$, and a parameter $t(\kappa) = \kappa/4 \in \C \setminus \{0\}$, such that 
$h = h_{r,s}(\kappa)$ and $c = c(\kappa)$ are given by 
\begin{align} \label{eq: Kac weights}
\begin{split}
h_{r,s}(\kappa) := \; & \frac{(r^2-1)}{4} t(\kappa) + \frac{(s^2-1)}{4} \frac{1}{t(\kappa)} + \frac{(1-rs)}{2} 
\\
c(\kappa) := \; & 13 - 6 \big( t(\kappa)+ \tfrac{1}{t(\kappa)} \big) .
\end{split}
\end{align}
In this case, the smallest such $\ell = rs$ is the lowest \emph{level} at which a singular vector occurs in $\mathrm{V}_{c,h}$.
The $\mathrm{L}_0$-eigenvalues $h_{r,s}(\kappa)$ are often termed \emph{Kac conformal weights}.

We have used the parameterization by $\kappa$ 
to make connection with $\SLE_\kappa$ theory (see~\cite{PeltolaCFTSLE} for references).
For example, we have
$h_{1,1}(\kappa) = 0$, 
$h_{1,2}(\kappa) = \frac{6-\kappa}{2\kappa}$, 
and $h_{1,3}(\kappa) = \frac{8-\kappa}{\kappa}$.
It is believed that the $\SLE_\kappa$ curve is in some sense generated by a CFT (primary) field $\Phi_{1,2}$ of weight $h_{1,2}(\kappa)$,
that generates a representation which is a quotient of the Verma module $\mathrm{V}_{c(\kappa),\, h_{1,2}(\kappa)}$ (by the universality property).
The field $\Phi_{1,2}$ is also known as a \emph{boundary condition changing operator}~\cite{Cardy:Conformal_invariance_and_surface_critical_behavior, Bauer-Bernard:Conformal_field_theories_of_SLEs, Cardy:SLE_and_Dyson_circular_ensembles}.
Note that when $\kappa=8$, we have $h_{1,1}(8) = 0 = h_{1,3}(8)$. This results in a logarithmic correction in the asymptotics of the correlation functions for the field $\Phi_{1,2}$, that we observe rigorously in Theorems~\ref{thm::coulombgasintegral} and~\ref{thm::ppf}.
One way to see this heuristically is via so-called \emph{fusion} of the field with itself, that we next describe\footnote{For the sake of keeping the exposition to the point, 
we only discuss fusion of the simplest primary fields $\Phi_{1,2}$, and refer to the vast CFT literature for more general features (keeping in mind that most of it is written in the physics level of rigor).}.

Generically, one expects a property of type 
$\Phi_{1,2} \boxtimes \Phi_{1,2} = \Phi_{1,1} \boxplus \Phi_{1,3}$ to hold, 
where ``$\boxtimes$'' denotes some kind of an operator product (OPE)
and ``$\boxplus$'' indicates the possible outcomes (but does not represent a direct sum). 
That is, the following heuristic operator product asymptotic expansion should hold:
\begin{align} \label{eqn::fusion_OPE}
\textnormal{``} \; 
\Phi_{1,2}(z_1) \; \Phi_{1,2}(z_2) 
\; \sim \; 
\frac{c_1}{(z_2-z_1)^{\Delta_{1,1}}} \, \Phi_{1,1} (z_2)
\; + \; 
\frac{c_2}{(z_2-z_1)^{\Delta_{1,3}}} \, \Phi_{1,3} (z_2) \; \textnormal{''} ,
\quad \textnormal{as } 
|z_1 - z_2| \to 0 ,
\end{align}
where $c_1, c_2 \in \C$ are \emph{structure constants}, and
the exponents are 
$\Delta_{1,1} = 2 h_{1,2}(\kappa) - h_{1,1}(\kappa) = \frac{6-\kappa}{\kappa}$ for the so-called identity channel,
and $\Delta_{1,3} = 2 h_{1,2}(\kappa) - h_{1,3}(\kappa) = - \frac{2}{\kappa}$ for the other channel.
Caution is in order here:~\eqref{eqn::fusion_OPE} is to be understood in terms of \emph{correlation functions} of the fields, as the latter are not defined pointwise.
In other words, specifying~\eqref{eqn::fusion_OPE} means specifying the Frobenius series of the correlation functions.
When $\kappa=8$, the exponents coincide: $\Delta_{1,1} = \Delta_{1,3} = -1/4$.
This results in a phenomenon similar to the solution of the hypergeometric equation when the roots of its indicial exponents coincide (or differ by an integer) --- one of the linearly independent solutions has a logarithm.
We invite the reader to compare this with the statements \textnormal{(ASY)} in Theorems~\ref{thm::coulombgasintegral} and~\ref{thm::ppf}.

To accommodate this phenomenon, one could write the right-hand side of~\eqref{eqn::fusion_OPE} in the less restrictive form 
$c_1(z_1,z_2) \,  \Phi_{1,1}(z_2) + c_2(z_1,z_2) \,  \Phi_{1,3}(z_2)$,
for some functions $c_1(z_1,z_2)$ and $c_2(z_1,z_2)$ allowing logarithmic terms in the expansion.
Our results show that, for any CFT (boundary) fields describing the scaling limit of the UST Peano curves,
that is, $\SLE_8$ curves, the formal OPE product~\eqref{eqn::fusion_OPE} has the explicit form 
\begin{align} \label{eq: our_fusion}
(z_2-z_1)^{-1/4} \, \big( \pi \, \Phi_{1,1}(z_2) - \log (z_2 - z_1) \,  \Phi_{1,3}(z_2)  \big) .
\end{align}

From the point of view of Virasoro representation theory, if a Verma module $\mathrm{V}_{c,h}$ contains a singular vector, we may take one at the lowest level and form the quotient module of 
$\mathrm{V}_{c,h}$ by this submodule. 
This quotient module $\mathrm{S}_{c,h}$ is unique and simple (irreducible).
Minimal models are CFTs whose fields are constrained to live in such representations, 
whence the quotienting results in strict truncation of the operator content of the theory, forcing it to be finite.
More precisely, one can also parametrize the special series of central charges and Kac weights by two coprime integers $p,p' \geq 1$ 
as in~\cite[Chapter~7, Eq.~(7.65)]{DMS:CFT}. 
When $c(\kappa) \leq 1$, which is of interest to us, we have $t(\kappa) = p / p'$.
A minimal model of type $M(p,p')$ comprises fields 
$\{ \Phi_{r,s} \colon r,s \in \bZpos , \;  1 \leq r \leq p'-1 , \; 1 \leq s \leq p-1 , \; p r > p' s \}$. 
In particular, with $\kappa=8$, 
we find that the minimal model $M(2,1)$ is empty. Such a model would have central charge $c=-2$, and it is perhaps the most studied example of a log-CFT~\cite{Gurarie:log_CFT, Kausch:Symplectic_fermions}.
In particular, extending it beyond the minimal model to include, for instance, fields $\Phi_{1,s}$ with $s \geq 1$,
we can consider a theory in particular containing $\Phi _{1,1}$, $\Phi _{1,2}$, and $\Phi _{1,3}$. 
Note that the conformal weights of the fields $\Phi_{1,s}$ with $\kappa = 8$ read
$h_{1,s}(8) \in \{0 , -\frac{1}{8} , 0 , \frac{3}{8}, 1 , \frac{15}{8} , 3 , \frac{35}{8} , 6 , \frac{63}{8}, 10, \ldots\}$.
In such a model, one can consider fusion of representations of the Virasoro algebra and ask whether all relevant Virasoro modules are contained in the theory.
From the literature of fusion products in CFT~\cite{Gurarie:log_CFT, GK:Indecomposable_fusion_products}, 
in the case of $c=-2$ the fusion of two simple modules 
$\mathrm{S}_{1,2}$ 
(corresponding to $\Phi_{1,2}$ with $\kappa=8$) is 
described by the exact sequence
\begin{align} \label{eqn::staggered}
0 \longrightarrow \mathrm{S}_{1,1} \overset{\iota}{\longrightarrow} 
\mathrm{S}_{1,2} \boxtimes \mathrm{S}_{1,2}
\overset{\pi}{\longrightarrow} \mathrm{S}_{1,3} \longrightarrow 0 .
\end{align}
The resulting object $\mathrm{M} := \mathrm{S}_{1,2} \boxtimes \mathrm{S}_{1,2}$ is a so-called
\emph{staggered module}~\cite{Rohsiepe:Reducible_but_indecomposable, KR:Staggered} of the Virasoro algebra, which in this case is non-trivial but relatively innocent.
Gurarie showed in~\cite{Gurarie:log_CFT} that this staggered module exists, 
and Gaberdiel \&~Kausch constructed it explicitly~\cite{GK:Indecomposable_fusion_products} as the fusion~\eqref{eqn::staggered}.
We will not get into details of this construction, but only note that 
such a module is unique~\cite[Example~2 and Corollary~3.5]{KR:Staggered}\footnote{See also~\cite[Section~3.6]{Kytola:SLE_local_martingales_in_logarithmic_representations} relating local martingales for certain $\SLE$ variants to log-CFT.}. 
From its structure~\eqref{eqn::staggered}, we see that the staggered module $\mathrm{M}$
contains the simple module $\mathrm{S}_{1,1}$
as a submodule, and it projects onto the module $\mathrm{S}_{1,3}$ such that
$\mathrm{M}  / \mathrm{S}_{1,1} \cong \mathrm{S}_{1,3}$.
Loosely speaking, the asymptotics in Theorems~\ref{thm::coulombgasintegral} and~\ref{thm::ppf}
should correspond to this fusion. Note that in our formulas, also the structure constants ``$-1$'' and ``$\pi$'',
that cannot be inferred from the representation theory, are explicit as in~\eqref{eq: our_fusion}.
Following~\cite{Gurarie:log_CFT}, the asymptotic property without a logarithm 
corresponds to the identity field $\Phi_{1,1}$ that generates the simple submodule $\mathrm{S}_{1,1}$,
and the asymptotic property with the logarithm to its ``logarithmic partner'' $\Phi_{1,3}$ that generates a projective Virasoro module $\mathrm{S}_{1,3}$, which also happens to be simple in this case.
However, the fusion module $\mathrm{M}$ is not a direct sum of these pieces, whereas it is \emph{indecomposable but not semisimple}.
Indeed, $\mathrm{L}_0$ is not diagonalizable, and
we have an off-diagonal action from $\mathrm{S}_{1,3}$ to $\mathrm{S}_{1,1}$,
so $\mathrm{S}_{1,3}$ cannot be a Virasoro submodule of $\mathrm{M}$: 
namely in $\mathrm{M}$, we have $\mathrm{L}_0 \, \Phi_{1,3} = \Phi_{1,1}$ and $\mathrm{L}_n \, \Phi_{1,3} = 0$ for all $n \geq 1$.
This is drastically different from the case of unitary CFTs.

\smallbreak

\begin{speculation}
Correlation functions of the UST wired/free boundary condition changing operators $\Phi_{1,2}$ 
\`a la Cardy~\textnormal{\cite{Cardy:Conformal_invariance_and_surface_critical_behavior, Pearce_Rasmussen:Solvable_critical_dense_polymers}} 
must satisfy asymptotic properties encoded in the explicit OPE~\eqref{eq: our_fusion} 
\textnormal{(}more precisely, Theorems~\ref{thm::coulombgasintegral} and~\ref{thm::ppf} \textnormal{(ASY)}\textnormal{)}. 
In particular, the structure constants for this fusion are explicit. 
Assuming that these operators generate simple 
Virasoro modules $\mathrm{S}_{1,2}$ of central charge $-2$ and conformal weight $-1/8$, 
the fusion rules~\eqref{eq: our_fusion} are associated with the unique Virasoro staggered module~\eqref{eqn::staggered}. 
\end{speculation}

\smallbreak

\begin{remark} \label{rem::boundary_arm}
Lastly, we remark that the CFT predictions also agree with the known boundary arm exponents for the $\SLE_8$. Indeed, from~\textnormal{\cite[Eq.~(1.2)]{WuZhanSLEBoundaryArmExponents}}, 
the odd $(N-1)$-arm exponent equals $N(N-2)/8$
\textnormal{(}that is, $\alpha_{2n-1}^+$ with $2n=N$\textnormal{)}, 
and the even $(N-1)$-arm exponent equals $(N-1)^2/8$
\textnormal{(}that is, $\alpha_{2n}^+$ with $N-1=2n$\textnormal{)}.
With the Kac conformal weights 
$h_{1,N+1}(\kappa) = N(2(N+2)-\kappa)/2\kappa$
for the ``$N$-leg'' boundary operator $\Phi_{1,N+1}$,
we find agreement with a generalization of the OPE
\textnormal{(}see~\textnormal{\cite[Eq.~(4.16)]{PeltolaCFTSLE})} when $\kappa = 8$\textnormal{:}
\begin{align*}
\frac{N(N-2)}{8} \; = \; \; & h_{1,N+1}(8) - Nh_{1,2}(8) + Nh_{1,2}(8) 
\; = \; h_{1, N+1}(8) , && \textnormal{ $N$ even}, \\
\frac{(N-1)^2}{8} \; = \;\; & 
h_{1,N+1}(8) - Nh_{1,2}(8) + (N-1)h_{1,2}(8) 
\; = \; h_{1, N+1}(8) - h_{1,2}(8), && \textnormal{ $N$ odd}. 
\end{align*}
\end{remark}

\subsection*{Organization}

The structure of the subsequent sections 
of this article is the following.

In Section~\ref{sec::coulombgasintegrals_new}, we introduce the partition functions and relate them to both Coulomb gas integrals stemming from CFT, and to period matrices of hyperelliptic Riemann surfaces. In particular, we show that they have a determinantal structure. 
We prove most of the key properties of the partition functions $\LF_\beta$ in  Section~\ref{sec::coulombgasintegrals_new}.
The next Section~\ref{sec::ust} concerns the discrete UST model. 
The goal is to derive explicitly the law of the Peano curve in the scaling limit. 
The key ingredient is the identification step, which we establish by 
constructing a suitable martingale observable and analyzing it in detail. 
The scaling limit of the observable is uniquely determined by its boundary data (cf.~Proposition~\ref{prop::holo_cvg}) ---  
it is a conformal map onto a certain slit rectangle. 
The last Section~\ref{sec::ppf} 
is devoted to identifying the crossing probabilities of the UST Peano curves in the scaling limit, and using these results to compute the asymptotics of the pure partition functions $\PartF_\alpha$.

We also include four appendices in this article.
Appendix~\ref{app::matrices} contains simple computations relating the real integrals discussed above to loop integrals appearing in Section~\ref{sec::coulombgasintegrals_new}.
Appendix~\ref{app::examples} gives examples of partition functions. 
Appendix~\ref{app::SC_mappings} discusses Schwarz-Christoffel mappings. 
The last Appendix~\ref{app: marginal law} contains a standard martingale argument to derive the conditional laws of the scaling limit curves for each connectivity. We use it in the identification of the crossing probabilities in Theorem~\ref{thm::ust_crossing_proba}.

\subsection*{Acknowledgements}
\begin{itemize}
\item 
This material is part of a project that has received funding from the  European Research Council (ERC) under the European Union's Horizon 2020 research and innovation programme (101042460): 
ERC Starting grant ``Interplay of structures in conformal and universal random geometry'' (ISCoURaGe) 
and from the Academy of Finland grant number 340461 ``Conformal invariance in planar random geometry.''
E.P.~is also supported by 
the Academy of Finland Centre of Excellence Programme grant number 346315 ``Finnish centre of excellence in Randomness and STructures (FiRST)'' 
and by the Deutsche Forschungsgemeinschaft (DFG, German Research Foundation) under Germany's Excellence Strategy EXC-2047/1-390685813, 
as well as the DFG collaborative research centre ``The mathematics of emerging effects'' CRC-1060/211504053.

\smallbreak

\item 
H.W. is funded by Beijing Natural Science Foundation (JQ20001). H.W. is partly affiliated at Yanqi Lake Beijing Institute of Mathematical Sciences and Applications, Beijing, China. 

\smallbreak

\item 
We thank Xiaokui Yang for helpful discussions on complex analysis and Kalle Kyt{\"o}l{\"a} for pointing out useful references in log-CFT.
We are very grateful to the anonymous referee for several helpful suggestions that simplified many of the arguments and greatly shaped this article.
\end{itemize}


\smallskip{}

\section{Determinantal structure of Coulomb gas integrals for $c = -2$}
\label{sec::coulombgasintegrals_new}

Many correlation functions in conformal field theory can be written in terms of so-called Coulomb gas integrals~\cite{DF-multipoint_correlation_functions, DubedatEulerIntegralsCommutingSLEs, KytolaPeltolaConformalCovBoundaryCorrelation}.
In the present work, we apply this (initially heuristic) formalism to the case of
correlation functions arising from the scaling limit of 
UST models. 

We will consider a specific basis of such functions, denoted $\LF_{\beta}$ for $\beta \in \LP_N$.
Importantly and specifically to the present case, these Coulomb gas integrals can be written in terms of determinants of matrices involving $a$-periods of a hyperelliptic Riemann surface,  
which morally correspond to choices of various possible screening variables for the Coulomb gas integrals. 
Another remarkable feature of the basis functions $\LF_{\beta}$ is total positivity:
they can be chosen to be all simultaneously positive.
These functions are closely related to the scaling limit of the holomorphic UST observable of Section~\ref{sec::ust}, see Lemmas~\ref{lem::conformal_general_expansion}~\&~\ref{lem::ust_general_goal_new}. 

\smallbreak

In the Coulomb gas formalism of conformal field theory (CFT), 
one constructs correlation functions of vertex operators from a free-field representation 
involving certain exponentials of the free boson (Gaussian free field, GFF)~\cite[Chapter~9]{DMS:CFT}.
The resulting correlation functions have an explicit integral form. 
The additional determinantal structure of these correlation functions in the present special case of central charge $c = -2$ and $\SLE_\kappa$ parameter $\kappa = 8$ could be seen as a fermionic feature of the theory describing UST observables~\cite{Kausch:Symplectic_fermions}.

Specifically, we consider functions $\LF \colon \chamber_{2N} \to \C$, 
defined on the configuration space
\begin{align} \label{eq: chamber}
\chamber_{2N} := \big\{ \boldsymbol{x} = (x_{1},\ldots,x_{2N}) \in \R^{2N} \colon x_{1} < \cdots < x_{2N} \big\} .
\end{align}
For fixed $\boldsymbol{x}$, the value of $\LF(\boldsymbol{x})$
is written as a Dotsenko-Fateev type integral~\cite{DF-multipoint_correlation_functions, DubedatEulerIntegralsCommutingSLEs},
\begin{align} \label{eq: ansatz8}
\LF(\boldsymbol{x}) :=  \int_{\Gamma} f(\boldsymbol{x};u_1,\ldots,u_\ell) \; \ud u_1 \cdots \ud u_\ell ,
\end{align}
where $\Gamma \subset \C^\ell$ is an integration surface for the integration variables $u_1, \ldots, u_\ell$ 
(screening variables),
belonging to some compact subset of $\C^\ell$, with $\ell \in \bZnn$, 
and the integrand $f$ is a branch of the multivalued function
\begin{align} 
\label{eq: integrand}
f (\boldsymbol{x};u_1,\ldots,u_\ell) 
:= \; & f^{(0)}(\boldsymbol{x})
\prod_{1\leq r<s\leq \ell}(u_{s}-u_{r}) 
\prod_{\substack{1\leq i\leq 2N \\ 1\leq r\leq \ell}}
(u_{r}-x_{i})^{-1/2} 
= f^{(0)}(\boldsymbol{x}) \; 
\frac{\Delta (\boldsymbol{u})}{\wfuncnobranch(\boldsymbol{u}; \boldsymbol{x})} , 
\end{align} 
where the prefactor $f^{(0)}$ is independent of the integration variables,
\begin{align} 
\label{eqn::def_fnod} 
f^{(0)}(\boldsymbol{x}) := \; &
\prod_{1\leq i<j\leq 2N}(x_{j}-x_{i})^{1/4} ,
\end{align}
$\Delta$ is the Vandermonde determinant only involving the integration variables $\boldsymbol{u} = (u_1, \ldots, u_\ell)$, 
\begin{align} \label{eqn::def_vander} 
\Delta (\boldsymbol{u}) 
:= \prod_{1\leq r<s\leq \ell}(u_{s}-u_{r}) ,
\end{align}
and $\wfuncnobranch$ is defined as the (multivalued) product
\begin{align} \label{eqn::def_hypercurve_prod}
\wfuncnobranch(\boldsymbol{u}; \boldsymbol{x}) := \; &
\prod_{r = 1}^{\ell} \wfuncnobranch(u_r; \boldsymbol{x}) ,
\qquad 
\wfuncnobranch(u; \boldsymbol{x}) := \prod_{j=1}^{2N}(u-x_{j})^{1/2} .
\end{align}

Note that $f^{(0)}$ is a Coulomb gas correlation function without any screening --- or an $\SLE_8(2,2,\ldots,2)$ partition function in the imaginary geometry framework~\cite{DubedatSLEFreefield, MillerSheffieldIG4, BerestyckiLaslierRayDimersIG}, where certain types of SLE variants are coupled with the GFF.

We usually keep the variables $\boldsymbol{x}$ in~\eqref{eq: chamber},
but $\LF$ also extends
to a multivalued function~on
\begin{align} \label{eq: complex chamber}
\mathfrak{Y}_{2N} := \big\{ \boldsymbol{x} = (x_{1},\ldots,x_{2N}) \in \C^{2N} \colon x_{i} \neq x_{j} \textnormal{ for all } 1 \leq i \neq j \leq 2N \big\} .
\end{align}
As a function of the integration variables 
\begin{align*}
\boldsymbol{u} = (u_1, \ldots, u_\ell) \in 
\Wchamber^{(\ell)} = \; \Wchamber_{x_{1},\ldots,x_{2N}}^{(\ell)}
:= \big( \C\setminus \{x_{1},\ldots,x_{2N}\} \big)^{\ell} ,
\end{align*}
$f(\boldsymbol{x};\cdot)$ has branch points at $u_r = x_j$ for all $r$ and $j$, and zeros at $u_{r} = u_{s}$ for all $r \neq s$.
To define a branch of $f(\boldsymbol{x};\cdot)$ on a simply connected subset of $\Wchamber^{(\ell)}$, we can just determine its value at some point in this set, 
and then define its value for all other points by analytic continuation.

To construct the desired partition functions, we take $\ell = N$, ensuring the scaling property
\begin{align*}
\LF(\lambda x_{1},\ldots, \lambda x_{2N}) = \lambda^{N/4} \LF(x_{1},\ldots,x_{2N}) ,
\end{align*}
which corresponds to the scale-covariance of a CFT primary field of weight
$h_{1,2} = \frac{6-\kappa}{2\kappa} = - \frac{1}{8}$ with $\kappa = 8$. 
We must also choose the integration contours $\Gamma$ judiciously, as discussed below in detail. 
In fact, the choice of the integration contours is the most intricate part of the Coulomb gas formalism, and it is far from clear how to implement this, e.g., in the setup of SLE/GFF couplings, or a timelike (imaginary) version of Liouville theory, e.g.,~\cite{SantachiaraViti,GKR:Imaginary_Liouville}.

\smallbreak

The purpose of this section is to introduce the relevant functions $\LF_{\beta}$ in Coulomb gas integral form and show that they also have a determinantal structure. 
We first discuss the underlying hyperelliptic Riemann surfaces in Sections~\ref{subsec:hyperelliptic}--\ref{subsec::period_matrix_homology}, and the choices of integration contours in Section~\ref{subsec::Coulomb_gas_Fbeta}, which contains the definition of the basis functions $\{ \LF_{\beta} \colon \beta \in \LP_N\}$ (Definition~\ref{def: Lbeta}).
Sections~\ref{subsec::rotation_symmetry}--\ref{subsec::positivityproof} address salient properties of these functions: 
rotation symmetry, M\"obius covariance, partial differential equations, asymptotic properties, and total positivity 
--- which together with the summary in Section~\ref{subsec::ppf_concluding} will prove Theorem~\ref{thm::coulombgasintegral}.

\subsection{Integration contours and one-forms on a hyperelliptic Riemann surface}
\label{subsec:hyperelliptic}
Throughout, we fix $N \geq 1$ and $\beta \in \LP_N$ with link endpoints ordered as in~\eqref{eq: link pattern ordering}.
We define the following disjoint curves on $\C$ indexed by $r \in \{1,2,\ldots, N\}$: 
\begin{itemize}[leftmargin=2em]
\item $\gamma_r^\beta$ is a simple curve 
started from $x_{a_r}$, ending at $x_{b_r}$, and such that 
$\gamma_r^\beta \cap \gamma_s^\beta = \emptyset$ for all $r \neq s$, 
and $\gamma_r^\beta \subset \overline{\HH}^* := \{ z \in \C \colon \Im(z) < 0 \}$; and

\smallbreak

\item
$\acycle^\beta_r$ is a clockwise oriented simple loop surrounding $\gamma_r^\beta$ and no other $\gamma_s^\beta$, 
and such that $\acycle_r^\beta \cap \acycle_s^\beta = \emptyset$ for all $r \neq s$.
\end{itemize}
Then, as illustrated in Figure~\ref{fig: homology cycles},
the (homology classes of) $N-1$ of these loops,
e.g., $\{\acycle^\beta_2, \acycle^\beta_3, \ldots, \acycle^\beta_{N}\}$, 
form a half of a canonical homology basis (namely the $a$-cycles) 
for the first integral
homology group $H_1(\Sigma,\Z)$ of the hyperelliptic Riemann surface $\Sigma = \Sigma_{x_1,\ldots,x_{2N}}$ of genus $g = N-1$ associated to the hyperelliptic curve
\begin{align} \label{eq: hyperelliptic_curve}
\Big\{ (u,\wfuncnobranch) \in \C^2 \colon \wfuncnobranch^2 = \prod_{j = 1}^{2N} (u - x_j) \Big\} .
\end{align} 
(We also include here the case $g=1$ of an elliptic curve and $g=0$ with trivial homology.)
See, e.g., the book~\cite[Chapter~III.7]{RiemannSurfacesFarkasKra} for standard facts concerning such Riemann surfaces. 
Recall that $\Sigma$ is a two-sheeted branched covering of the Riemann sphere ramified at the points $x_1,\ldots,x_{2N}$.
In our case, none of the ramification points is at infinity 
--- note that there are two points $\infty^+$ and $\infty^-$ 
that correspond to infinity, lying on the two sheets $\Sigma^+$ and $\Sigma^-$
of $\Sigma$ (that is, on the two copies of the Riemann sphere).
In the case of $g=0$, $\Sigma_{x_1,x_2}$ is the surface on which 
$u \mapsto \sqrt{(u - x_1) (u - x_2)}$ is single-valued, obtained by gluing two copies of the Riemann sphere together along the cut $[x_1, x_2]$.
In general, $\Sigma_{x_1,\ldots,x_{2N}}$ can be formed by choosing $N$ disjoint and non-intersecting cuts (e.g., the cuts $\gamma_r^\beta$ from $x_{a_r}$ to $x_{b_r}$ according to the link pattern $\beta$) and gluing two copies of the Riemann sphere together along these cuts (then clearly the genus of $\Sigma$ is $N-1$).
The function
\begin{align} \label{eqn::def_hypercurve}
u \quad \mapsto \quad \wfuncnobranch(u; \boldsymbol{x}) := 
\prod_{j=1}^{2N}(u-x_{j})^{1/2} 
\end{align}
admits a single-valued meromorphic branch on $\Sigma$ (see, e.g.,~\cite[Chapter~III.7.4]{RiemannSurfacesFarkasKra}), which determines the complex structure of $\Sigma$. We make the standard choice\footnote{Note that this branch choice is common to all $\beta \in \LP_N$, and all choices of the branch cuts $\gamma_1^\beta$, $\ldots$, $\gamma_N^\beta$ with various $\beta \in \LP_N$ correspond to the same Riemann surface.} where the branch is real and positive for $u > x_{2N}$ on $\Sigma^+$, which we denote as $\wfuncnobranch(\cdot \,; \boldsymbol{x}) = \wfunc(\cdot \,; \boldsymbol{x})$.

\begin{figure}[ht!]
\begin{subfigure}[b]{\textwidth}
\begin{center}
\includegraphics[width=0.7\textwidth]{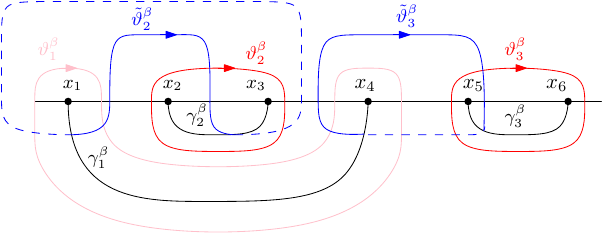}
\end{center}
\caption{Branch cuts $\gamma_1^\beta, \gamma_2^\beta, \gamma_3^\beta$ and 
loops $\acycle_1^\beta, \acycle_2^\beta, \acycle_3^\beta$
and $\bcycle_2^\beta, \bcycle_3^\beta$ 
associated to the link pattern $\beta = \{ \{1,4\}, \{2,3\}, \{5,6\} \}$.}
\end{subfigure}
\begin{subfigure}[b]{\textwidth}
\begin{center}
\includegraphics[width=0.7\textwidth]{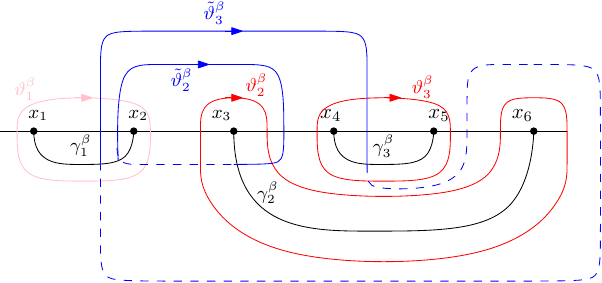}
\end{center}
\caption{Branch cuts $\gamma_1^\beta, \gamma_2^\beta, \gamma_3^\beta$ and 
loops $\acycle_1^\beta, \acycle_2^\beta, \acycle_3^\beta$
and $\bcycle_2^\beta, \bcycle_3^\beta$ 
associated to the link pattern $\beta = \{ \{1,2\}, \{3,6\}, \{4,5\} \}$.}
\end{subfigure}
\caption{\label{fig: homology cycles}
In this figure, the solid and dashed lines lie on different Riemann sheets $\Sigma^+$ and $\Sigma^-$.
The red loops $\smash{\acycle_2^\beta, \acycle_3^\beta, \ldots, \acycle_{N}^\beta}$ ($a$-cycles)
and the blue loops $\smash{\bcycle_2^\beta, \bcycle_3^\beta,\ldots,\bcycle_{N}^\beta}$ ($b$-cycles)
form a canonical homology basis for $H_1(\Sigma, \Z)$. 
The additional lighter red loop $\smash{\acycle_{1}^\beta}$
is a linear combination of the others in $H_1(\Sigma, \Z)$, see Equation~\eqref{eq::relation_of_acycles}.
}
\end{figure}

A useful basis for holomorphic differentials on $\Sigma$ is given by the $N-1$ one-forms 
\begin{align} \label{eqn::parametera_polynomial}
\omega_{r} 
:= \frac{u^{ r }\ud u}{\prod_{j=1}^{2N}(u-x_{j})^{1/2}} 
= \frac{u^{ r }\ud u}{\wfunc(u; \boldsymbol{x})} , 
\qquad \textnormal{for } r\in \{0,1,\ldots,N-2\} ,
\end{align}
see, e.g.,~\cite[Corollary~1 on page~98]{RiemannSurfacesFarkasKra}.
We also use the one-form 
\begin{align*}
\omega_{N-1} 
:= \frac{u^{N-1} \ud u}{\prod_{j=1}^{2N}(u-x_{j})^{1/2}} 
= \frac{u^{N-1} \ud u}{\wfunc(u; \boldsymbol{x})} , 
\end{align*}
which is not holomorphic but meromorphic: it has two simple poles with opposite residues 
$-1$ and $+1$ at the two points $\infty^+$ and $\infty^-$ 
that correspond to infinity.

Let us also recall that integrals 
$\ointclockwise_{\acycle^\beta_{r}} \theta$
of holomorphic or meromorphic one-forms $\theta$ on $\Sigma$ 
(i.e., Abelian differentials)
over the $a$-cycles are termed \emph{$a$-periods}. 
While the integral of a holomorphic one-form only depends on the homology class of the loop, the integral of a meromorphic one-form can also depend on the actual loop: 
if $\theta$ has non-zero residues at some points on $\Sigma$, 
then the value of $\ointclockwise_{\acycle} \theta$ depends on whether the loop $\acycle$ surrounds any of those points. 
For example, for the meromorphic one-form $\omega_{N-1}$, we have 
\begin{align} \label{eq: mero_one_form_residue}
- 2 \pi \ii 
\; = \; 2 \pi \ii \; \underset{\infty^+}{\mathrm{Res}} \, \omega_{N-1}
\; = \; \sum_{r=1}^N \ointclockwise_{\acycle^\beta_{r}} \omega_{N-1} ,
\end{align}
since $\acycle_1^\beta + \cdots + \acycle_{N}^\beta$ is a contractible loop on $\Sigma^+$ surrounding $\infty^+$ counterclockwise 
--- see, e.g.,~\cite[Chapter~III.3]{RiemannSurfacesFarkasKra} and Figure~\ref{fig: homology cycles}.

\subsection{Matrices involving a-periods}
\label{subsec::period_matrix_homology}

In our application to the UST model, we will use the following matrices of $a$-periods of the holomorphic one-forms~\eqref{eqn::parametera_polynomial} on $\Sigma$, and their line integral counterparts: 
\begin{align} 
\label{eq::P_matrices}
& P_{\beta}^{\circ} := \Big( \ointclockwise_{\acycle^\beta_{r}} \omega_{s-1} \Big)_{r,s=1}^{N} 
\;\;\;\, \qquad \textnormal{and} \qquad
P_{\beta} := \Big( \landupint_{x_{a_r}}^{x_{b_r}} \omega_{s-1}  \Big)_{r,s=1}^{N} , \\
\label{eq::A_matrices}
& A_\beta^\circ := \Big( \ointclockwise_{\acycle^\beta_{r+1}} \omega_{s-1} \Big)_{r,s=1}^{N-1} 
\qquad \textnormal{and} \qquad
A_\beta := \Big( \landupint_{x_{a_{r+1}}}^{x_{b_{r+1}}} \omega_{s-1} \Big)_{r,s=1}^{N-1} ,
\end{align}
where the integration symbols ``$\landupint$'' indicate that the integration\footnote{Note that these integrals are convergent, since the blow-ups 
at the endpoints of the contours are mild enough.}
is performed in the upper half-plane,
and we use the convention that $\det A_{{\vcenter{\hbox{\includegraphics[scale=0.2]{figures/link-0.pdf}}}}}^{\circ} = 1$ and $\det A_{{\vcenter{\hbox{\includegraphics[scale=0.2]{figures/link-0.pdf}}}}} = \tfrac{1}{2}$ when $N=1$. 
We will also repeatedly use the relation  
\begin{align} \label{eq::relation_of_acycles}
\acycle_{1}^\beta = -(\acycle_2^\beta + \acycle_3^\beta + \cdots + \acycle_{N}^\beta) \quad 
\textnormal{in } H_1(\Sigma, \Z) .
\end{align}

Let us record as a lemma the well-known but crucial property that all of these matrices are invertible. 
We also include a short argument for the readers' convenience, as well as an explicit relation between the matrices $A_\beta^\circ$, $P_{\beta}^{\circ}$ 
involving the $a$-periods with the matrices 
$A_\beta$, $P_{\beta}$ 
involving integrals over the real line.
(The latter are much more useful in computations.)

\begin{lemma} \label{lem: matrices_invertible}
The four matrices defined in \textnormal{(\ref{eq::A_matrices},~\ref{eq::P_matrices})} are invertible.
\end{lemma}

\begin{proof}
Note that $A_\beta^\circ$ (resp.~$A_\beta$) is the principal submatrix of $P_{\beta}^{\circ}$ (resp.~$P_\beta$) 
obtained by removing the first row and the last column. 
Furthermore, after adding all of the other rows of $P_{\beta}^{\circ}$ to its first row, recalling the relation~\eqref{eq::relation_of_acycles} 
and the residue~\eqref{eq: mero_one_form_residue}, 
we see that $\det P_{\beta}^{\circ}$ equals the determinant of
the matrix whose first row comprises $N-1$ zeros and $-2 \pi \ii$, while its other rows coincide with those of $P_{\beta}^{\circ}$. 
Hence, 
\begin{align} \label{eq::relate_A_and_P}
\det P_{\beta}^{\circ} = 2 \pi \ii \, (-1)^N \det A_{\beta}^{\circ} .
\end{align}

Next, relating the integrals over the $a$-cycles to the integrals over the corresponding real intervals 
(cf.~Lemma~\ref{lem: relation_of_Pcirc_and_P_matrix} in Appendix~\ref{app::matrices}) 
we see that $O_\beta \, P_{\beta} = P_{\beta}^{\circ}$, 
where $O_\beta$ is an explicit upper-triangular matrix defined in~\eqref{eq::Mbeta_explicit}.
In particular, the upper-triangular structure 
also implies that
$(\hat{O}_\beta)_{1,1} \, A_\beta = A_\beta^\circ$, 
where $(\hat{O}_\beta)_{1,1}$ is the (invertible) principal submatrix of $O_\beta$ obtained by removing the first row and the first column. 
Concretely, we have
\begin{align} \label{eq::relate_det_and_circle}
\det P_{\beta} = 2^{-N} \det P_{\beta}^{\circ} 
\qquad \textnormal{and} \qquad
\det A_{\beta} = 2^{1-N} \det A_{\beta}^{\circ} . 
\end{align}
It now suffices to show that the matrix $A_\beta^\circ$ is invertible. 
To this end, we consider the equation $A_{\beta}^{\circ} \boldsymbol{\upsilon}^t = \boldsymbol{0}$ 
for $\boldsymbol{\upsilon} = (\upsilon_1, \ldots, \upsilon_{N-1}) \in \C^{N-1}$ in terms of holomorphic one-forms  
\begin{align*}
\theta := \; & \sum_{s=1}^{N-1} \upsilon_s \, \omega_{s-1} , \qquad \upsilon_1, \ldots, \upsilon_{N-1} \in \C ,
\end{align*}
such that 
\begin{align} \label{eq: theta_integral_zero}
\ointclockwise_{\acycle^\beta_{r+1}} \theta = \; & 0 , \qquad \textnormal{for all } r \in \{1,2,\ldots,N-1\} .
\end{align}
Since $\{\omega_0, \ldots, \omega_{N-2}\}$ is a basis for the space of holomorphic differentials on $\Sigma$, and $\{\acycle^\beta_2, \ldots, \acycle^\beta_{N}\}$ are the $a$-cycles in a canonical homology basis for $H_1(\Sigma,\Z)$,
Equation~\eqref{eq: theta_integral_zero} implies that $\upsilon_1 = \cdots = \upsilon_{N-1} = 0$ 
(see, e.g.,~\cite[Proposition~III.3.3 in Chapter~III.3]{RiemannSurfacesFarkasKra}). 
\end{proof}

\subsection{Coulomb gas basis functions}
\label{subsec::Coulomb_gas_Fbeta}

To discuss concrete functions of complex variables, we identify $\C \cup \{\infty\}$ with the Riemann sphere, regarding it as $\Sigma^+$, one of the sheets of $\Sigma = \Sigma_{x_1,\ldots,x_{2N}}$ --- so that the solid paths in Figure~\ref{fig: homology cycles} lie on $\C$ and $\infty$ corresponds to $\infty^+ \in \Sigma$.
We then choose a branch 
(depending on $\beta$) 
of the multivalued integrand in~\eqref{eq: ansatz8} (with $\ell = N$) 
which is real and positive on 
\begin{align} \label{eq:: branch choice set}
\mathcal{R}_\beta :=  \big\{\boldsymbol{u} \in \Wchamber^{(N)}  \colon x_{a_r} < u_r < x_{a_r+1} , \; \textnormal{ for all } 1 \leq r \leq N \big\} .
\end{align}
For definiteness, we denote this branch choice as $f_\beta(\boldsymbol{x};\cdot) \colon \mathcal{R}_\beta \to \R_{>0}$. 
We extend the definition of $f_\beta(\boldsymbol{x};\cdot)$ to $\Wchamber^{(N)}$
via analytic continuation, so it becomes a multivalued function 
(single-valued but complex valued on the complement of the branch cuts)
\begin{align*}
f_\beta(\boldsymbol{x};\cdot) \;\colon\; \Wchamber^{(N)} \to \C .
\end{align*}

\begin{definition} \label{def: Lbeta}
With these choices, we define the Coulomb gas integral functions $\LF_\beta, \LF_\beta^\circ \colon \chamber_{2N} \to \C$ as\footnote{Note that the integrals in $\LF_\beta$ are convergent, since the blow-ups 
as $u_r$ tends to $\{ x_{a_r}, x_{b_r}\}$ are mild enough.}
\begin{align} 
\label{eq: Def of F beta}
\LF_\beta (\boldsymbol{x}) 
:= \; & \landupint_{x_{a_1}}^{x_{b_1}} \ud u_1 
\landupint_{x_{a_2}}^{x_{b_2}} \ud u_2 
\cdots \landupint_{x_{a_N}}^{x_{b_N}} \ud u_N 
\; f_\beta(\boldsymbol{x};\boldsymbol{u}) , \\
\label{eq: Def of F beta loops}
\LF_\beta^\circ (\boldsymbol{x}) 
:= \; & \ointclockwise_{\acycle^\beta_1}  \ud u_1 \ointclockwise_{\acycle^\beta_2}  \ud u_2 \cdots \ointclockwise_{\acycle^\beta_N} \ud u_N 
\; f_\beta(\boldsymbol{x};\boldsymbol{u}) .
\end{align}
\end{definition}
We also extend $\LF_\beta$ and $\LF_\beta^\circ$ to multivalued functions on $\mathfrak{Y}_{2N}$~\eqref{eq: complex chamber}.

\begin{remark}
The integrals in $\LF_\beta$ are evaluated, avoiding the branch cuts, by moving the points from above the real line: for instance, for 
$\beta  = \vcenter{\hbox{\includegraphics[scale=0.3]{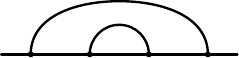}}} = \{\{1, 4\}, \{2, 3\} \}$, we have
\begin{align*}
\landupint_{x_1}^{x_4} \ud u_1 \; f_{\vcenter{\hbox{\includegraphics[scale=0.2]{figures/link-2.pdf}}}} (\boldsymbol{x};\boldsymbol{u})   
= \; & 
\landupint_{x_1}^{x_2} \ud u_1 \; | f_{\vcenter{\hbox{\includegraphics[scale=0.2]{figures/link-2.pdf}}}} (\boldsymbol{x};\boldsymbol{u}) | 
+ \ii \, \landupint_{x_2}^{x_3} \ud u_1 \; | f_{\vcenter{\hbox{\includegraphics[scale=0.2]{figures/link-2.pdf}}}} (\boldsymbol{x};\boldsymbol{u}) | 
- \landupint_{x_3}^{x_4} \ud u_1 
\; | f_{\vcenter{\hbox{\includegraphics[scale=0.2]{figures/link-2.pdf}}}}  (\boldsymbol{x};\boldsymbol{u}) | .
\end{align*}
\end{remark}

Importantly and specifically to the present case where the $\SLE_\kappa$ parameter is $\kappa=8$ and the central charge is $c=-2$, 
these Coulomb gas integrals can be written in terms of determinants of the matrices involving $a$-periods introduced in Section~\ref{subsec:hyperelliptic}. 
To this end, we observe that the latter can be evaluated in terms of the Vandermonde determinant: 
\begin{align} \label{eq: VandermondeA}
\begin{split}
\det A_{\beta}^{\circ} (\boldsymbol{x}) 
= \; & \ointclockwise_{\acycle^\beta_2} \ud u_2
\cdots \ointclockwise_{\acycle^\beta_N} \ud u_N 
\; \frac{\Delta (\boldsymbol{\dot{u}}_1)}{\wfunc(\boldsymbol{\dot{u}}_1; \boldsymbol{x})} \\
= \; & (-1)^{r+1} \, \ointclockwise_{\acycle^\beta_1} \ud u_1 
\cdots \ointclockwise_{\acycle^\beta_{r-1}} \ud u_{r-1} 
\ointclockwise_{\acycle^\beta_{r+1}} \ud u_{r+1} 
\cdots \ointclockwise_{\acycle^\beta_{N}} \ud u_{N} 
\; \frac{\Delta (\boldsymbol{\dot{u}}_r)}{\wfunc(\boldsymbol{\dot{u}}_r; \boldsymbol{x})} ,
\end{split}
\end{align}
for any $r \in \{1,2,\ldots,N\}$, 
where we write $\boldsymbol{\dot{u}}_r := (u_1, \ldots, u_{r-1},u_{r+1},\ldots,u_{N})$ 
and $\Delta$ is defined in~\eqref{eqn::def_vander},  
and $\wfunc$ in~\eqref{eqn::def_hypercurve_prod}.
Note that $\det A_{\beta}^{\circ} (\boldsymbol{x})$ in~\eqref{eq: VandermondeA} is well-defined only 
up to a branch choice for the function  
$\smash{u \mapsto \wfunc(u; \boldsymbol{x}) = \prod_{j}(u-x_{j})^{1/2}}$ on $\Sigma$. 
With our standard choice\footnote{Here, we identify the Riemann sphere $\C \cup \{\infty\} = \Sigma^+ \ni u$ as the sheet containing the $a$-cycles.}  
$\wfunc(u \,; \boldsymbol{x})$ which is real and positive when $u > x_{2N}$, we can state the explicit relations between $\LF_\beta$, $\LF_\beta^\circ$, and $\det A_{\beta}^{\circ}$:

\begin{proposition} \label{prop::two_Fs_as_determinants}
We have 
\begin{align} \label{eq: relate_F_and_Fcirc}
\LF_\beta (\boldsymbol{x}) = 2^{-N} \LF_\beta^\circ (\boldsymbol{x}) 
\; \neq 0  , \qquad \textnormal{ for all } \boldsymbol{x} \in \chamber_{2N} , 
\end{align}
and using the $a$-period matrix $A_{\beta}^{\circ} = A_{\beta}^{\circ}(\boldsymbol{x})$ from~\eqref{eq::A_matrices}, we have 
\begin{align} 
\label{eqn::PartFcirc_detAcirc}
\frac{\LF_\beta^\circ(\boldsymbol{x})}{2 \pi \ii \, f^{(0)} (\boldsymbol{x})} 
= \; & \zeta^\circ_{\beta} \, \det A_{\beta}^{\circ} (\boldsymbol{x}) , \qquad \boldsymbol{x} \in \chamber_{2N} , 
\end{align}
where $f^{(0)}$ is defined in~\eqref{eqn::def_fnod} and  
$\zeta^\circ_{\beta} := \ii^{\sum_{s=1}^N (2N-a_s+2)}$ is a phase factor. 
\end{proposition}

While making the expression for $\LF_{\beta}^{\circ}$ less symmetric, this form~(\ref{eq: VandermondeA},~\ref{eqn::PartFcirc_detAcirc}) 
is quite useful for simplifying formulas.
We include the $\beta$-dependent phase factors here due to the choice of branch~\eqref{eq:: branch choice set} of the integrand $f_\beta(\boldsymbol{x};\cdot)$,  which is manifest for the total positivity of the collection of functions $\{\LF_{\beta} \colon \beta \in \LP_N\}$ (see Proposition~\ref{prop: positivity}).

\begin{proof}
Using the Vandermonde determinant and the relation~\eqref{eq::relate_A_and_P}, we see that, up to some multiplicative phase factor,  
$\LF_\beta^\circ(\boldsymbol{x})$
equals 
$f^{(0)} (\boldsymbol{x}) \, \det P_{\beta}^{\circ}(\boldsymbol{x})
= 2 \pi \ii \, (-1)^N \, f^{(0)} (\boldsymbol{x}) \, \det A_{\beta}^{\circ}(\boldsymbol{x})$. 
We thus obtain the asserted identity~\eqref{eqn::PartFcirc_detAcirc} 
from the Vandermonde determinant~\eqref{eq: VandermondeA} and our branch choices.
Lemma~\ref{lem: matrices_invertible} implies that the function is non-zero, 
and the asserted identity~\eqref{eq: relate_F_and_Fcirc} follows using 
Lemma~\ref{lem: relation_of_Pcirc_and_P_matrix} from Appendix~\ref{app::matrices} with $2^N = \det O_\beta$. 
\end{proof}

\subsection{Rotation symmetry}
\label{subsec::rotation_symmetry}

We next record a simple 
property of $\LF_\beta$ when its variables are cyclically permuted within $\mathfrak{Y}_{2N}$.
To state it, we denote by 
$\sigma = \bigl(\begin{smallmatrix}
  1 & 2 & 3 & \cdots & 2N-1 & 2N \\
  2 & 3 & 4 & \cdots &  2N  & 1
\end{smallmatrix}\bigr)$
the cyclic counterclockwise permutation of the indices, and 
for each $\beta \in \LP_N$, we denote by $\sigma(\beta) \in \LP_N$ the link pattern obtained from $\beta$ via permuting the indices by $\sigma$ and then ordering the link endpoints appropriately. 
E.g., with $N=2$ we have
\begin{align*}
\beta = \; &
\vcenter{\hbox{\includegraphics[scale=0.3]{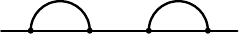}}} = \{\{1, 2\}, \{3, 4\} \} 
\qquad \longmapsto \qquad
\sigma(\beta) = \vcenter{\hbox{\includegraphics[scale=0.3]{figures/link-2.pdf}}} = \{\{1, 4\}, \{2, 3\} \}
\\
\beta = \; &
\vcenter{\hbox{\includegraphics[scale=0.3]{figures/link-2.pdf}}} = \{\{1, 4\}, \{2, 3\} \}
\qquad \longmapsto \qquad
\sigma(\beta) = \vcenter{\hbox{\includegraphics[scale=0.3]{figures/link-1.pdf}}} = \{\{1, 2\}, \{3, 4\} \} .
\end{align*}

\begin{lemma} \label{lem: rotation symmetry}
Let $\boldsymbol{\rho} \colon [0,1] \to \mathfrak{Y}_{2N}$ 
be a path from 
$\boldsymbol{\rho}(0) = \boldsymbol{x} = (x_{1},x_{2},\ldots,x_{2N-1},x_{2N}) \in \chamber_{2N}$
to $\boldsymbol{\rho}(1) := (x_{2},x_{3},\ldots,x_{2N-1},x_{2N},x_{1}) \in \mathfrak{Y}_{2N}$ 
such that $\boldsymbol{\rho} = (\rho_1, \rho_2, \ldots, \rho_{2N})$ satisfy
\begin{align*}
\begin{cases}
\Im(\rho_j(t)) = 0 , & \textnormal{for all } j \in \{1,2,\ldots,2N-1\} , \\
\Im(\rho_j(t)) \geq 0 , & \textnormal{for } j = 2N ,
\end{cases}
\qquad \textnormal{for all } t \in [0,1] .
\end{align*}
Then, we have
\begin{align} \label{eq: rotation symmetry}
\LF_{\sigma(\beta)} (\boldsymbol{\rho}(1))
= e^{\frac{\pi \ii}{4}} \LF_\beta (\boldsymbol{\rho}(0)) .
\end{align}
\end{lemma}

\begin{proof}
Note that the path $\boldsymbol{\rho}$ transforms the integration contours associated to $\beta$ in~\eqref{eq: Def of F beta} into the integration contours associated to $\sigma(\beta)$ in~\eqref{eq: Def of F beta}.
Hence, it is clear from the definition~\eqref{eq: Def of F beta} that~\eqref{eq: rotation symmetry} holds up to some multiplicative phase factor. To find it, we simplify $\LF_\beta = 2^{-N} \LF_\beta^\circ$ using Proposition~\ref{prop::two_Fs_as_determinants}: 
choosing $r$ in~\eqref{eq: VandermondeA} such that $b_r=2N$ (in this case, $a_r = 2r-1$), 
we see that $\boldsymbol{\rho}$ gives rise to the following phase factors:
\begin{itemize}[leftmargin=2em]
\item the factor $f^{(0)}(\boldsymbol{x})$ in~\eqref{eqn::PartFcirc_detAcirc} gets a phase 
$\exp( \frac{\pi \ii }{4} (2N-1))$; and

\smallbreak

\item the integral in~\eqref{eqn::PartFcirc_detAcirc} gets a phase $\exp( - \frac{\pi \ii }{2} (N-1))$. 
\end{itemize}
The asserted equality~\eqref{eq: rotation symmetry} follows by collecting the overall phase factor. 
\end{proof}

\subsection{M\"obius covariance}
\label{subsec::covariance}

We next verify the 
covariance property~\eqref{eqn::USTCOV} for Theorem~\ref{thm::coulombgasintegral}.

\begin{proposition} \label{prop: full Mobius covariance F}
For each $\beta \in \LP_N$, the function $\LF_{\beta}$ satisfies  
the M\"obius covariance~\eqref{eqn::USTCOV} 
for all M\"obius maps $\varphi \colon \HH \to \HH$ of the upper half-plane\textnormal{:}
writing $\varphi(\boldsymbol{x}) := (\varphi(x_{j+1}) \ldots , \varphi(x_{2N}) , \varphi(x_{1}) , \ldots , \varphi(x_{j}))$, 
\begin{align*}
\LF_{\varphi(\beta)}(\varphi(\boldsymbol{x}))
= \prod_{i=1}^{2N} \varphi'(x_{i})^{1/8} 
\times \LF_{\beta}(\boldsymbol{x}) ,
\end{align*}
where 
$\varphi(x_{j+1}) < \varphi(x_{j+2}) < \cdots < \varphi(x_{2N}) < \varphi(x_{1}) < \varphi(x_{2}) < \cdots  < \varphi(x_{j})$ with $j \in \{0,1,2,\ldots,2N-1\}$, and where $\varphi(\beta) \in \LP_N$ 
is the link pattern obtained from $\beta$ via permuting the indices according to the permutation of the boundary points induced by $\varphi$ and then ordering the link endpoints appropriately.  
\end{proposition}

\begin{proof}
Let $\varphi \colon \HH \to \HH$ be a M\"obius map. 
It is straightforward to check that $\frac{\varphi(z)-\varphi(w)}{z-w} = \sqrt{\varphi'(z)} \sqrt{\varphi'(w)}$ for any $z,w \in \overline{\HH}$ 
(see, e.g.,~\cite[Lemma~4.7]{KytolaPeltolaPurePartitionFunctions}). 
It immediately follows that $f^{(0)}$ defined in~\eqref{eqn::def_fnod} satisfies the M\"obius covariance
\begin{align*} 
f^{(0)}(\varphi(\boldsymbol{x}))
= \; & \prod_{i=1}^{2N} \varphi'(x_{i})^{(2N-1)/8} 
\times f^{(0)}(\boldsymbol{x}) .
\end{align*}
Moreover, if
$\varphi(x_{1}) < \varphi(x_{2}) < \cdots < \varphi(x_{2N})$, or 
$\varphi(x_{2}) < \varphi(x_{3}) < \cdots < \varphi(x_{2N}) < \varphi(x_{1})$, 
the homotopy 
type of the integration surface 
$\acycle_2^\beta \times \cdots \times \acycle_N^\beta$ 
does not change while moving the points $x_{1},\ldots,x_{2N}$ to their images 
$\varphi(x_{1}),\dots,\varphi(x_{2N})$.
Hence, after making the change of variables $u_s := \varphi(v_s)$ 
in~\eqref{eq: VandermondeA} with $r=1$, recalling the branch choice $\wfuncnobranch(\cdot \,; \boldsymbol{x}) = \wfunc(\cdot \,; \boldsymbol{x})$ of
\eqref{eqn::def_hypercurve},  
using the identity $\frac{\varphi(z)-\varphi(w)}{z-w} = \sqrt{\varphi'(z)} \sqrt{\varphi'(w)}$ we obtain 
\begin{align} \label{eqn::DETCOV}
\det A_{\varphi(\beta)}^{\circ} (\varphi(\boldsymbol{x})) 
= \; & \frac{\zeta^\circ_{\beta}}{\zeta^\circ_{\varphi(\beta)}} \times  \prod_{i=1}^{2N} \varphi'(x_i)^{-(N-1)/4} 
\times \det A_{\beta}^{\circ} (\boldsymbol{x}) .
\end{align}
If 
$\varphi(x_{j+1}) < \varphi(x_{j+2}) < \cdots < \varphi(x_{2N}) < \varphi(x_{1}) < \varphi(x_{2}) < \cdots  < \varphi(x_{j})$ for some $j \in \{2,3,\ldots,2N-1\}$, we can still iterate this argument using~\eqref{eq: VandermondeA} to conclude that~\eqref{eqn::DETCOV} holds for all M\"obius maps $\varphi$ of the upper half-plane. 
The asserted M\"obius covariance~\eqref{eqn::USTCOV} now follows by Proposition~\ref{prop::two_Fs_as_determinants}. 
\end{proof}

In essence, a M\"obius transformation amounts to changing the hyperelliptic curve~\eqref{eq: hyperelliptic_curve} within its (birational) equivalence class. 
The proof of Proposition~\ref{prop: full Mobius covariance F} shows that the determinant of the $a$-periods is covariant in the sense that (writing $\varphi(\boldsymbol{x}) = (\varphi(x_{1}),\ldots,\varphi(x_{2N}))$ with 
$\varphi(x_{1}) < \varphi(x_{2}) < \cdots < \varphi(x_{2N})$)
\begin{align*} 
\det A_{\beta}^{\circ} (\boldsymbol{x}) 
= \; & \prod_{i=1}^{2N} \varphi'(x_i)^{g/4} \times \det A_{\beta}^{\circ} (\varphi(\boldsymbol{x})) ,
\end{align*}
where $g=N-1$ is the genus of the hyperelliptic curve~\eqref{eq: hyperelliptic_curve}. 
It would be interesting to see whether such a covariance property also carries a more intrinsic geometric meaning.

\subsection{Partial differential equations}
\label{subsec::PDEs}

In this section, we give a short analytic proof for the PDE system~\eqref{eqn::USTPDE} for Theorem~\ref{thm::coulombgasintegral}. 
We give a more probabilistic 
proof in Corollary~\ref{cor::coulombgasintegralPDE}.
We denote the differential operators in~\eqref{eqn::USTPDE} by 
\begin{align*}
\mathcal{D}^{(j)} := 4\pdder{x_{j}} + \sum_{i \neq j} \Big( \frac{2}{x_{i}-x_{j}} \pder{x_{i}} 
+ \frac{1/4}{(x_{i}-x_{j})^{2}} \Big) .
\end{align*}

\begin{lemma} \label{lem: the BSA operator gives total derivative}
The integrand function $f$ defined in~\eqref{eq: integrand} satisfies the PDEs
\begin{align*}
\big(\mathcal{D}^{(j)} f \big)(\boldsymbol{x};\boldsymbol{u})
=\; & \sum_{r=1}^{N} \pder{u_{r}} \big( R(u_{r};\boldsymbol{x};\boldsymbol{\dot{u}}_r) \, f(\boldsymbol{x};\boldsymbol{u}) \big) , 
\qquad \textnormal{for all } j \in \{1,\ldots,2N\} ,
\end{align*}
where $\boldsymbol{\dot{u}}_r = (u_{1},\ldots,u_{r-1},u_{r+1},\ldots,u_{N})$
and $R$ is a rational function which is symmetric in its last $N-1$
variables, and whose only poles are where some of its arguments coincide.
\end{lemma}

\begin{proof}
This was proven in~\cite[Corollary~4.11]{KytolaPeltolaConformalCovBoundaryCorrelation}\footnote{In the notation of~\cite{KytolaPeltolaConformalCovBoundaryCorrelation}, 
$n=2N$, and $d_i = 2$ for all~$i$.} 
under the assumption $\kappa \notin \QQ$, but the same proof works for all $\kappa > 0$; in particular for $\kappa = 8$.
The proof is a relatively straightforward argument using
explicit analysis of $\mathcal{D}^{(j)}$ and $f$ together with the fact that the function $f^{(0)}$ defined in~\eqref{eqn::def_fnod} 
is a solution to the PDE system
$\mathcal{D}^{(j)} f^{(0)} = 0 $ for $j \in \{1,\ldots,2N\}$,
which can be shown by an explicit calculation. 
\end{proof}

To conclude from Lemma~\ref{lem: the BSA operator gives total derivative} 
that $\LF_\beta$ satisfy the PDEs~\eqref{eqn::USTPDE}, 
we use integration by parts.

\begin{proposition} \label{prop: PDEs F}
For each $\beta \in \LP_N$, the function $\LF_{\beta}$ satisfies 
the PDE system~\eqref{eqn::USTPDE}. 
\end{proposition}

\begin{proof}
Fix $j \in \{1,\ldots,2N\}$. 
By dominated convergence, we can take the differential operator $\mathcal{D}^{(j)}$ inside the integral in 
$\LF_{\beta}^\circ$, and thus let it act directly to the integrand $f_\beta$. 
Lemma~\ref{lem: the BSA operator gives total derivative}
then gives
\begin{align} \label{eq: integral sum}
\big( \mathcal{D}^{(j)} \LF_{\beta}^\circ \big)(\boldsymbol{x})
 = 
\sum_{r=1}^{N} \ointclockwise_{\Gamma_\beta} 
\pder{u_{r}} \big( R(u_{r};\boldsymbol{x};\boldsymbol{\dot{u}}_r) \; f_\beta (\boldsymbol{x};\boldsymbol{u}) \big) 
 \; \ud u_1 \cdots \ud u_N ,
\end{align}
where $\Gamma_\beta := \acycle_1^\beta \times \cdots \times \acycle_N^\beta$. 
Now, for each fixed $r$, we perform integration by parts in the $r$:th 
term of~\eqref{eq: integral sum}.
As the other integration variables are bounded away from $u_{r}$,
and the values of the integrand at the beginning and end points of the loop $\acycle_r^\beta \ni u_{r}$ coincide, the boundary terms cancel out.
We conclude that each term in~\eqref{eq: integral sum} actually equals zero, 
which gives the asserted PDE: $\mathcal{D}^{(j)} \LF_{\beta} = 2^{-N} \mathcal{D}^{(j)} \LF_{\beta}^\circ = 0$. 
\end{proof}

\subsection{Asymptotics}
\label{subsec: asy}
{

The goal of this section is to prove that $\LF_{\beta}$ 
satisfy the recursive asymptotics~(\ref{eqn::USTASY1},~\ref{eqn::USTASY2})  motivated by CFT fusion rules --- the result is stated in Proposition~\ref{prop: Asymptotics F}, which proves part of Theorem~\ref{thm::coulombgasintegral}.

\begin{proposition} \label{prop: Asymptotics F}
The collection $\{\LF_{\beta} \colon \beta \in \LP_N , \; N \in \bZnn \}$ of functions  
satisfies $\LF_{\emptyset} \equiv 1$ and 
the recursive asymptotics~\textnormal{(}\ref{eqn::USTASY1},~\ref{eqn::USTASY2}\textnormal{)}.
\end{proposition}

\begin{proof}
The normalization property $\LF_{\emptyset} \equiv 1$ 
is understood as an empty product of integrals in the definition~\eqref{eq: Def of F beta}. 
For the asymptotics, note that after conjugating by a suitable M\"obius transformation, 
by Proposition~\ref{prop: full Mobius covariance F} 
it suffices to prove~(\ref{eqn::USTASY1},~\ref{eqn::USTASY2}) for $j=1$. 
Furthermore, by translation invariance (e.g., Proposition~\ref{prop: full Mobius covariance F}), 
we may assume without loss of generality that $\xi = 0$ (this will simplify some computations).

Now, fix $\beta \in \LP_N$ with link endpoints ordered as in~\eqref{eq: link pattern ordering}. 
It thus remains to consider the limit of $\LF_{\beta}$
as $x_1, x_2 \to 0$ for $0 < x_3 < x_4 < \cdots < x_{2N}$ in the following two cases:
\begin{enumerate} 
\item  \label{item::case1} $\{a_1,b_1\} = \{1,2\} \in \beta$; or 

\smallbreak

\item  \label{item::case2} $\{1,2\} \notin \beta$, in which case there are some indices $b_1 \in \{4,6,\ldots,2N\}$ and
$b_2 \in \{3,5,\ldots,2N-1\}$ such that 
$\{a_1,b_1\} = \{1,b_1\}$ and $\{a_2,b_2\} = \{2,b_2\}$. 
\end{enumerate}
We write $\boldsymbol{x} = (x_1, \ldots, x_{2N}) \in \chamber_{2N}$
and $\boldsymbol{\ddot{x}} = (x_3, x_4, \ldots, x_{2N})$. 
By the identities in  
Proposition~\ref{prop::two_Fs_as_determinants}, we have 
$\LF_\beta(\boldsymbol{x}) 
= 2^{-N} \LF_\beta^\circ(\boldsymbol{x}) 
= 2^{1-N}\pi \ii \, \zeta^\circ_{\beta} \, f^{(0)} (\boldsymbol{x}) \, \det A_{\beta}^{\circ} (\boldsymbol{x})$, 
where
\begin{itemize}[leftmargin=2em]
\item $\zeta^\circ_{\beta} = \ii^{\sum_{s=1}^N (2N-a_s+2)}$ is a phase factor satisfying 
\begin{align*}
\zeta^\circ_{\beta} = -\ii \, \zeta^\circ_{\beta / \{1,2\}}
\qquad \textnormal{and} \qquad 
\zeta^\circ_{\beta} = \ii^{b_2+1} \, \zeta^\circ_{\wp_1(\beta)/\{1,2\}} ;
\end{align*}

\smallbreak

\item $f^{(0)}$ is defined in Equation~\eqref{eqn::def_fnod} and satisfies the simple asymptotics
\begin{align*}
\lim_{x_1 , x_2 \to 0} \frac{f^{(0)}(\boldsymbol{x})}{(x_{2}-x_{1})^{1/4}} 
= \; & f^{(0)}(\boldsymbol{\ddot{x}})\times \prod_{3\le i\le 2N} x_i^{1/2} ;
\end{align*}

\smallbreak

\item and $\det A_{\beta}^{\circ}$ is the determinant of the matrix of $a$-periods defined in~\eqref{eq::A_matrices}, 
whose asymptotics we derive in Lemmas~\ref{lem: Asymptotics1 A} and~\ref{lem: Asymptotics2 A} below (respectively in Cases~\ref{item::case1} and~\ref{item::case2}).
\end{itemize}
Combining these inputs, we obtain the asserted limits~(\ref{eqn::USTASY1},~\ref{eqn::USTASY2}).
\end{proof}

\begin{lemma} \label{lem: Asymptotics1 A}
Suppose $\{a_1,b_1\} = \{1,2\} \in \beta$, and write $\boldsymbol{x} = (x_1, \ldots, x_{2N}) \in \chamber_{2N}$
and $\boldsymbol{\ddot{x}} = (x_3, x_4, \ldots, x_{2N})$, where $0 < x_3 < x_4 < \cdots < x_{2N}$.
Then, we have
\begin{align} \label{eq: Asymptotics1 P}
\lim_{x_1 , x_2 \to 0} \det A_{\beta}^{\circ} (\boldsymbol{x})
= \; & \frac{ 2 \pi \ii }{\prod_{i=3}^{2N}x_{i}^{1/2}} \; \det A_{\beta/\{1,2\}}^{\circ} (\boldsymbol{\ddot{x}}) ,
\end{align}
where $\beta/\{1,2\} \in \LP_{N-1}$ denotes the link pattern obtained from $\beta$ by removing the link $\{1,2\}$ and relabeling the remaining indices by $1, 2, \ldots, 2N-2$.
\end{lemma}

\begin{proof}
We expand the determinant according to the cofactors along the first column:
\begin{align} \label{eq: det_expansion1}
\det A_{\beta}^{\circ} 
= \; & \sum_{r=1}^{N-1} (-1)^{1+r} \, ( \det (\hat{A}_{\beta}^{\circ})_{r,1} ) \,
\ointclockwise_{\acycle^\beta_{r+1}} \frac{\ud u}{\wfunc(u; \boldsymbol{x})} ,
\end{align}
where 
$\det (\hat{A}_{\beta}^{\circ})_{r,1}$ is the minor obtained from $A_{\beta}^{\circ}$ by removing the first column and $r$:th row.

On the one hand, because the integration contours $\acycle_2^\beta, \acycle_3^\beta, \ldots, \acycle_{N}^\beta$ 
remain bounded away from each other and from the points $x_1$ and $x_2$, 
and their homotopy types do not change upon taking the limit $x_1 , x_2 \to 0$, 
by dominated convergence we see that, 
the matrix entries of $(\hat{A}_{\beta}^{\circ})_{r,1}$, with $r \in \{1,2,\ldots,N-1\}$ and $s \in \{2,3,\ldots,N-1\}$, have finite limits: 
\begin{align*}
\lim_{x_1 , x_2 \to 0} \ointclockwise_{\acycle^\beta_{r+1}} \omega_{s-1} 
= \; & \lim_{x_1 , x_2 \to 0} \ointclockwise_{\acycle^\beta_{r+1}} \frac{u^{s-1} \ud u}{\wfunc(u; \boldsymbol{x})} 
\; = \; 
\ointclockwise_{\acycle^\beta_{r+1}} \frac{u^{s-2} \ud u}{\wfunc(u; \boldsymbol{\ddot{x}})} 
\; = \; \ointclockwise_{\hat{\acycle}^\beta_{r}} \hat{\omega}_{s-2} ,
\end{align*}
where $\hat{\omega}_{s-2}$ are the holomorphic one-forms 
and $\hat{\acycle}^\beta_{r}$ the $a$-cycles associated to $\Sigma_{x_3,x_4,\ldots,x_{2N}}$.

On the other hand, the matrix entry 
$(A_{\beta}^{\circ})_{r,1} = \ointclockwise_{\acycle^\beta_{r+1}} \omega_0 = \ointclockwise_{\acycle^\beta_{r+1}} \frac{\ud u}{\wfunc(u; \boldsymbol{x})}$ has a similar limit: 
\begin{align*}
\lim_{x_1 , x_2 \to 0} \ointclockwise_{\acycle^\beta_{r+1}} \frac{\ud u}{\wfunc(u; \boldsymbol{x})}
\; = \; \ointclockwise_{\acycle^\beta_{r+1}} \frac{\ud u}{u \, \wfunc(u; \boldsymbol{\ddot{x}})}
\; = \; \ointclockwise_{\hat{\acycle}^\beta_{r}} \tilde{\omega} ,
\end{align*}
where $\smash{\tilde{\omega} = \frac{\ud u}{u \, \wfunc(u; \boldsymbol{\ddot{x}})}}$ is 
a meromorphic one-form on $\Sigma_{x_3,x_4,\ldots,x_{2N}}$ 
with simple poles at the two copies $0^\pm$ of the origin with residues
$\smash{\pm \frac{1}{\wfunc(0; \boldsymbol{\ddot{x}})}}$. 
In conclusion, we have
\begin{align} \label{eq: Asymptotics1 A pre}
\lim_{x_1 , x_2 \to 0} \det A_{\beta}^{\circ} (\boldsymbol{x}) 
\quad = \quad
\begin{vmatrix}
\underset{\hat{\acycle}^\beta_{1}}{\ointclockwise} \tilde{\omega} \quad
& \underset{\hat{\acycle}^\beta_{1}}{\ointclockwise} \hat{\omega}_{0} \quad
& \underset{\hat{\acycle}^\beta_{1}}{\ointclockwise} \hat{\omega}_{1} \quad
& \cdots & 
\underset{\hat{\acycle}^\beta_{1}}{\ointclockwise} \hat{\omega}_{N-3} \\[4ex]
\underset{\hat{\acycle}^\beta_{2}}{\ointclockwise} \tilde{\omega} \quad
& \underset{\hat{\acycle}^\beta_{2}}{\ointclockwise} \hat{\omega}_{0} \quad
& \underset{\hat{\acycle}^\beta_{2}}{\ointclockwise} \hat{\omega}_{1} \quad
& \cdots & \underset{\hat{\acycle}^\beta_{2}}{\ointclockwise} \hat{\omega}_{N-3}  \\[2ex]
\vdots & \cdots & \ddots & \vdots \\[2ex]
\underset{\hat{\acycle}^\beta_{N-1}}{\ointclockwise} \tilde{\omega}\quad
& \underset{\hat{\acycle}^\beta_{N-1}}{\ointclockwise} \hat{\omega}_{0} \quad
& \underset{\hat{\acycle}^\beta_{N-1}}{\ointclockwise} \hat{\omega}_{1} \quad
& \cdots & \underset{\hat{\acycle}^\beta_{N-1}}{\ointclockwise} \hat{\omega}_{N-3} 
\end{vmatrix} .
\end{align}
After adding all of the other rows to the first row in~\eqref{eq: Asymptotics1 A pre}, recalling the relation 
\begin{align*}
\hat{\acycle}^\beta_{1} = - ( \hat{\acycle}^\beta_{2} + \hat{\acycle}^\beta_{3} + \cdots + \hat{\acycle}^\beta_{N-1} ) \quad 
\textnormal{in } H_1(\Sigma_{x_3,x_4,\ldots,x_{2N}}, \Z) ,
\end{align*}
and noting that 
\begin{align} \label{eq: res_at_zero}
\sum_{r=1}^{N-1} \ointclockwise_{\hat{\acycle}^\beta_{r}} \tilde{\omega}
\; = \; 2 \pi \ii \; \underset{0^+}{\mathrm{Res}} \, \tilde{\omega}
\; = \; \frac{2 \pi \ii}{\wfunc(0; \boldsymbol{\ddot{x}})} 
\; = \; \frac{2 \pi \ii}{\prod_{i=3}^{2N}x_{i}^{1/2}} ,
\end{align}
we see that the right-hand side of~\eqref{eq: Asymptotics1 A pre} equals the determinant of
the matrix whose first row comprises~\eqref{eq: res_at_zero} and $N-2$ zeros, 
while its other (unchanged) rows coincide with those of the right-hand side of~\eqref{eq: Asymptotics1 A pre}.
In particular, since its principal submatrix obtained by removing the first row and the first column is $A_{\beta/\{1,2\}}^{\circ} (\boldsymbol{\ddot{x}})$, we conclude that~\eqref{eq: Asymptotics1 P} indeed holds. 
\end{proof}

\begin{lemma} \label{lem: Asymptotics2 A}
Suppose that $\{1,2\} \notin \beta$, 
and write $\{a_1,b_1\} = \{1,b_1\}$ and $\{a_2,b_2\} = \{2,b_2\}$ with indices $b_1 \in \{4,6,\ldots,2N\}$ and $b_2 \in \{3,5,\ldots,2N-1\}$, and
$\boldsymbol{x} = (x_1, \ldots, x_{2N}) \in \chamber_{2N}$
and $\boldsymbol{\ddot{x}} = (x_3, x_4, \ldots, x_{2N})$, where $0 < x_3 < x_4 < \cdots < x_{2N}$.
Then, we have
\begin{align} \label{eq: Asymptotics2 A}
\lim_{x_1 , x_2 \to 0} \frac{\det A_{\beta}^{\circ} (\boldsymbol{x})}{|\log(x_2-x_1)| } 
= \; & \frac{2 (-1)^{(b_2+1)/2}}{\prod_{i=3}^{2N}x_{i}^{1/2}} \; \det A_{\wp_1(\beta)/\{1,2\}}^{\circ} (\boldsymbol{\ddot{x}}) ,
\end{align}
where $\wp_1$ is the tying operation~\eqref{eq::tying_operation}.
\end{lemma}

\begin{proof}
We expand the determinant as in Equation~\eqref{eq: det_expansion1}.
On the one hand, the matrix elements of $(\hat{A}_{\beta}^{\circ})_{r,1}$ all remain finite in the limit $x_1 , x_2 \to 0$.
On the other hand, the matrix entry $(A_{\beta}^{\circ})_{r,1}$ also has a finite limit unless $r=1$. 
For the matrix entry $(A_{\beta}^{\circ})_{1,1}$, we can evaluate its limit 
by performing the change of variables $\smash{v = \frac{u - x_2}{x_{b_2} - x_2}}$:
\begin{align} \label{eq: limit of first matrix entry}
\begin{split}
\; & \lim_{x_1 , x_2 \to 0} \frac{1}{|\log(x_2-x_1)| } \; \ointclockwise_{\acycle^\beta_{2}} \omega_0 \\
= \; & \lim_{x_1 , x_2 \to 0} \frac{1}{|\log(x_2-x_1)| } \; 
\ointclockwise_{0}^{1} \frac{\ud v}{\sqrt{v \big(v+\frac{x_2-x_1}{x_{b_2}-x_2}\big)}}
\; \frac{1}{\wfunc((x_{b_2} - x_2) v + x_2; \boldsymbol{\ddot{x}})} .
\end{split}
\end{align}
For every $\epsilon>0$ and $c_1>0$, we can choose small $c_2 > 0$ such that $|\wfunc(u; \boldsymbol{\ddot{x}})| \le M$ and $| \wfunc(u; \boldsymbol{\ddot{x}}) - \wfunc(0; \boldsymbol{\ddot{x}}) | \le \epsilon$ for all $u \in [- 2 \, c_2 \, x_{b_2} , c_2 \, x_{b_2}]$,  
and $c_1 \frac{x_2-x_1}{x_{b_2}-x_2} \le c_2$. 
Then, 
\begin{align*}
\bigg| 
\landupint_0^{c_1 \frac{x_2-x_1}{x_{b_2}-x_2}} \ud v \,
\frac{\ud v}{\sqrt{v \big(v+\frac{x_2-x_1}{x_{b_2}-x_2}\big)}}
\; \frac{1}{\wfunc((x_{b_2} - x_2) v + x_2; \boldsymbol{\ddot{x}})} 
\bigg| 
\; \le \; M \landupint_0^{c_1} \frac{\ud v}{ \sqrt{v(v+1)} } ,
\end{align*}
which is bounded for any finite $c_1$; 
\begin{align*}
\bigg| 
\landupint_{c_2}^{1}  \ud v \,
\frac{\ud v}{\sqrt{v \big(v+\frac{x_2-x_1}{x_{b_2}-x_2}\big)}}
\; \frac{1}{\wfunc((x_{b_2} - x_2) v + x_2; \boldsymbol{\ddot{x}})} 
\bigg| 
\; \le \; \frac{1}{c_2} 
\landupint_{0}^{1} \ud v \, \frac{1}{|\wfunc((x_{b_2} - x_2) v + x_2; \boldsymbol{\ddot{x}})|}  ,
\end{align*}
which is bounded for any non-zero $c_2$;  
and 
using the bound $\sqrt{\frac{1}{v(v+1)}} \leq \frac{1}{v}$, we also find 
\begin{align*}
\bigg|
\landupint_{c_1 \frac{x_2-x_1}{x_{b_2}-x_2}}^{c_2} 
\big( \wfunc(u; \boldsymbol{\ddot{x}}) - \wfunc(0; \boldsymbol{\ddot{x}}) \big) \,
\frac{\ud v}{\sqrt{v \big(v+\frac{x_2-x_1}{x_{b_2}-x_2}\big)}}
\bigg|
\; \le \; & \epsilon 
\log \Big( \frac{c_2 \, (x_{b_2}-x_2)}{c_1 \, (x_2-x_1) } \Big) ,
\end{align*}
which blows up like $|\log(x_2-x_1)|$ as $x_1, x_2 \to 0$, 
while $c_1$, $c_2$, and $\epsilon$ are kept fixed.  
Noting that $v \mapsto \smash{\sqrt{\tfrac{v}{v+\lambda}}}$ is increasing on $[0,\infty)$ for any constant $\lambda > 0$, we find (with $\lambda = c_1 \frac{x_2-x_1}{x_{b_2}-x_2}$)
\begin{align*}
\sqrt{\frac{c_1}{c_1+1}} \,
\log \Big( \frac{c_2 \, (x_{b_2}-x_2)}{c_1 \, (x_2-x_1) } \Big) 
\; \le \; 
\landupint_{c_1 \frac{x_2-x_1}{x_{b_2}-x_2}}^{c_2} 
\frac{\ud v}{\sqrt{v \big(v+\frac{x_2-x_1}{x_{b_2}-x_2}\big)}}
\; \le \; 
\log \Big( \frac{c_2 \, (x_{b_2}-x_2)}{c_1 \, (x_2-x_1) } \Big) ,
\end{align*}
whose upper bound is of order $|\log(x_2-x_1)|$ and lower bound 
of order $\smash{\sqrt{\frac{c_1}{c_1+1}}} \, |\log(x_2-x_1)|$ as $x_1, x_2 \to 0$, 
while $c_1$ is kept fixed.
Combining the above four estimates and taking the limits
$x_1, x_2 \to 0$; then $c_2 \to 0$; then $\epsilon \to 0$; and then $c_1 \to \infty$ (in this order), we obtain\footnote{Here, we can decompose the loop integral in~\eqref{eq: limit of first matrix entry} as a linear combination of line integrals as in Appendix~\ref{app::matrices}, which gives the factor ``$2$''. 
The integral concentrates near the starting point of the contour in the limit.} 
\begin{align*}
\textnormal{\eqref{eq: limit of first matrix entry}}
\; = \; \frac{2}{\wfunc(0; \boldsymbol{\ddot{x}})} 
\; = \; \frac{2}{\prod_{i=3}^{2N}x_{i}^{1/2}} ,
\end{align*}
In conclusion, the only term in the determinant~\eqref{eq: det_expansion1} divided by $|\log(x_2-x_1)|$ which survives in the limit $x_1, x_2 \to 0$ is the one with $r=1$, and this limit equals
\begin{align*} 
\lim_{x_1 , x_2 \to 0} \frac{\det A_{\beta}^{\circ} (\boldsymbol{x})}{|\log(x_2-x_1)| } 
= \; & \lim_{x_1 , x_2 \to 0} ( \det (\hat{A}_{\beta}^{\circ})_{1,1} ) \, 
\frac{1}{|\log(x_2-x_1)| }  \ointclockwise_{\acycle^\beta_{2}} \omega_0 \\
= \; & \frac{2  (-1)^{(b_2+1)/2}}{\prod_{i=3}^{2N}x_{i}^{1/2}} \; \det A_{\wp_1(\beta)/\{1,2\}}^{\circ} (\boldsymbol{\ddot{x}}) ,
\end{align*}
where we evaluated the limit of $\det (\hat{A}_{\beta}^{\circ})_{1,1}$ similarly as in the proof of Lemma~\ref{lem: Asymptotics1 A}, 
observing that the result is the period matrix for $\wp_1(\beta)/\{1,2\}$ with the $\smash{(\tfrac{b_2-1}{2})}$:th $a$-cycle removed, 
yielding\footnote{Recall that in the $a$-period matrices, we omit, by convention, the first $a$-cycle.} the multiplicative sign factor 
$(-1)^{1+(b_2-1)/2} = (-1)^{(b_2+1)/2}$.
\end{proof}
}

\subsection{Positivity}
\label{subsec::positivityproof}

{

In this section, we prove that $\LF_{\beta}$ can be chosen to be simultaneously positive, thus verifying property~(POS) in Theorem~\ref{thm::coulombgasintegral}.

\smallbreak

We first record a very useful general property of Coulomb gas integrals of type~\eqref{eq: ansatz8} 
with $\ell \in \{2,3,\ldots, N\}$ and $\Gamma = \acycle_1 \times \acycle_2 \times \cdots \times \acycle_\ell$, where $\acycle_1, \acycle_2, \ldots, \acycle_\ell$ are clockwise\footnote{One could also orient them counterclockwise --- what is important is that all loops have the same orientation.} oriented simple mutually non-intersecting loops on $\C \setminus \{\gamma_1^\beta, \ldots, \gamma_{2N}^\beta\}$. 
Namely, for any fixed $r \in \{1,2,\ldots,\ell\}$, we can replace the integration 
along $\acycle_r$ by an integration along another simple loop $\Gloop_r$ obtained from $\acycle_r$ by pulling it over some of the other loops in $\Gamma$ (as specified in Lemma~\ref{lem::circle_integral_general} below).
This property hold regardless of the branch choice for the integrand $f$ defined in~\eqref{eq: integrand}. 
The setup is illustrated in Figure~\ref{fig: nesting deformation}. 

\begin{lemma} \label{lem::circle_integral_general}
Fix $\ell \in \{2,3,\ldots,N\}$ and $r \in \{1,2,\ldots,\ell\}$. 
Let $\Gloop_r$ be a clockwise oriented simple loop on $\C \setminus \{\gamma_1^\beta, \ldots, \gamma_{2N}^\beta\}$ obtained from $\acycle_r$ by pulling $\acycle_r$ over some of the other loops in $\{\acycle_1, \acycle_2, \ldots, \acycle_\ell\} \setminus \{\acycle_r\}$\textnormal{:}
\begin{align} \label{eq::Gloop_in_homology}
\Gloop_r = \acycle_r + \sum_{s \in I_r} \acycle_s \quad 
\textnormal{in } H_1(\Sigma, \Z) ,
\end{align}
where $I_r \subset \{1,2,\ldots, \ell\} \setminus \{r\}$. 
Then, writing $\boldsymbol{u} = (u_1, \ldots, u_\ell)$, we have
\begin{align} \label{eq::integra_deformation}
\begin{split}
\; & 
\ointclockwise_{\acycle_1}  \ud u_1 \ointclockwise_{\acycle_2}  \ud u_2 
\cdots \ointclockwise_{\acycle_\ell} \ud u_\ell 
\; f (\boldsymbol{x};\boldsymbol{u}) \\
= \; & \ointclockwise_{\acycle_1} \ud u_1 
\cdots \ointclockwise_{\acycle_{r-1}} \ud u_{r-1} 
\ointclockwise_{\Gloop_r} \ud u_r
\ointclockwise_{\acycle_{r+1}} \ud u_{r+1} 
\cdots \ointclockwise_{\acycle_{\ell}} \ud u_{\ell} 
\; f (\boldsymbol{x};\boldsymbol{u}) .
\end{split}
\end{align}
\end{lemma}

Note that the replacement in Lemma~\ref{lem::circle_integral_general} is only valid when we integrate $f (\boldsymbol{x}; u_1, \ldots, u_\ell)$
along all $\ell$ loops $\acycle_1, \acycle_2, \ldots, \acycle_\ell$, 
in which case we can use antisymmetry to get cancellations.

\begin{proof}
After deforming and decomposing the loop $\Gloop_r$ into the linear combination~\eqref{eq::Gloop_in_homology},  
we see by antisymmetry of the integrand and Fubini's theorem 
that only $\acycle_r$ can give a non-zero contribution on the right-hand side of~\eqref{eq::integra_deformation}: 
indeed, for any $s \neq r$, the double-integral of~\eqref{eq: integrand} along $\acycle_s$ vanishes (here, we use the assumption that $\ell \geq 2$), 
\begin{align*}
\ointclockwise_{\acycle_s} \ud u_s \ointclockwise_{\acycle_s}  \ud u_r
\; f (\boldsymbol{x};u_1,\ldots,u_\ell)
= 0 ,
\end{align*}
by antisymmetry of the integrand~\eqref{eq: integrand} with respect to the exchange 
$u_s \leftrightarrow u_r$.
\end{proof}

\begin{figure}[ht!]
\includegraphics[width=0.75\textwidth]{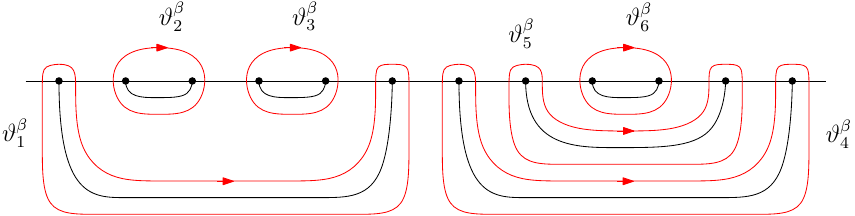}
\[\Downarrow\]
\includegraphics[width=0.75\textwidth]{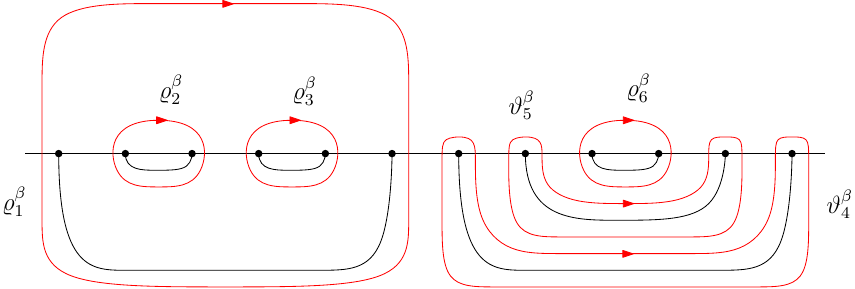}
\[\Downarrow\]
\includegraphics[width=0.75\textwidth]{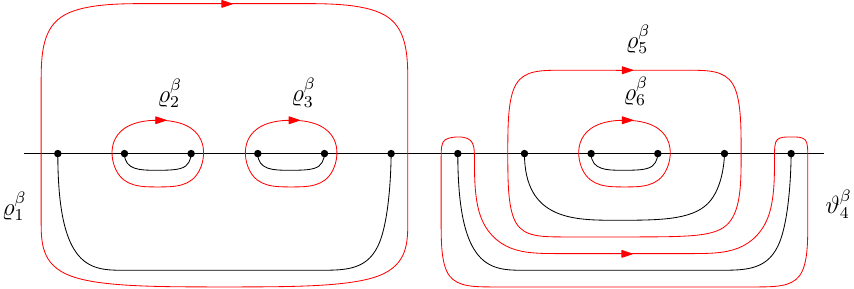}
\[\Downarrow\]
\includegraphics[width=0.75\textwidth]{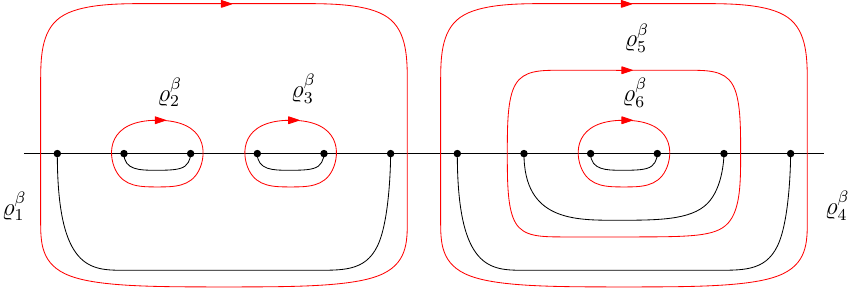}
\caption{\label{fig: nesting deformation}Illustration of a repeated application of Lemma~\ref{lem::circle_integral_general} in the proof of Corollary~\ref{cor: F_real}: we iteratively transform the integration contours $\smash{\Gamma_\beta := \acycle_1^\beta \times \cdots \times \acycle_N^\beta}$ (top line) into the integration contours $\smash{\GGloop_\beta := \Gloop_1^\beta \times \cdots \times \Gloop_N^\beta}$ (bottom line), 
which are symmetric with respect to the real axis.
In this example, we have $\beta=\{\{1,6\}, \{2,3\}, \{4,5\}, \{7,12\}, \{8,11\}, \{9,10\}\}$.
}
\end{figure}

\begin{corollary} \label{cor: F_real}
For each $\beta \in \LP_N$, 
the functions $\LF_\beta, \LF_\beta^\circ \colon \chamber_{2N} \to \R$ are real-valued.
\end{corollary}

\begin{proof}
A repeated application of Lemma~\ref{lem::circle_integral_general} shows that 
the integration contours 
$\Gamma_\beta := \acycle_1^\beta \times \cdots \times \acycle_N^\beta$ in 
$\LF_\beta^\circ (\boldsymbol{x}) 
:= \ointclockwise_{\Gamma_\beta} f_\beta(\boldsymbol{x};\boldsymbol{u}) \; \ud \boldsymbol{u}$  
can be deformed into ones that are symmetric with respect to the real axis, which we write as 
$\smash{\GGloop_\beta := \Gloop_1^\beta \times \cdots \times \Gloop_N^\beta}$ (see also Figure~\ref{fig: nesting deformation}). 
Then, after making the change of variables by complex conjugation on the integrations along each $\Gloop_r^\beta \cap \overline{\HH}^*$, we see that
\begin{align*}
\LF_\beta^\circ (\boldsymbol{x}) 
= \; & \ointclockwise_{\Gloop^\beta_1}  \ud u_1 \ointclockwise_{\Gloop^\beta_2}  \ud u_2 \cdots \ointclockwise_{\Gloop^\beta_N} \ud u_N 
\; f_\beta(\boldsymbol{x};\boldsymbol{u}) \\
= \; & 2^N \ointclockwise_{\Gloop^\beta_1 \cap \overline{\HH}}  \ud u_1 \ointclockwise_{\Gloop^\beta_2 \cap \overline{\HH}} \ud u_2 \cdots \ointclockwise_{\Gloop^\beta_N \cap \overline{\HH}} \ud u_N 
\; \Re f_\beta(\boldsymbol{x};\boldsymbol{u}) 
&& \textnormal{[by Lem.~\ref{lem::circle_integral_general}]}
\\
= \; & 2^N \LF_\beta (\boldsymbol{x}) ,
&& \textnormal{[by Prop.~\ref{prop::two_Fs_as_determinants}]}
\end{align*}
which shows that $\LF_\beta, \LF_\beta^\circ \colon \chamber_{2N} \to \R$ are real-valued.
\end{proof}

\begin{proposition} \label{prop: positivity}
For each $\beta\in\LP_N$, we have $\LF_{\beta}(\boldsymbol{x})>0$, for all $\boldsymbol{x}\in\chamber_{2N}$. 
\end{proposition}

\begin{proof}
It follows from Corollary~\ref{cor: F_real} that $\LF_{\beta}(\boldsymbol{x}) \in \R$ 
and from Proposition~\ref{prop::two_Fs_as_determinants} that $\LF_{\beta}(\boldsymbol{x}) \neq 0$, for all $\boldsymbol{x}\in\chamber_{2N}$ and $\beta\in\LP_N$. 
As each function $\boldsymbol{x} \mapsto \LF_{\beta}(\boldsymbol{x})$ is continuous on $\chamber_{2N}$, we see that for each $\beta\in\LP_N$, we have either $\LF_{\beta} < 0$ on $\chamber_{2N}$,  
or $\LF_{\beta} > 0$ on $\chamber_{2N}$.
By the recursive asymptotics~(\ref{eqn::USTASY1},~\ref{eqn::USTASY2})  combined with the normalization $\LF_{\emptyset} \equiv 1$, we see that all of them must be simultaneously positive. 
\end{proof} 
}


\smallskip{}

\section{Uniform spanning tree in polygons}
\label{sec::ust}

In this section, we consider scaling limits of Peano curves for uniform spanning trees (UST).
In the pioneering work~\cite{LawlerSchrammWernerLERWUST}, 
Lawler, Schramm \& Werner showed that Schramm's $\SLE_\kappa$ curve~\cite{SchrammScalinglimitsLERWUST} with $\kappa = 8$
describes the scaling limit of the UST Peano curve with Dobrushin boundary conditions. 
We will address the general case of any number of Peano curves, with boundary conditions $\beta$ encoded in~\eqref{eq: link pattern ordering}.
Such generalizations have been also considered  in~\cite{DubedatEulerIntegralsCommutingSLEs, KenyonWilsonBoundaryPartitionsTreesDimers, HanLiuWuUST}.
To identify the limit and show its uniqueness, one uses a discrete holomorphic observable. 
The case of two Peano curves with alternating (free/wired/free/wired) boundary conditions was detailed in~\cite{HanLiuWuUST} 
by using Smirnov's four-point observable.
However, in order to address the general case of any number of Peano curves,
it is crucial to find a new appropriate observable (see Definition~\ref{def: exploration path} and 
Lemma~\ref{lem::holo_general}).
For cases involving non-trivial conformal moduli, 
in the simplest example $\beta=\unnested$ (as in~\eqref{eqn::unnested}),  
where every other boundary interval is independently wired, 
the scaling limit of the relevant observable for the identification was also pointed out in~Dub\'edat~\cite[Section~3.3]{DubedatEulerIntegralsCommutingSLEs} 
and Kenyon~\&~Wilson~\cite[Section~5.2]{KenyonWilsonBoundaryPartitionsTreesDimers}.
We provide a general formula for it in Proposition~\ref{prop::holo_cvg} and Appendix~\ref{app::SC_mappings}.

\smallbreak

The main result of this section, Theorem~\ref{thm::ust_general}, 
describes the scaling limit curves explicitly in terms of $\SLE_8$ type processes 
with specific partition functions (namely $\LF_{\beta}$ introduced in 
Definition~\ref{def: Lbeta}, Section~\ref{subsec::Coulomb_gas_Fbeta}). 
The proof of Theorem~\ref{thm::ust_general} uses techniques from discrete complex analysis,
conventional for addressing conformally invariant scaling limits of discrete systems.
To begin, we collect some preliminaries.

\subsubsection*{Square lattice}
$\Z^2$ is the graph with vertex set $V(\Z^2):=\{(m, n) \colon m, n\in \Z\}$ and edge set $E(\Z^2)$ given by edges between nearest neighbors (i.e., pairs of vertices with distance one).
This is our primal lattice. Its dual lattice is denoted by $(\Z^2)^*$. The medial lattice $(\Z^2)^{\diamond}$ is the graph with centers of edges of $\Z^2$ as vertex set and edges connecting nearest neighbors. 
In this article, when we add the subscript or superscript $\delta$, we mean that subgraphs of the lattices $\Z^2, (\Z^2)^*, (\Z^2)^\diamond$ have been scaled by $\delta > 0$. 
We shall consider the models in the scaling limit $\delta \to 0$.

\subsubsection*{Discrete holomorphicity}
A function $\phi \colon \Z^2 \cup (\Z^2)^* \to \C$ is (discrete) \emph{holomorphic} around a medial vertex $x^{\diamond}$ if we have $\phi(n)-\phi(s) = \ii \, (\phi(e)-\phi(w))$, where $n, w, s, e$ are the vertices incident to $x^{\diamond}$ in counterclockwise order
(two of them are vertices of $\Z^2$ and the other two are vertices of $(\Z^2)^*$). 
We say that $\phi$ is holomorphic on a subgraph of $\Z^2\cup(\Z^2)^*$ if it is holomorphic at all vertices in the subgraph. 
See~\cite[Section~8]{DuminilCopinParafermionic} for more details.

\subsubsection*{Uniform spanning tree (UST)}
Suppose that $G=(V,E)$ is a finite connected graph. 
A \emph{forest} is a subgraph of $G$ that has no loops. 
A \emph{tree} is a connected forest. 
A subgraph of $G$ is \emph{spanning} if it covers $V$. 
The \emph{uniform spanning tree} (UST) on $G$ is a probability measure on the set of all spanning trees of $G$ in which every tree is chosen with equal probability.  
Given a disjoint set $\tau = \bigcup_{k=1}^N \tau_k$ of trees $\tau_k$ of $G$, a spanning tree with $\tau$ \emph{wired} is a spanning tree $T$ of $G$ such that $\tau \subset T$.
The \emph{uniform spanning tree with $\tau$ wired} is a probability measure on the set of all spanning trees of $G$ with $\tau$ wired in which every tree is chosen with equal probability.
In this article, we focus on UST in polygons, with various boundary conditions described via  wired boundary arcs.

\subsubsection*{Discrete polygons}

Informally speaking, a discrete polygon is a bounded simply connected subgraph $\Omega$ of $\Z^2$ 
with $2N$ fixed boundary points $x_1, x_2, \ldots, x_{2N}$ in counterclockwise order, 
and such that the 
``odd''
boundary arcs $(x_{2r-1} \, x_{2r})$ with $1\le r\le N$ are on the primal lattice $\Z^2$, 
and the 
``even''
boundary arcs $(x_{2r} \, x_{2r+1})$ with $1\le r\le N$ are on the dual lattice $(\Z^2)^*$. 
The precise definition is given below; see also Figure~\ref{fig::polygonpre} for an illustration.

Consider the medial lattice $(\Z^2)^{\diamond}$ with the following orientation of its edges:
edges of each face containing a vertex of $\Z^2$ are oriented counterclockwise, 
and edges of each face containing a vertex of $(\Z^2)^*$ are oriented clockwise.
Let $x_1^{\diamond}, \ldots, x_{2N}^{\diamond}$ be $2N$ distinct medial vertices, 
and let $(x_1^{\diamond} \, x_2^{\diamond}), \ldots, (x_{2N-1}^{\diamond} \, x_{2N}^{\diamond}), (x_{2N}^{\diamond} \, x_1^{\diamond})$ denote $2N$ oriented 
paths 
on $(\Z^2)^{\diamond}$
satisfying the following conditions\footnote{Throughout, we use the cyclic indexing convention $x_{2N+1}^{\diamond} := x_{1}^{\diamond}$ and $x_{2N+1} := x_{1}$ etc.}: 
\begin{itemize}[leftmargin=2em]
\item 
each path $(x_{2r-1}^{\diamond} \, x_{2r}^{\diamond})$ has clockwise oriented edges for $1\le r\le N$; 

\smallbreak

\item  
each path $(x_{2r}^{\diamond} \, x_{2r+1}^{\diamond})$ has counterclockwise oriented edges for $1\leq r \leq N$; and

\smallbreak

\item 
all paths are edge-avoiding and satisfy 
$(x_{i-1}^{\diamond} \, x_{i}^{\diamond}) \cap (x_i^{\diamond} \, x_{i+1}^{\diamond})=\{x_i^{\diamond}\}$ for $1\le i\le 2N$. 
\end{itemize}
Given  
$\{ (x_i^{\diamond} \, x_{i+1}^{\diamond}) \colon 1\le i\le 2N \}$, 
the \emph{medial polygon} $(\Omega^{\diamond}; x_1^{\diamond}, \ldots, x_{2N}^{\diamond})$ is defined as the subgraph of $(\Z^2)^{\diamond}$ induced by the vertices enclosed by or lying on the 
non-oriented loop $\partial\Omega^\diamond$ obtained by concatenating all of $(x_i^{\diamond} \, x_{i+1}^{\diamond})$, illustrated in blue in Figure~\ref{fig::polygonpre}.

Next, let $\Omega\subset\Z^2$ be the graph with edge set consisting of all edges passing through endpoints of medial edges 
in $E(\Omega^{\diamond})\setminus \bigcup_{r=1}^N (x_{2r}^{\diamond} \, x_{2r+1}^{\diamond})$ 
and vertex set comprising the endpoints of these edges. 
For each $i \in \{1,2,\ldots,2N\}$, 
we denote by $x_i$ the vertex of $\Omega$ nearest to $x_i^{\diamond}$, and we call $(\Omega; x_1, \ldots, x_{2N})$ the \emph{primal polygon}\footnote{A cautious reader will notice that we abuse the notation slightly: 
in some occasions, we use $(\Omega; x_1, \ldots, x_{2N})$ to indicate a (continuum) polygon, i.e., 
a bounded simply connected domain $\Omega$ with $2N$ distinct marked boundary points; 
while in some other occasions, we use $(\Omega; x_1, \ldots, x_{2N})$ to indicate a primal polygon with $2N$ distinct marked boundary vertices. 
We believe this will not cause confusion, as it is clear from the context which one is being considered.}. 
We let $(x_{2r-1} \, x_{2r})$ be the set of edges corresponding to medial vertices in $(x_{2r-1}^{\diamond} \, x_{2r}^{\diamond})\cap\partial\Omega^{\diamond}$. 
Similarly, let $\Omega^*\subset (\Z^2)^*$ be the graph with edge set consisting of all edges passing through endpoints of medial edges in $E(\Omega^{\diamond})\setminus \bigcup_{r=1}^N(x_{2r-1}^{\diamond} \, x_{2r}^{\diamond})$ 
and vertex set comprising the endpoints of these edges. 
For each $i \in \{1,2,\ldots,2N\}$, 
we denote by $x_i^*$ the vertex of $\Omega^*$ nearest to $x_i^{\diamond}$, and we call $(\Omega^*; x_1^*, \ldots, x_{2N}^*)$ the \emph{dual polygon}. 
Lastly, we let $(x_{2r}^* \, x_{2r+1}^*)$ be the set of edges corresponding to medial vertices in $(x_{2r}^{\diamond} \, x_{2r}^{\diamond})\cap\partial\Omega^{\diamond}$. 
We will assume that $\Omega$ and $\Omega^*$ form bounded simply connected domains, ensuring the existence of spanning trees on both graphs.

\begin{figure}[ht!]
\begin{center}
\includegraphics[width=0.4\textwidth]{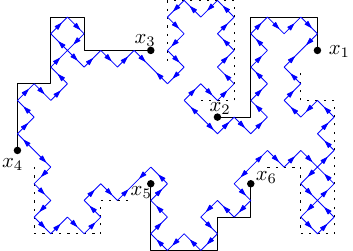}
\end{center}
\caption{\label{fig::polygonpre} 
Illustration of the boundary of a discrete polygon with six marked points on the boundary. 
The boundary arcs $(x_1 \, x_2)$, $(x_3\, x_4)$, and $(x_5 \, x_6)$ are wired.
The dashed arcs are parts of the boundary of the dual polygon.
The blue loop is the boundary $\partial\Omega^\diamond$ of the medial polygon, with the alternating orientations of its components  illustrated.}
\end{figure}

\subsubsection*{Boundary conditions and Peano curves}

We consider the following UST model on the primal polygon 
$(\Omega; x_1, \ldots, x_{2N})$.
First, each ``odd'' arc $(x_{2r-1} \, x_{2r})$ 
is wired, for $1\le r\le N$, 
and second, some of these $N$ arcs are further wired together outside of $\Omega$ according to a non-crossing partition described by a planar link pattern $\beta\in\LP_N$, 
represented by $N$ disjoint chords 
traversing between the primal components and the dual components of the non-crossing partition --- 
see Figure~\ref{fig::8pointsE_meander}(a).  
We say that the UST has \emph{boundary condition} (b.c.)~$\beta$.

Suppose that $\tree$ is a spanning tree of the primal polygon $(\Omega; x_1, \ldots, x_{2N})$ 
with b.c. $\beta\in\LP_N$. 
Then, there exist $N$ paths on $(\Z^2)^{\diamond}$ running along $\tree$ and connecting among $\{x_1^{\diamond}, \ldots, x_{2N}^{\diamond}\}$, which we call \emph{Peano curves} --- see Figure~\ref{fig::polygon}(b). 
The endpoints of these $N$ Peano curves form a random planar link pattern $\conn$ in $\LP_N$. 
As $\tree$ is a spanning tree, we see that the loop configuration formed from the chords of $\conn$ inside $\Omega$ 
and the chords of $\beta$ outside of $\Omega$ must have exactly one loop --- see Figure~\ref{fig::8pointsE_meander}(b).

\begin{figure}[ht!]
\begin{subfigure}[b]{0.3\textwidth}
\begin{center}
\includegraphics[width=0.45\textwidth]{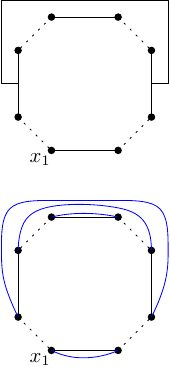}
\end{center}
\caption{A boundary condition (b.c.) $\beta$.}
\end{subfigure}
\begin{subfigure}[b]{0.675\textwidth}
\begin{center}
\includegraphics[width=0.2\textwidth]{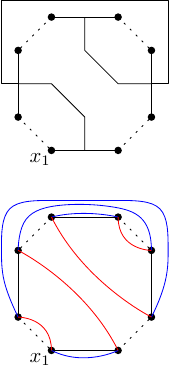}$\;$
\includegraphics[width=0.2\textwidth]{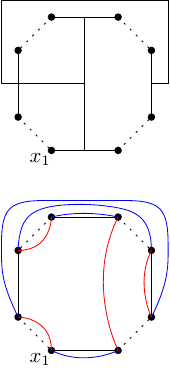}$\;$
\includegraphics[width=0.2\textwidth]{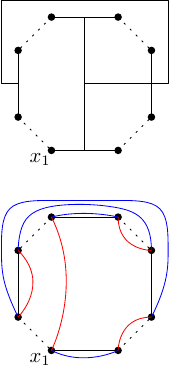}$\;$
\includegraphics[width=0.2\textwidth]{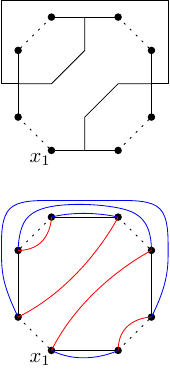}
\end{center}
\caption{Four possible planar link patterns $\alpha$ with $\LM_{\alpha,\beta}=1$.}
\end{subfigure}
\caption{\label{fig::8pointsE_meander} 
Consider discrete polygons with eight marked points on the boundary. 
One possible boundary condition 
for the UST model is indicated in the left panel. 
This b.c., encoded in the planar link pattern $\beta=\{\{1,2\},\{3,8\},\{4,7\},\{5,6\}\}$ (bottom), 
corresponds to the non-crossing partition $\{\{1\}, \{2,4\}, \{3\}\}$ of the $N$ wired boundary arcs outside of the polygon (top).
The corresponding possible planar link patterns $\conn$ formed by the Peano curves are indicated in the right panel 
(bottom), and they also correspond to non-crossing partitions inside the polygon (top). Indeed, there is a natural bijection between non-crossing partitions of the $N$ wired boundary arcs and planar link patterns with $N$ links. This correspondence will be used in Sections~\ref{sec::ust} and~\ref{sec::ppf}.
}
\end{figure}

\subsubsection*{Outline of this section}

We will consider discrete polygons approximating some continuum polygon in the plane. 
Fix $N\ge 1$ and a polygon $(\Omega; x_1, \ldots, x_{2N})$ whose boundary $\partial\Omega$ is a $C^1$-Jordan curve. 
Suppose that a sequence $(\Omega^{\delta, \diamond}; x_1^{\delta, \diamond}, \ldots, x_{2N}^{\delta, \diamond})$  of medial polygons converges to $(\Omega; x_1, \ldots, x_{2N})$ in the sense detailed in Equation~\eqref{eqn::polygon_cvg}. 
We consider the UST on the primal polygon $(\Omega^{\delta}; x_1^{\delta}, \ldots, x_{2N}^{\delta})$ with b.c. $\beta\in\LP_N$.
For each index $i \in \{1,2,\ldots,2N\}$, let $\eta_i^{\delta}$ be the Peano curve started from $x_{i}^{\delta, \diamond}$. 
Let us first note that each family $\{\eta_i^{\delta}\}_{\delta>0}$ is precompact, so we can consider its subsequential scaling limits as $\delta \to 0$.

\begin{lemma}\label{lem::Peanocurve_tight}
Assume the same setup as in Theorem~\ref{thm::ust_general}. Fix $i\in\{1, 2, \ldots, 2N\}$. 
The family of laws of $\{\eta^{\delta}_i\}_{\delta>0}$ is 
precompact in the curve space~\eqref{eqn::metric_curvesspace}.
Furthermore, any subsequential limit 
does not hit any other point in $\{x_1, x_2, \ldots, x_{2N}\}$ than its two endpoints, almost surely. 
\end{lemma}
\begin{proof}
The same argument as in~\cite[Theorems~4.7~\&~4.8]{LawlerSchrammWernerLERWUST} 
(see also~\cite[Proposition~4.3]{HanLiuWuUST}) 
shows the precompactness.
The second statement can be proven similarly as~\cite[Proposition~4.3]{HanLiuWuUST}.
\end{proof}

The goal of this section is to derive explicitly the scaling limit of the law of $\eta_i^{\delta}$ as $\delta\to 0$. 
We follow the standard strategy: 
first, we have precompactness of the sequence from Lemma~\ref{lem::Peanocurve_tight}; 
second, we construct a suitable  martingale observable in Section~\ref{subsec::holo_general}; 
and we then identify all subsequential limits through this observable in Sections~\ref{subsec::ust_cvg_observable}--\ref{subsec::ust_Peano_conv}.
The explicit identification relies on somewhat complicated analysis of the scaling limit of the observable, 
which --- quite interestingly --- is closely related to both the $a$-period matrices discussed in Section~\ref{sec::coulombgasintegrals_new}, 
and to explicit Schwarz-Christoffel type conformal mappings discussed in Appendix~\ref{app::SC_mappings}.

For definiteness, we shall construct the observable explicitly, and derive the limit of the Peano curve $\eta_1^{\delta}$ with $i=1$, in the case where $\{1, 2N\}\not\in\beta$. 
The general case follows from this 
after conjugating by a suitable M\"obius transformation, 
by Proposition~\ref{prop: full Mobius covariance F}, and possibly working with the dual tree instead of the primal tree.

\subsection{Exploration path and discrete holomorphic observable}
\label{subsec::holo_general}

Throughout, we 
fix $N\ge 2$ and a b.c. $\beta\in\LP_N$ with link endpoints ordered as in~\eqref{eq: link pattern ordering} and such that $\{1, 2N\}\not\in\beta$. 

\begin{figure}[ht!]
\begin{subfigure}[b]{0.3\textwidth}
\begin{center}
\includegraphics[width=0.5\textwidth]{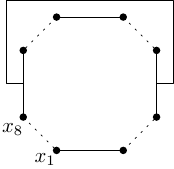}
\end{center}
\caption{A boundary condition (b.c.) $\beta$.}
\end{subfigure}
\begin{subfigure}[b]{0.675\textwidth}
\begin{center}
\includegraphics[width=0.22\textwidth]{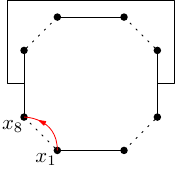}$\;$
\includegraphics[width=0.22\textwidth]{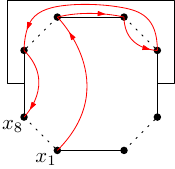}$\;$
\includegraphics[width=0.22\textwidth]{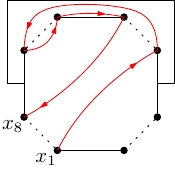}
\end{center}
\caption{Three possibilities for the exploration path from $x_1$ to $x_8$.}
\end{subfigure}
\caption{\label{fig::8pointsE_curve} 
Consider discrete polygons with eight marked points on the boundary. 
One possible boundary condition $\beta=\{\{1,2\},\{3,8\},\{4,7\},\{5,6\}\}$ is indicated in the left panel. 
The corresponding exploration path from $x_1$ to $x_8$ is indicated in the right panel.}
\end{figure}

Consider the set of spanning trees of the primal polygon $(\Omega^{\delta}; x_1^{\delta}, \ldots, x_{2N}^{\delta})$ with b.c. $\beta\in\LP_N$. 
Let $\tree_\beta^\delta$ be chosen uniformly among these spanning trees. 
The $N$ Peano curves $\smash{\eta^{\delta}_{2r-1}}$ started from $\smash{x_{2r-1}^{\delta, \diamond}}$ terminate among the medial vertices $\smash{\{x_{2r}^{\delta, \diamond} \colon 1\le r \le N\}}$. 
We define an exploration path $\xi_\beta^\delta$ along $\tree_\beta^\delta$ starting from $x_1^{\delta, \diamond}$ and terminating at $x_{2N}^{\delta, \diamond}$, which detects the meander formed from the random Peano curves inside $\Omega$ (encoded in $\conn$ in $\LP_N$)
and the given chords of the b.c.~$\beta$ outside of $\Omega$, via the following procedure (see Figure~\ref{fig::8pointsE_curve}).

\begin{definition} \label{def: exploration path}
The following rules uniquely determine 
$\xi_\beta^\delta$, called the \emph{exploration path} associated to the spanning tree $\tree_\beta^\delta$ with b.c.~$\beta$. 
\begin{enumerate} 
\item \label{item::obs-a}
$\xi_\beta^\delta$ starts from $x_1^{\delta, \diamond}$ and follows $\eta_1^{\delta}$ until it reaches some point in $\{x_{2r}^{\delta, \diamond} \colon 1\le r \le N\}$. 

\smallbreak

\item \label{item::obs-b}
When $\xi_\beta^\delta$ arrives at some point in $\{x_{2r}^{\delta, \diamond} \colon 1\le r \le N\}$, it follows the chord given by $\beta$ outside of $\Omega^{\delta}$ until it reaches some point in $\{x_{2r-1}^{\delta, \diamond} \colon 1\le r \le N\}$. 

\smallbreak

\item \label{item::obs-c}
When $\xi_\beta^\delta$ arrives at some point in $\{x_{2r-1}^{\delta, \diamond} \colon 1\le r \le N\}$, it follows the corresponding Peano curve along $\tree_\beta^\delta$ until it reaches some point in $\{x_{2r}^{\delta, \diamond}: 1\le r \le N\}$.

\smallbreak

\item \label{item::obs-d}
After repeating the steps~\ref{item::obs-b}--\ref{item::obs-c} sufficiently many times, 
$\xi_\beta^\delta$ arrives at $x_{2N}^{\delta, \diamond}$ and it then stops. 
\end{enumerate}
\end{definition}

Now we are ready to define the observable. We summarize the setup below. 
\begin{itemize} 
\item 
Consider the set $\st_\beta^\delta$
of spanning trees of the primal polygon $(\Omega^{\delta}; x_1^{\delta}, \ldots, x_{2N}^{\delta})$ with b.c. $\beta\in\LP_N$. 
Let $\tree_\beta^\delta$ be chosen uniformly among these spanning trees, 
and denote by $\xi_\beta^\delta$ the exploration path from $\smash{x_1^{\delta, \diamond}}$ to $x_{2N}^{\delta, \diamond}$.    
For each vertex $z^*$ of $\Omega^{\delta, *}$, define 
\begin{align*}
u_\beta^\delta(z^*) := \PP_{\beta}^{\delta} \big[ \textnormal{$z^*$ lies to the right of $\xi_\beta^\delta$} \big] .
\end{align*}

\smallbreak

\item 
Consider the set  
$\sf_\beta^\delta$
of spanning forests of the primal polygon $(\Omega^{\delta}; x_1^{\delta}, \ldots, x_{2N}^{\delta})$ with b.c. $\beta\in\LP_N$ 
which consist of exactly two trees: one of them
contains the wired arc $(x_1^{\delta} \, x_2^{\delta})$ and the other one contains the wired arc $(x_{2N-1}^{\delta} \, x_{2N}^{\delta})$. 
Let $\forest_\beta^\delta$ be chosen uniformly among these forests. 
For each vertex $z$ of $\Omega^{\delta}$, define 
\begin{align*}
v_\beta^\delta(z) := \PP_{\beta}^{\delta} \big[ \textnormal{$z$ lies in the same tree as the wired arc $(x_{2N-1}^{\delta} \, x_{2N}^{\delta})$ in $\forest_\beta^\delta$} \big] .
\end{align*}
\end{itemize}

We will also use the following notation. 
We consider the connected components (c.c.) of the complement of $\beta$ outside of $\Omega$. 
These c.c.'s have a chequerboard structure:  
we call the c.c.'s touching the ``even'' arcs $\bigcup_{r=1}^N (x_{2r}^{\delta, *} \, x_{2r+1}^{\delta, *})$ ``black'' and denote them 
$\LC_0^\bullet,\LC_1^\bullet,\ldots,\LC_{n(\beta)}^\bullet$, where $n(\beta)$ depends on the nesting of $\beta$, 
and by convention we denote the c.c. containing the interval $(x_{2N}^{\delta, *} \, x_{1}^{\delta, *})$ as $\LC_0^\bullet$. 
We similarly call the c.c.'s touching the ``odd'' arcs $\bigcup_{r=1}^N (x_{2r-1}^{\delta} \, x_{2r}^{\delta})$ ``white'' and denote them 
$\LC_1^\circ,\ldots,\LC_{m(\beta)}^\circ$.

\begin{lemma} \label{lem::holo_general}
The function $\phi_{\beta}^{\delta} \colon \Omega^\delta\cup \Omega^{\delta,*} \to \C$ defined as
\begin{align*}
\phi_{\beta}^{\delta}(\cdot) := u_\beta^\delta(\cdot) + \ii \, \frac{|\sf_\beta^\delta|}{|\st_\beta^\delta|} \, v_\beta^\delta(\cdot) , 
\end{align*}
which equals $u_\beta^\delta(\cdot)$ on $\Omega^{\delta, *}$ and 
$\ii \, \frac{|\sf_\beta^\delta|}{|\st_\beta^\delta|} \, v_\beta^\delta(\cdot)$ on $\Omega^\delta$, 
is discrete holomorphic on the set 
\begin{align*}
( \Omega^\delta\cup \Omega^{\delta,*} ) 
\setminus \big((x_1^{\delta} \, x_2^{\delta})\cup(x_2^{\delta,*} \, x_3^{\delta,*})\cup\cdots\cup(x_{2N-1}^{\delta} \, x_{2N}^{\delta})\cup(x_{2N}^{\delta,*} \, x_{1}^{\delta,*})\big) . 
\end{align*}
Moreover, it has the boundary data
\begin{align*}
\begin{cases}
\Re \phi_{\beta}^{\delta} \equiv 0 \; \textnormal{on }(x_{2N}^{\delta, *} \, x_1^{\delta, *}) ; \\[.3em]  
\Re \phi_{\beta}^{\delta} \equiv 1 \; \textnormal{on the other arcs in the black c.c. } \LC_0^\bullet \textnormal{ in the b.c. $\beta$} ; \\[.3em]   
\Re \phi_{\beta}^{\delta} \; \textnormal{ is constant along the other black c.c.'s } \LC_i^\bullet \textnormal{ in the b.c. $\beta$} ; \\[.3em]  
\Im \phi_{\beta}^{\delta} \equiv 0 \;\textnormal{on }(x_1^{\delta} \, x_2^{\delta}); \textnormal{ and } \\[.3em]  
\Im \phi_{\beta}^{\delta} \; \textnormal{ is constant along all white c.c.'s } \LC_j^\circ \textnormal{ in the b.c. $\beta$} .
\end{cases}
\end{align*}
\end{lemma}

\begin{proof}
The boundary data follows by construction, so we only need to show the discrete holomorphicity. 
For each vertex $z^*$ in $\Omega^{\delta, *}$, denote by $\st_\beta^\delta(z^*)$ the subset of $\st_\beta^\delta$ consisting of spanning trees 
such that $z^*$ lies to the right of $\xi_\beta^\delta$. 
For each vertex $z$ in $\Omega^{\delta}$, denote by $\sf_\beta^\delta(z)$ the subset of $\sf_\beta^\delta$ consisting of spanning forests such that $z$ lies in the same tree as the wired arc $(x_{2N-1}^{\delta} \, x_{2N}^{\delta})$. 
If $\langle n, s\rangle$ is a primal edge of $\Omega^{\delta}$, the corresponding dual edge is denoted by $\langle w, e\rangle$ such that $n$, $w$, $s$, and $e$ are in counterclockwise order.
Now, we have
\begin{align*}
u_\beta^\delta(e)-u_\beta^\delta(w)
= \;\, & \PP_{\beta}^{\delta} \big[\tree_\beta^\delta\in \st_\beta^\delta( e)\big] - \PP_{\beta}^{\delta} \big[\tree_\beta^\delta\in \st_\beta^\delta( w)\big] \\
= \;\, & \PP_{\beta}^{\delta} \big[\tree_\beta^\delta\in \st_\beta^\delta( e)\setminus \st_\beta^\delta( w)\big] - \PP_{\beta}^{\delta} \big[\tree_\beta^\delta\in \st_\beta^\delta( w)\setminus\st_\beta^\delta( e)\big] \\
\overset{(\star)}{=} \; & 
\frac{|\sf_\beta^\delta|}{|\st_\beta^\delta|} \, 
\Big(
\PP_{\beta}^{\delta} \big[\forest_\beta^\delta\in \sf_\beta^\delta( n) \setminus\sf_\beta^\delta( s)\big] 
- 
\PP_{\beta}^{\delta} \big[\forest_\beta^\delta\in \sf_\beta^\delta( s) \setminus\sf_\beta^\delta( n)\big]
\Big) 
 \\
= \; & \frac{|\sf_\beta^\delta|}{|\st_\beta^\delta|}
\, \Big(\PP_{\beta}^{\delta} \big[\forest_\beta^\delta \in \sf_\beta^\delta( n)\big] 
- \PP_{\beta}^{\delta} \big[\forest_\beta^\delta\in \sf_\beta^\delta( s)\big] \Big) \\
= \; & \frac{|\sf_\beta^\delta|}{|\st_\beta^\delta|} \, ( v_\beta^\delta(n)-v_\beta^\delta(s) ) ,
\end{align*}
where ($\star$) follows by a simple bijection between spanning trees 
in $\st_\beta^\delta(e)\setminus\st_\beta^\delta(w)$  
and spanning forests 
in $\sf_\beta^\delta( n)\setminus\sf_\beta^\delta(s)$, 
obtained by deleting the edge $\langle n,s\rangle$.
\end{proof}

We will identify the limit of the observable $\phi_{\beta}^{\delta}$ as a conformal map from $\Omega$ onto a rectangle of unit width with horizontal or vertical slits, uniquely determined by the boundary data.

\subsection{Convergence of the observable}
\label{subsec::ust_cvg_observable}

We still fix a boundary condition $\beta\in\LP_N$ with link endpoints ordered as in~\eqref{eq: link pattern ordering} and such that 
$\{1, 2N\}\not\in\beta$.
In Lemma~\ref{lem::holo_general}, we have constructed a discrete holomorphic observable $\phi_{\beta}^{\delta}$. 
In the course of the present and the subsequent sections, 
we will see that it converges as $\delta \to 0$ to its continuum analogue: 
a conformal map $\phi_{\beta}$ from the polygon $(\Omega;x_1,\ldots,x_{2N})$ onto a certain slit rectangle depending on $\beta$, 
explicitly given in terms of a (degenerate) Schwarz-Christoffel mapping, and uniquely determined by the boundary data in Lemma~\ref{lem::holo_general} --- for concrete examples,  see~Proposition~\ref{prop::holo_cvg} and Equation~\eqref{eq: Schwarz-Christoffel formula app} in Appendix~\ref{app::SC_mappings}, as well as Figure~\ref{fig::8points_observable}.

To address the convergence of the observable, we use a notion of convergence of polygons in terms of uniformizing maps with respect to the unit disc $\U = \{ z \in \C \colon |z| < 1 \}$.
As before, we regard any planar graph also as a planar domain 
by considering the union of all its vertices, edges, and faces.
We say that a sequence $(\Omega^{\delta}; x_1^{\delta}, \ldots, x_n^{\delta})$ of discrete polygons on $\delta\Z^2$ 
converges as $\delta \to 0$ to a polygon $(\Omega; x_1, \ldots, x_n)$ \emph{in the Carath\'{e}odory sense}\footnote{Equivalently, this could be phrased in terms of uniformizing maps with respect to the upper half-plane $\HH$.} 
if there exist conformal maps $\varphi_{\delta}$ from $\Omega^{\delta}$ onto $\U$,
and a conformal map $\varphi$ from $\Omega$ onto $\U$,
such that $\varphi_{\delta}^{-1} \to \varphi^{-1}$ locally uniformly on $\U$,
and $\varphi_{\delta}(x_j^{\delta}) \to \varphi(x_j)$ for all $1\le j\le n$.

The convergence of the observable will hold locally uniformly in the following sense.
If $\psi^{\delta}$ are functions on vertices of $\Omega^{\delta}$, 
we extend them to functions on the corresponding planar domains by linear interpolation.
Then, we say that the sequence $(\psi^{\delta})_{\delta>0}$ \emph{converges to a holomorphic function $\psi$ locally uniformly} 
as $\delta \to 0$ if
the corresponding maps $\psi^{\delta} \circ \varphi_{\delta}^{-1}$ converge to $\psi \circ \varphi^{-1}$ locally uniformly on $\U$.

\begin{proposition} \label{prop::holo_cvg}
Suppose that a sequence $(\Omega^{\delta, \diamond}; x_1^{\delta, \diamond}, \ldots, x_{2N}^{\delta, \diamond})$  of medial polygons converges to $(\Omega; x_1, \ldots, x_{2N})$ in the Carath\'{e}odory sense as $\delta\rightarrow 0$. 
Then, the discrete holomorphic function $\phi_{\beta}^{\delta}$ of Lemma~\ref{lem::holo_general} converges locally uniformly as $\delta\to 0$ to the unique holomorphic function $\phi_{\beta}$ on $\Omega$ with boundary data \textnormal{(}analogous to that in Lemma~\ref{lem::holo_general}\textnormal{)}
\begin{align} \label{eqn::boundarydata}
\begin{cases}
\Re \phi_{\beta} \equiv 0 \; \textnormal{on }(x_{2N} \, x_1) ; \\[.3em]  
\Re \phi_{\beta} \equiv 1 \; \textnormal{on the other arcs in the black c.c. } \LC_0^\bullet \textnormal{ in the b.c. $\beta$} ; \\[.3em]   
\Re \phi_{\beta} \; \textnormal{ is constant along the other black c.c.'s } \LC_i^\bullet \textnormal{ in the b.c. $\beta$} ; \\[.3em]  
\Im \phi_{\beta} \equiv 0 \;\textnormal{on }(x_1 \, x_2); \textnormal{ and } \\[.3em]  
\Im \phi_{\beta} \; \textnormal{ is constant along all white c.c.'s } \LC_j^\circ \textnormal{ in the b.c. $\beta$} .
\end{cases}
\end{align}
\end{proposition}

We emphasize that the convergence of the observable 
does not require any additional regularity of $\partial\Omega$: 
we only need that $\partial\Omega$ is locally connected. 
Moreover, we only require the convergence of polygons in the Carath\'{e}odory sense, 
that is weaker than that in Equation~\eqref{eqn::polygon_cvg}.
These two points are essential for the proof of Theorem~\ref{thm::ust_general}, since upon exploring discrete interfaces, we cannot guarantee much regularity for the boundaries of the domains thus obtained.

\begin{figure}[ht!]
\vspace*{-6mm}

\begin{subfigure}[b]{0.3\textwidth}
\begin{center}
\includegraphics[width=0.5\textwidth]{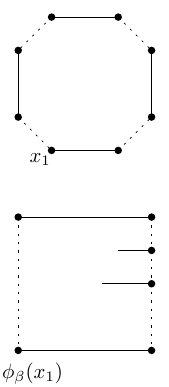}
\end{center}
\caption{\tiny{$\beta=\{\{1,2\}, \{3,4\}, \{5,6\}, \{7,8\}\}$.}}
\end{subfigure}
\begin{subfigure}[b]{0.3\textwidth}
\begin{center}
\includegraphics[width=0.5\textwidth]{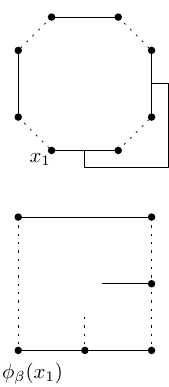}
\end{center}
\caption{\tiny{$\beta=\{\{1,4\},\{2,3\}, \{5,6\}, \{7,8\}\}$.}}
\end{subfigure}
\begin{subfigure}[b]{0.3\textwidth}
\begin{center}
\includegraphics[width=0.5\textwidth]{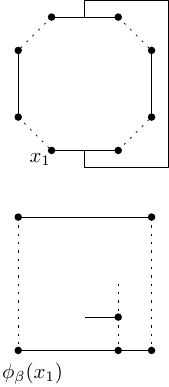}
\end{center}
\caption{\tiny{$\beta=\{\{1,6\}, \{2,5\}, \{3,4\}, \{7,8\}\}$.}}
\end{subfigure}\\

\vspace*{5mm}

\begin{subfigure}[b]{0.3\textwidth}
\begin{center}
\includegraphics[width=0.5\textwidth]{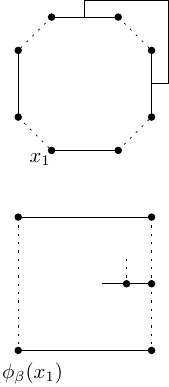}
\end{center}
\caption{\tiny{$\beta=\{\{1,2\}, \{3,6\}, \{4,5\}, \{7,8\}\}$.}}
\end{subfigure}
\begin{subfigure}[b]{0.3\textwidth}
\begin{center}
\includegraphics[width=0.5\textwidth]{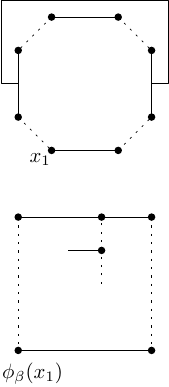}
\end{center}
\caption{\tiny{$\beta=\{\{1,2\},\{3,8\},\{4,7\},\{5,6\}\}$.}}
\end{subfigure}
\begin{subfigure}[b]{0.3\textwidth}
\begin{center}
\includegraphics[width=0.5\textwidth]{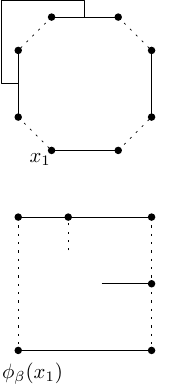}
\end{center}
\caption{\tiny{$\beta=\{\{1,2\},\{3,4\},\{5,8\}, \{6,7\}\}$.}}
\end{subfigure}\\

\vspace*{5mm}

\begin{subfigure}[b]{0.3\textwidth}
\begin{center}
\includegraphics[width=0.5\textwidth]{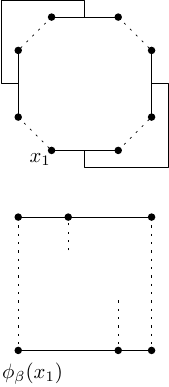}
\end{center}
\caption{\tiny{$\beta=\{\{1,4\}, \{2,3\}, \{5,8\}, \{6,7\}\}$.}}
\end{subfigure}
\begin{subfigure}[b]{0.3\textwidth}
\begin{center}
\includegraphics[width=0.5\textwidth]{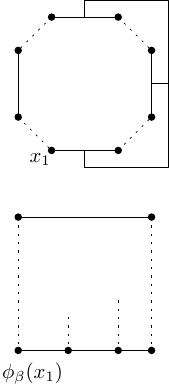}
\end{center}
\caption{\tiny{$\beta=\{\{1,6\}, \{2,3\}, \{4,5\}, \{7,8\}\}$.}}
\end{subfigure}
\begin{subfigure}[b]{0.3\textwidth}
\begin{center}
\includegraphics[width=0.5\textwidth]{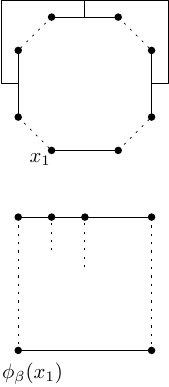}
\end{center}
\caption{\tiny{$\beta=\{\{1,2\}, \{3,8\}, \{4,5\}, \{6,7\}\}$.}}
\end{subfigure}

\caption{\label{fig::8points_observable}
For discrete polygons with eight marked points,
there are nine possibilities for planar boundary conditions $\beta$ such that $\{1,8\}\not\in\beta$. 
In each case, we draw the boundary conditions on the top, 
and the corresponding image of the conformal map 
$\phi_{\beta}$ in Proposition~\ref{prop::holo_cvg}
at the bottom.
}
\end{figure}

{
\begin{proof}[Proof of Proposition~\ref{prop::holo_cvg}]
The proof has the following three main steps. 
\begin{enumerate}
\item \label{step1}
The sequence 
$(\phi_{\beta}^{\delta})_{\delta>0}$
is uniformly bounded and its any subsequential limit under the locally uniform convergence is holomorphic on $\Omega$. 

\smallbreak

\item \label{step2}
Any subsequential limit $\psi$ 
has the boundary data~\eqref{eqn::boundarydata}. 

\smallbreak

\item \label{step3}
The boundary data~\eqref{eqn::boundarydata} uniquely determines $\psi=\phi_{\beta}$.
\end{enumerate}

\paragraph*{Step~\ref{step1}.}
We use standard discrete complex analysis arguments (see, e.g.~\cite[Section~8]{DuminilCopinParafermionic}).
If the sequence $(\phi_{\beta}^{\delta})_{\delta>0}$ is uniformly bounded, 
then it has a locally uniformly convergent subsequence.
The discrete holomorphicity from Lemma~\ref{lem::holo_general} implies that any discrete contour integral of any $\phi_{\beta}^{\delta}$ in $\Omega^{\delta, \diamond}$ vanishes,
and this property is inherited by any contour integral of any subsequential limit, which thus is holomorphic in $\Omega$ by Morera's theorem.

To complete Step~\ref{step1}, it remains to show the uniform boundedness of $\phi_{\beta}^{\delta}$, 
which by Lemma~\ref{lem::holo_general} 
is equivalent to the uniform boundedness of 
$h_\delta := |\sf_\beta^\delta| / |\st_\beta^\delta|$, for all $\delta>0$.
The latter fact can be argued via contradiction: 
if $h_{\delta_n} \to \infty$ along some sequence $\delta_n\to 0$,
then, as above, $h_{\delta_n}^{-1} \circ \phi_{\beta}^{\delta_n}$
converges locally uniformly to some holomorphic function $h$. 
Now, on the one hand the limit function $h$ satisfies $\Re h=0$ on $\Omega$, so $h$ has to be a constant, while on the other hand, 
from Lemma~\ref{lem::holo_general} and the discrete Beurling estimate we see that $\Im h$ cannot be a constant. 
This is the sought contradiction.

\paragraph*{Step~\ref{step2}.}
Lemma~\ref{lem::holo_general} and the discrete Beurling estimate imply that 
$\Re\psi$ (resp.~$\Im\psi$)
can be extended continuously to $\Omega \cup \bigcup_{r=1}^{N}(x_{2r} \, x_{2r+1})$ 
(resp.~to $\Omega \cup \bigcup_{r=1}^{N}(x_{2r-1} \, x_{2r})$) 
with~\eqref{eqn::boundarydata}.

\paragraph*{Step~\ref{step3}.}
Without loss of generality, we assume that $\Omega = \HH$ and $x_1<\cdots<x_{2N}$, and identify the c.c.'s $\LC_0^\bullet, \LC_1^\bullet, \ldots, \LC_n^\bullet$ and $\LC_1^\circ, \LC_2^\circ, \ldots, \LC_m^\circ$ with connected components of $\HH^* \setminus \bigcup_{r=1}^N \gamma^\beta_r$, that is, the lower half-plane with the branch cuts 
introduced in Section~\ref{subsec:hyperelliptic}. 
Consider the differential of $g := \psi_1 - \psi_2$, where $\psi_1,\psi_2$ are two subsequential limits. Note that 
\begin{align*}
\begin{cases}
\Re g \equiv 0 \;\textnormal{on }(x_{2N} \, x_1);  \\[.3em]  
\Re g \equiv  c^\bullet_i \; \textnormal{ along each black c.c. } \LC_i^\bullet \textnormal{ in the b.c. $\beta$}, \textnormal{ for } i \in \{1,2,\ldots,n(\beta)\} ; \\[.3em]  
\Im g \equiv 0 \;\textnormal{on }(x_1 \, x_2); \textnormal{ and } \\[.3em]  
\Im g \equiv c_j^\circ \; \textnormal{ along each white c.c. } \LC_j^\circ \textnormal{ in the b.c. $\beta$} , \textnormal{ for } j \in \{1,2,\ldots,m(\beta)\} ,
\end{cases}
\end{align*}
with some constants $c^\bullet_1, \ldots, c^\bullet_{n(\beta)}$ and $c_1^\circ, \ldots, c_{m(\beta)}^\circ$. 
The Schwarz reflection principle (see, e.g.,~\cite[Chapter~4, Section~6.5]{AhlforsComplexAnalysis}) shows that $g$ extends to 
a holomorphic function 
\begin{align*}
g(z) :=
\begin{cases}
g(z) , & z \in \overline{\HH} , \\
2c^\bullet_i- \overline{g(\bar{z})} , & z\in \LC_i^\bullet , \quad i \in \{0,1,2,\ldots,n(\beta)\} ,  \\ 
2\ii c^\circ_j+ \overline{g(\bar{z})} , & z\in \LC_j^\circ , \quad j \in \{1,2,\ldots,m(\beta)\} ,
\end{cases}
\quad 
\textnormal{for all } z \in (\C \cup \{\infty\}) \setminus \bigcup_{r=1}^N \gamma^\beta_r .
\end{align*}
In particular, $\ud g$ extends to a holomorphic differential on the Riemann surface $\Sigma= \Sigma_{x_1,\ldots,x_{2N}}$ associated to the hyperelliptic curve~\eqref{eq: hyperelliptic_curve}, 
defined as $\ud g$ on $\Sigma^+ = \C \cup \{\infty\}$ and as $-\ud g$ on $\Sigma^-$.
Expanding it in the basis~\eqref{eqn::parametera_polynomial} of holomorphic differentials on $\Sigma$, 
\begin{align} \label{eqn::differential_expan}
\ud g = \; & \sum_{s=1}^{N-1} \upsilon_s \, \omega_{s-1} , \qquad 
\boldsymbol{\upsilon} = (\upsilon_1, \ldots, \upsilon_{N-1}) \in \C^{N-1} ,
\end{align}
and integrating both sides of~\eqref{eqn::differential_expan} along $\gamma^\beta_2,\ldots,\gamma^\beta_N$ respectively, 
we see that $A_\beta \boldsymbol{\upsilon}^t = \boldsymbol{0}$. Since $A_\beta$ is invertible by Lemma~\ref{lem: matrices_invertible}, we obtain 
$\upsilon_1 = \cdots = \upsilon_{N-1} = 0$, so $g$ is the constant function. Because $\Re g=0$ on $(x_{2N},x_1)$ and $\Im g=0$ on $(x_1,x_2)$, we see that $g$ equals zero.  
\end{proof}
}

The holomorphic function $\phi_{\beta}$ is, in fact, a conformal map from the polygon $(\Omega;x_1,\ldots,x_{2N})$ to a certain slit rectangle depending on $\beta$. 
In particular, for $\Omega = \HH$, the map $\phi_{\beta}$ has an explicit (degenerate) Schwarz-Christoffel formula, see Equation~\eqref{eq: Schwarz-Christoffel formula app}.
Geometrically, if $\{1,b\}, \{a,2N\} \in \beta$, then 
$\phi_{\beta}$ is the unique conformal map from $\HH$ onto a rectangle of unit width with horizontal and vertical slits such that 
$\phi_{\beta}$ maps the four points $(x_1, x_{b}, x_{a}, x_{2N})$ 
to the four corners of the rectangle so that $\phi_{\beta}(x_1)=0$, 
and $\phi_{\beta}$ maps c.c.'s of 
$\{(x_{2r} \, x_{2r+1}) \colon 1\le r < N \}$
to vertical slits and c.c.'s of 
$\{(x_{2r-1} \, x_{2r}) \colon 1 < r <  N\}$ 
to horizontal slits, according to the boundary 
data~\eqref{eqn::boundarydata}. 
See Figure~\ref{fig::slitrectangle} and~\ref{fig::8points_observable} for some illustrations.

\subsection{Expansion near a marked point}
\label{subsec::expansion}

To derive the scaling limit of the Peano curve (Theorem~\ref{thm::ust_general} in Section~\ref{subsec::ust_Peano_conv}), we need  
to analyze the asymptotics of the function $\phi_{\beta}(z)$ as $z$ approaches one of the marked points (Lemma~\ref{lem::conformal_general_expansion}).
In particular, we will explicitly relate it to the partition function $\LF_{\beta}$
(Lemma~\ref{lem::ust_general_goal_new}). 
For definiteness and without loss of generality 
(by the full M\"obius covariance 
from Proposition~\ref{prop: full Mobius covariance F}
and duality for the UST model), 
we consider the case where $\{2N-1, 2N\}\in\beta$ and we study the limit of $\phi_{\beta}(z)$ as $z \to x_1$.

\begin{lemma}\label{lem::conformal_general_expansion}
Write $\boldsymbol{x} = (x_1, \ldots, x_{2N}) \in \chamber_{2N}$ and $\boldsymbol{\dot{x}}_1 = (x_2, \ldots, x_{2N})$.
On $\Omega = \HH$, the holomorphic function $\phi_{\beta}$ with boundary data~\eqref{eqn::boundarydata} 
\textnormal{(}cf.~Proposition~\ref{prop::holo_cvg}\textnormal{)} 
has the following expansion\textnormal{:} 
\begin{align} \label{eqn::conformal_general_expansion_new}
\phi_{\beta}(z; \boldsymbol{x}) = 
\LH_{\beta}(\boldsymbol{x}) \, (z-x_1)^{1/2} 
+ \LK_{\beta}(\boldsymbol{x}) \, (z-x_1)^{3/2} 
+ o((z-x_1)^{3/2}) , \quad \textnormal{as }z\to x_1, 
\end{align}
where 
\begin{align} 
\label{eqn::conformal_general_H}
\LH_{\beta}(\boldsymbol{x}) 
= \; & \frac{2 \, (-1)^{N-1} }{\wfunc(x_1; \boldsymbol{\dot{x}}_1)}
\, \frac{\det C_{\beta} (x_1; \boldsymbol{x})}{\det A_{\beta}(\boldsymbol{x})} , \\
\label{eqn::conformal_general_K}
\LK_{\beta}(\boldsymbol{x})
= \; & \frac{(-1)^{N-1}}{\wfunc(x_1; \boldsymbol{\dot{x}}_1)} 
\frac{\det C_{\beta} (x_1; \boldsymbol{x})}{\det A_{\beta}(\boldsymbol{x})}
\bigg( \sum_{j=2}^{2N} \frac{1/3}{x_j-x_1}
+ \frac{2}{3} \, \frac{\partial_{1} \det C_{\beta} (x_1; \boldsymbol{x})}{\det C_{\beta} (x_1; \boldsymbol{x})} \bigg) ,
\end{align}
with $\wfunc(u; \boldsymbol{x})$ the meromorphic function~\eqref{eqn::def_hypercurve} on the Riemann surface 
$\Sigma = \Sigma_{x_1, \ldots, x_{2N}}$, and $A_{\beta}$ the matrix in~\eqref{eq::A_matrices}, 
and $C_{\beta}$ the $(N-1)\times(N-1)$ matrix with entries
\begin{align*} 
(C_{\beta})_{r,s}(u; \boldsymbol{x}) := 
\begin{cases}
(A_{\beta})_{r,s}(\boldsymbol{x}) 
, & 
r \in \{1,2,\ldots,N-2\}, \; s \in \{1,2,\ldots,N-1\} , \\
u^{s-1}, & r = N-1 , \; s \in \{1,2,\ldots,N-1\} .
\end{cases}
\end{align*}
\end{lemma}

Both determinants can be evaluated in terms of the Vandermonde determinant: $\det A_{\beta}(\boldsymbol{x})$ is given by the line integral analogue of~\eqref{eq: VandermondeA} (cf.~Lemma~\ref{lem: relation_of_Pcirc_and_P_matrix} 
in Appendix~\ref{app::matrices}), and 
\begin{align} \label{eq: VandermondeC}
\begin{split} 
\det C_{\beta} (u; \boldsymbol{x})
= \; & (-1)^{N} \landupint_{x_{a_{2}}}^{x_{b_{2}}} \ud u_2 
\cdots \landupint_{x_{a_{N-1}}}^{x_{b_{N-1}}} \ud u_{N-1} 
\; \frac{\Delta (\boldsymbol{\ddot{u}}_{1,N})}{\wfunc(\boldsymbol{\ddot{u}}_{1,N}; \boldsymbol{x})} \; \prod_{s=2}^{N-1} (u_{s} - u) ,
\end{split}
\end{align}
where we write 
$\boldsymbol{\ddot{u}}_{1,N} := (u_2,\ldots,u_{N-1})$, 
and $\Delta$ is defined in~\eqref{eqn::def_vander}.

\begin{proof}
Expanding $\ud \phi_{\beta}$ in the basis~\eqref{eqn::parametera_polynomial} of holomorphic differentials on $\Sigma$ 
and integrating, 
we see that the holomorphic function 
$\phi_{\beta}$ has the form
\begin{align} \label{eq: Schwarz-Christoffel formula tilde}
\phi_{\beta}(z; \boldsymbol{x}) 
= \landupint_{x_1}^z \frac{Q_{\beta}(u) \, \ud u}{\wfunc(u; \boldsymbol{x})} , \qquad z \in \overline{\HH} ,
\end{align} 
uniquely\footnote{Note that $Q_{\beta}$ is a polynomial of degree at most $N-2$ (in fact, exactly $N-2$ by~\eqref{eq: Schwarz-Christoffel formula app} in Proposition~\ref{prop: SC map}).} 
determined by the boundary data~\eqref{eqn::boundarydata}:
\begin{align} 
\label{eq: bdrydata_for_SC-1}
\landupint_{x_1}^{x_{b_1}} \frac{Q_{\beta}(u) \, \ud u}{\wfunc(u; \boldsymbol{x})} = \; & 1 = - \landupint_{x_{2N-1}}^{x_{2N}} \frac{Q_{\beta}(u) \, \ud u}{\wfunc(u; \boldsymbol{x})}  , \\
\label{eq: bdrydata_for_SC-2}
\landupint_{x_{a_{r+1}}}^{x_{b_{r+1}}} \frac{Q_{\beta}(u) \, \ud u}{\wfunc(u; \boldsymbol{x})} = \; & 0 , \qquad r \in \{1,2,\ldots,N-2\} . 
\end{align}
Indeed, as the matrix $A_{\beta}$ is invertible by Lemma~\ref{lem: matrices_invertible}, 
these equations have a unique solution 
$\boldsymbol{\nu} = (\nu_0,\nu_1,\ldots,\nu_{N-2})$, 
which determine the coefficients of $Q_{\beta} = \sum_{s=0}^{N-2} \nu_s u^s$:
by~(\ref{eq: bdrydata_for_SC-1},~\ref{eq: bdrydata_for_SC-2}), we have
\begin{align*}
A_{\beta} \boldsymbol{\nu}^t = (0,0,\ldots,0,0,-1)^t .
\end{align*}
By Cramer's rule, 
we have $(A_{\beta})^{-1} = \frac{1}{\det A_{\beta}} \, \adj{A_{\beta}}$, where $\adj{A_{\beta}}$ is the adjugate matrix of $A_{\beta}$. Hence,
we obtain from Cramer's rule the relation
\begin{align*} 
Q_{\beta}(u; \boldsymbol{x}) 
= (1, u, \ldots, u^{N-2}) \, \boldsymbol{\nu}^t 
= \; & (1, u, \ldots, u^{N-2}) \,
\frac{\adj{A_{\beta}(\boldsymbol{x})}}{\det A_{\beta}(\boldsymbol{x})} \, (0,0,\ldots,0,0,-1)^t \\
= \; & (-1)^{N-1} \frac{\det C_{\beta}(u; \boldsymbol{x}) }{\det A_{\beta}(\boldsymbol{x})} 
\end{align*}
where $\boldsymbol{\nu} = (\nu_0,\nu_1,\ldots,\nu_{N-2})$ 
are the coefficients of $Q_{\beta} = \sum_{s=0}^{N-2} \nu_s u^s$. 
Now, we obtain the desired expansion~\eqref{eqn::conformal_general_expansion_new} by a direct computation. 
First, from~(\ref{eq: VandermondeC},~\ref{eq: Schwarz-Christoffel formula tilde}) we have
\begin{align} \label{eq: observable_limit_detA-DetC} 
\phi_{\beta}(z) 
= \; & \frac{-1}{\det A_{\beta}(\boldsymbol{x})} 
\landupint_{x_1}^z \ud u_1 
\landupint_{x_{a_{2}}}^{x_{b_{2}}} \ud u_2 
\cdots \landupint_{x_{a_{N-1}}}^{x_{b_{N-1}}} \ud u_{N-1} 
\; \frac{\Delta (\boldsymbol{\dot{u}}_{N})}{\wfunc(\boldsymbol{\dot{u}}_{N}; \boldsymbol{x})} ,
\end{align}
writing $\boldsymbol{\dot{u}}_{N} := (u_1,u_2,\ldots,u_{N-1})$.  
Next, making the change of variables
$v = \frac{u_1 - x_1}{z - x_1}$, we obtain 
\begin{align*} 
\phi_{\beta}(z) 
= \; & \frac{-(z - x_1)^{1/2}}{\det A_{\beta}(\boldsymbol{x})} 
 \landupint_0^1 v^{-1/2} \ud v 
\landupint_{x_{a_{2}}}^{x_{b_{2}}} \ud u_2 
\cdots \landupint_{x_{a_{N-1}}}^{x_{b_{N-1}}} \ud u_{N-1}  
\frac{\Delta (\boldsymbol{\ddot{u}}_{1,N})}{\wfunc(\boldsymbol{\ddot{u}}_{1,N}; \boldsymbol{x})}
\\
\; & \qquad \qquad \qquad \qquad \qquad \times
\; \frac{1}{\wfunc((z - x_1) v + x_1; \boldsymbol{\dot{x}}_1)} 
\; \prod_{s=2}^{N-1} (u_{s} - (z - x_1) v - x_1).
\end{align*}
Now, we have
\begin{align*}
\frac{1}{\wfunc((z - x_1) v + x_1; \boldsymbol{\dot{x}}_1)}
= \; & \frac{1}{\wfunc(x_1; \boldsymbol{\dot{x}}_1)}
\bigg( 1 + \frac{v}{2} \sum_{j=2}^{2N} \frac{z - x_1}{x_j-x_1}
+ o(z - x_1) \bigg) , \\
\prod_{s=2}^{N-1} (u_{s} - (z - x_1) v - x_1)
=  \; & 
\bigg( \prod_{s=2}^{N-1} (u_{s} - x_1)
\bigg) 
\bigg( 1 + v \sum_{s=2}^{N-1} \frac{z - x_1}{u_{s} - x_1} + o(z - x_1) \bigg) .
\end{align*}
Collecting the terms according to different powers of $(z - x_1)$, we obtain~\eqref{eqn::conformal_general_expansion_new}.
\end{proof}

\begin{lemma}\label{lem::ust_general_goal_new}
We have
\begin{align} \label{eqn::ust_polygon_goal_PDE_general_new}
\partial_1\log \LF_{\beta}(\boldsymbol{x}) 
= \frac{3\LK_{\beta}(\boldsymbol{x})-2\partial_1\LH_{\beta}(\boldsymbol{x})}{2\LH_{\beta}(\boldsymbol{x})},  \qquad \boldsymbol{x} \in \chamber_{2N} ,
\end{align}
where $\LF_{\beta}$ is defined in~\eqref{eq: Def of F beta}, 
and $\LH_{\beta}$ and $\LK_{\beta}$ are defined in~\eqref{eqn::conformal_general_H} and~\eqref{eqn::conformal_general_K}. 
\end{lemma}

\begin{proof}
By Proposition~\ref{prop::two_Fs_as_determinants}, we have
\begin{align*}
\partial_1\log \LF_{\beta}(\boldsymbol{x}) 
= \frac{\partial_1\det{A_{\beta}}(\boldsymbol{x})}{\det{A_{\beta}}(\boldsymbol{x})} 
- \frac{1}{4}\sum_{j=2}^{2N}\frac{1}{x_j-x_1} .
\end{align*}
Note that
\begin{align*}
\frac{3\LK_{\beta}(\boldsymbol{x})}{2\LH_{\beta}(\boldsymbol{x})} 
= \; & \frac{1}{4}\sum_{j=2}^{2N}\frac{1}{x_j-x_1} 
+ \frac{\partial_{1} \det C_{\beta} (x_1; \boldsymbol{x})}{\det C_{\beta} (x_1; \boldsymbol{x})} ,
\end{align*}
and
\begin{align*}
\frac{\partial_1\LH_\beta(\boldsymbol{x})}{\LH_\beta(\boldsymbol{x})}
= \; & \frac{\partial_{1} \det C_{\beta} (x_1; \boldsymbol{x})}{\det C_{\beta} (x_1; \boldsymbol{x})}
- \frac{\partial_{1} \det A_{\beta}(\boldsymbol{x})}{\det A_{\beta}(\boldsymbol{x})}
+ \frac{1}{2}\sum_{j=2}^{2N}\frac{1}{x_j-x_1} .
\end{align*}
This proves~\eqref{eqn::ust_polygon_goal_PDE_general_new}. 
\end{proof}

\subsection{Scaling limits of Peano curves --- proof of Theorem~\ref{thm::ust_general}}
\label{subsec::ust_Peano_conv}

We will next complete the proof of Theorem~\ref{thm::ust_general} by deriving the limit of the law of the Peano curve $\eta_1^{\delta}$ in the case where $\{1, 2N\}\not\in\beta$ (recall that the other cases can be treated via duality and rotation symmetry). 

\subsubsection*{Loewner chains}
Let us first collect some preliminaries on Loewner chains ---
see~\cite{LawlerConformallyInvariantProcesses} for background. 
Suppose that a continuous function $W_t \colon [0,\infty) \to \R$, 
called the \emph{driving function}, is given.
Consider solutions $(g_{t}, t\ge 0)$ to the \emph{Loewner equation} 
\begin{align} \label{eqn:LE}
\partial_{t}{g}_{t}(z) 
= \frac{2}{g_{t}(z)-W_{t}} 
\qquad \textnormal{with initial condition} \qquad g_{0}(z)=z .
\end{align}
For each $z \in \HH$, the ordinary differential equation~\eqref{eqn:LE} has a unique solution up to 
\begin{align*}
T_{z} := \sup\{t\ge 0 \colon \min_{s\in[0,t]} |g_{s}(z)-W_{s}|>0\} ,
\end{align*}
the \emph{swallowing time} of $z$. 
As a function of $z$, each map $g_t$ is a well-defined conformal transformation
from $H_{t}:=\HH\setminus K_{t}$ onto $\HH$, where the hull of swallowed points is $K_{t}:=\overline{\{z\in\HH \colon T_{z}\le t\}}$.
The collection $(K_{t}, t\ge 0)$ of hulls growing in time is called a \emph{Loewner chain}.
In fact, $g_t \colon H_{t} \to \HH$ is the unique conformal map 
such that $|g_K(z)-z| \to 0$ as $z \to \infty$.
The \emph{half-plane capacity} $\mathrm{hcap}(K_{t})$ of the hull $K_{t}$ is defined as the coefficient of $z^{-1}$ in the series expansion of $g_t$ at infinity, and~\eqref{eqn:LE} 
implies that $\mathrm{hcap}(K_{t}) = 2 t$. 
We say that the process $(K_t, t\ge 0)$ is \emph{parameterized by the half-plane capacity}.

\subsubsection*{SLE processes}
Chordal \emph{Schramm-Loewner evolution}, $\SLE_{\kappa}$, 
is the random Loewner chain driven by $W = \sqrt{\kappa} \, B$,
a standard one-dimensional Brownian motion $B$ of speed $\kappa \geq 0$. 
See~\cite{LawlerConformallyInvariantProcesses, RohdeSchrammSLEBasicProperty} for background and further properties of this process. In this article, we assume that $\kappa=8$.

Recall that a \emph{partition function} (with $\kappa=8$) refers to a positive smooth function $\PartF \colon \chamber_{2N} \to\R_{>0}$ satisfying the PDE system~\eqref{eqn::USTPDE} and M\"obius covariance~\eqref{eqn::USTCOV}. 
We can use any partition function to define a 
\emph{Loewner chain associated to} $\PartF$: in the upper half-plane $\HH$, started from $x_i \in \R$, 
and with marked points $(x_1, \ldots, x_{i-1}, x_{i+1}, \ldots, x_{2N})$,
that is, the Loewner chain driven by 
the solution $W$ to the SDEs~\eqref{eqn::driving_general}. 
This process is well-defined up to the first time when either $x_{i-1}$ or $x_{i+1}$ is swallowed,  
and each $V_t^j = g_t(x_j)$ is the time-evolution of $x_j$ for times $t$ smaller than the swallowing time of~$x_j$.

\smallbreak

Now are ready to complete the proof of Theorem~\ref{thm::ust_general}. 
The key is to identify the driving process of the scaling limit curve as the solution to the SDEs~\eqref{eqn::driving_general} with $\PartF = \LF_\beta$ and $i=1$.

\begin{proof}[Proof of Theorem~\ref{thm::ust_general}]
By assumption, the medial polygons $(\Omega^{\delta, \diamond}; x_1^{\delta, \diamond}, \ldots, x_{2N}^{\delta, \diamond})$ converge to the polygon $(\Omega; x_1, \ldots, x_{2N})$ in the sense of Equation~\eqref{eqn::polygon_cvg}, 
so they also converge in the Carath\'{e}odory sense: there exist conformal maps $\varphi_{\delta} \colon \Omega^{\delta, \diamond} \to \HH$ and $\varphi \colon \Omega \to \HH$
such that $\varphi(x_1)<\cdots<\varphi(x_{2N})$ and,
as $\delta \to 0$, the maps $\varphi_{\delta}^{-1}$ converge to $\varphi^{-1}$ locally uniformly on $\HH$, and $\smash{\varphi_{\delta}(x_j^{\delta, \diamond}) \to \varphi(x_j)}$ for all $1\le j\le 2N$.

As before, consider the UST on the primal polygon $(\Omega^{\delta}; x_1^{\delta}, \ldots, x_{2N}^{\delta})$ 
with boundary condition $\beta$ such that $\{1,2N\}\not\in\beta$, 
and let $\eta_1^{\delta}$ be the Peano curve started from $\smash{x_{1}^{\delta, \diamond}}$. 
Denote by $\tilde{\eta}_1^{\delta} := \varphi_{\delta}(\eta_1^{\delta})$ its conformal image parameterized by half-plane capacity. 
By Lemma~\ref{lem::Peanocurve_tight}, we may choose a subsequence $\delta_{n}\to 0$ such that $\eta^{\delta_{n}}_1$ converges weakly in the 
space~\eqref{eqn::metric_curvesspace} 
as $n\to\infty$. 
We denote the limit by $\eta_1$, define $\tilde{\eta}_1:=\varphi(\eta_1)$ and parameterize it also by half-plane capacity. 
By Lemma~\ref{lem::Peanocurve_tight} and a similar argument as in~\cite[Corollary~4.11]{HanLiuWuUST}, the family $\{\tilde{\eta}^{\delta_{n}}_1|_{[0,t]} \colon [0,t] \to \overline{\HH} \}_{n \geq 1}$ is 
precompact in the uniform topology of curves parameterized by half-plane capacity. 
Thus, using the diagonal method and the Skorohod representation theorem, we may choose a subsequence, 
still denoted by $\delta_n$, such that $\tilde{\eta}^{\delta_{n}}_1$ converges to $\tilde{\eta}_1$ locally uniformly as $n\to\infty$, almost surely.

Next, we define $\tau^{\delta_n}$ to be the first time when $\eta_1^{\delta_n}$ hits $(x_2^{\delta_n} \, x_{2N}^{\delta_n})$ 
and $\tau$ the first time when $\eta_1$ hits $(x_2 \, x_{2N})$. 
By properly adjusting the coupling (as in~\cite[Proof of Theorem~4.2, Equation~(4.14)]{HanLiuWuUST};
see also~\cite{KarrilaUSTBranches, GarbanWuFKIsing}), 
we may assume that $\underset{n \to \infty}{\lim} \tau^{\delta_n} = \tau$ almost surely.

Now, denote by $(W_t, t\ge 0)$ the driving function of $\tilde{\eta}_1$ and by $(g_t, t\ge 0)$ the corresponding conformal maps. 
Write $V_t^j := g_t(\varphi(x_j))$ for $j\in\{2, \ldots, 2N\}$. 
Via a standard argument (see, e.g.,~\cite[Lemma~4.8 and Lemma~4.9]{HanLiuWuUST}), 
we derive from the holomorphic observable $\phi_{\beta}(z;\boldsymbol{x})$ of Proposition~\ref{prop::holo_cvg}
the local martingale 
\begin{align} \label{eqn::ust_polygon_mart}
M_t = M_t(z;\boldsymbol{x}) 
:= \phi_{\beta}(g_t(z); W_t, V_t^2, \ldots, V_t^{2N}) , \qquad t < \tau .
\end{align}
It remains to argue that $(W_t, t\ge 0)$ is a semimartingale and find the SDE for it. 
This step is also standard, see, e.g.,~\cite[Proof of Lemma~4.12]{HanLiuWuUST} 
(a similar argument also appeared in~\cite[Lemma~5.3]{KarrilaUSTBranches}). 
For any $w<y_2<\cdots<y_{2N}$, the function $\partial_w \phi_\beta(\cdot;w,y_2,\ldots,y_{2N})$ 
is holomorphic and not identically zero, so its zeros are isolated. 
Pick $z\in\HH$ with $|z|$ large enough such that $\partial_w\phi_\beta(z;w,y_2,\ldots,y_{2N})\neq 0$.
By the implicit function theorem, $w$ is locally a smooth function of $(\phi_\beta,z,y_2,\ldots,y_{2N})$. 
Thus, by continuity, each time $t < \tau$ has a neighborhood $I_t$ for which
we can choose a deterministic $z$ such that $W_s$ is locally a smooth function of 
$(M_s(z),g_s(z),g_s(y_2),\ldots,g_s(y_{2N}))$ for $s \in I_t$.
This implies that $(W_t, t\ge 0)$ is a semimartingale. 
To find the SDE for it, let $D_t$ denote the drift term of $W_t$. 
By a computation using It\^o's formula, we find from~\eqref{eqn::ust_polygon_mart} and using the Loewner equation~\eqref{eqn:LE}  the identities
\begin{align*}
\ud M_t 
= (\partial_z\phi_{\beta}) \frac{2 \, \ud t}{g_t(z)-W_t} 
+ (\partial_1 \phi_{\beta}) \, \ud W_t 
+ \sum_{j=2}^{2N} (\partial_{x_j} \phi_{\beta}) \frac{2 \, \ud t}{V_t^j-W_t} 
+ \frac{1}{2} (\partial^2_{x_1} \phi_{\beta}) \, \ud \langle W\rangle_t . 
\end{align*}
Combining this with the explicit relation~\eqref{eqn::conformal_general_expansion_new} from Lemma~\ref{lem::conformal_general_expansion}, we find the expansion 
\begin{align*}
\ud M_t 
= & \;\;(g_t(z)-W_t)^{-3/2} \, 
\LH_{\beta} \big( \ud t - \tfrac{1}{8}\ud \langle W\rangle_t \big) \\
& \; + (g_t(z)-W_t)^{-1/2} \, 
\Big( 3\LK_{\beta} \, \ud t - \tfrac{1}{2}\LH_{\beta} \,\ud W_t 
+ \big( \tfrac{3}{8}\LK_{\beta} - \tfrac{1}{2}\partial_1 \LH_{\beta}\big) \ud \langle W \rangle_t \Big) \\
& \; + o(g_t(z)-W_t)^{-1/2} ,\qquad\textnormal{as }z\to \tilde{\eta}_1(t). 
\end{align*}
As the drift term of $M_t$ has to vanish, we conclude that 
\begin{align}
\begin{split}
\label{eqn::ust_polygon_aux}
\ud \langle W\rangle_t = \; & 8  \, \ud t 
\qquad  \textnormal{and} \qquad
3\LK_{\beta} \, \ud t - \tfrac{1}{2}\LH_{\beta} \, \ud D_t + \big( \tfrac{3}{8}\LK_{\beta} -\tfrac{1}{2}\partial_1 \LH_{\beta} \big)\ud \langle W\rangle_t = 0 \\
\qquad \Longrightarrow \qquad
\ud \langle W\rangle_t = \; & 8 \, \ud t
\qquad  \textnormal{and} \qquad
\ud D_t = \frac{12\LK_{\beta} - 8\partial_1 \LH_{\beta}}{\LH_{\beta}} \, \ud t . 
\end{split}
\end{align}
Recall now that the goal is to derive for the driving process $W$ the SDE
\begin{align} \label{eqn::ust_polygon_goal}
\ud W_t = \sqrt{8} \, \ud B_t + 8 (\partial_1 \log\LF_{\beta})(W_t, V_t^2, \ldots, V_t^{2N})\, \ud t , \qquad t < \tau . 
\end{align}
Indeed, by comparing~\eqref{eqn::ust_polygon_aux} and~\eqref{eqn::ust_polygon_goal} 
with Lemma~\ref{lem::ust_general_goal_new}, we see that the driving function $W$ of  the curve $\tilde{\eta}_1$ 
satisfies~\eqref{eqn::ust_polygon_goal} up to the stopping time $\tau$. 
This completes the proof. 
\end{proof}
\subsection{Consequences}
\label{subsec::consequences}

We conclude with another proof for the PDE system~\eqref{eqn::USTCOV} using the convergence result from Theorem~\ref{thm::ust_general} (see Corollary~\ref{cor::coulombgasintegralPDE}), 
and then comment briefly on the relation of our work to that of Dub\'edat~\cite{DubedatCommutationSLE, DubedatEulerIntegralsCommutingSLEs}.

\begin{corollary} \label{cor::coulombgasintegralPDE}
For each $\beta \in \LP_N$, the function $\LF_{\beta}$ satisfies the PDE system~\eqref{eqn::USTPDE}.
\end{corollary}

Note that we already know that the function $\LF_{\beta}$ is smooth, by its explicit formula~\eqref{eqn::coulombgasintegral}.
\begin{proof}
The PDEs follow from the commutation relations for SLEs~\cite{DubedatCommutationSLE}.
Take any disjoint localization neighborhoods $U_1, \ldots, U_{2N} \subset \overline{\Omega}$ of the marked points (cf.~\cite[Appendix~A]{KytolaPeltolaPurePartitionFunctions}), 
and approximate them by $U_i^{\delta, \diamond}$ on the medial lattice.
For each $i$, let $\eta_i^{\delta}$ be the Peano curve started from $x_i^{\delta, \diamond}$ and stopped when it exits $U_i^{\delta, \diamond}$. 
Thanks to the precompactness (Lemma~\ref{lem::Peanocurve_tight}), 
we find a sequence $\delta_n \to 0$ such that the collection $\{\eta_i^{\delta_n} \colon 1\le i\le 2N \}$ converges weakly to 
some collection $\{\eta_i \colon 1\le i \le 2N \}$ of curves on $\Omega$, with 
law $\mathrm{P}^{(\Omega; x_1, \ldots, x_{2N})}_{(U_1, \ldots, U_{2N})}$.  
It follows that this family of probability measures has the following properties.
\begin{itemize}[leftmargin=2em]
\item \emph{Conformal invariance.} 
For any conformal map $\varphi \colon \Omega \to \varphi(\Omega)$, the law of the curve collection $\{\varphi(\eta_i) \colon 1\le i\le 2N \}$
is the same as $\mathrm{P}^{(\varphi(\Omega); \varphi(x_1), \ldots, \varphi(x_{2N}))}_{(\varphi(U_1), \ldots, \varphi(U_{2N}))}$. 

\smallbreak

\item \emph{Domain Markov property.} 
Given initial segments $\{\eta_i(t) \colon 1\le i\le 2N , \; 0\le t\le \tau_i \}$ up to some stopping times $\tau_i$,
the remaining parts 
$\{\eta_i(t) \colon 1\le i\le 2N , \; t \geq \tau_i \}$
have the law $\mathrm{P}^{(\Omega'; x_1', \ldots, x_{2N}')}_{(U_1', \ldots, U_{2N}')}$, 
where 
\begin{align*}
\Omega'=\Omega\setminus\bigcup_{i=1}^{2N}\eta_i[0,\tau_i] , 
\qquad \textnormal{and} \qquad
x_i'=\eta_i(\tau_i)
\quad \textnormal{and} \quad
U_i'=U_i\cap \Omega'
\textnormal{ for each }
i \in \{1,2,\ldots,2N\} .
\end{align*}

\smallbreak

\item For each $i$, the marginal law of $\eta_i$ is the same as the law of the Loewner chain 
with driving function~\eqref{eqn::ust_polygon_driving_general} started from $x_i$ and stopped when it exits $U_i$. 
\end{itemize}
It thus follows from~\cite[Theorem~7]{DubedatCommutationSLE} 
that $\LF_{\beta}$ satisfies the asserted PDEs~\eqref{eqn::USTPDE}.
\end{proof}

\begin{remark}
With $\beta=\unnested$ in Theorem~\ref{thm::ust_general}, 
the curve $\eta_i^{\delta}$ converges to the image under $\varphi^{-1} \colon \HH \to \Omega$
of the Loewner chain associated to the partition function $\LF_{\unnested}$. 
In contrast, Dub\'edat argued 
in~\textnormal{\cite[Section~3.3]{DubedatEulerIntegralsCommutingSLEs}}
the same conclusion but with 
$\LF_{\unnested}$ replaced by 
\begin{align*} 
\int_{x_1}^{x_2}\ud u_1\cdots\int_{x_{2N-3}}^{x_{2N-2}}\ud u_{N-1} \, f(\boldsymbol{x}; \boldsymbol{\dot{u}}_N) , 
\qquad \boldsymbol{\dot{u}}_N := (u_1, \ldots, u_{N-1}) .
\end{align*}
We see from Proposition~\ref{prop::two_Fs_as_determinants} 
that our result is consistent with~\textnormal{\cite[Section~3.3]{DubedatEulerIntegralsCommutingSLEs}}. 
\end{remark}


\smallskip{}

\section{Crossing probabilities, groves, and pure partition functions}
\label{sec::ppf}

In this section, we prove the explicit formulas for the scaling limits of crossing probabilities of the UST Peano curves (Theorem~\ref{thm::ust_crossing_proba} in Section~\ref{subsec::ust_crossingproba}). 
In the cases $N =1$ and $N=2$, the crossing probability formulas are evident, as they are just given by the meander entries $\LM_{\alpha,\beta}$ already in the discrete model.
Explicitly, we have
\begin{align*}
\PartF_{\vcenter{\hbox{\includegraphics[scale=0.2]{figures/link-0.pdf}}}} =\LF_{\vcenter{\hbox{\includegraphics[scale=0.2]{figures/link-0.pdf}}}} 
\quad \textnormal{when } N=1 ; \qquad \textnormal{and}\qquad 
\begin{cases}
\PartF_{\vcenter{\hbox{\includegraphics[scale=0.2]{figures/link-1.pdf}}}}=\LF_{\vcenter{\hbox{\includegraphics[scale=0.2]{figures/link-2.pdf}}}} ,\\ 
\PartF_{\vcenter{\hbox{\includegraphics[scale=0.2]{figures/link-2.pdf}}}}=\LF_{\vcenter{\hbox{\includegraphics[scale=0.2]{figures/link-1.pdf}}}}
\end{cases}
\quad \textnormal{when } N=2.
\end{align*}

The probabilities in Theorem~\ref{thm::ust_crossing_proba} become interesting when $N\ge 3$, still retaining some combinatorial structure.
Indeed, Kenyon and Wilson~\cite{KenyonWilsonBoundaryPartitionsTreesDimers} provided a systematic method for calculating the discrete crossing probabilities for groves, including spanning trees. 

\subsubsection*{Groves}
Consider the discrete polygon $(\Omega^{\delta}; x_1^{\delta}, \ldots, x_{2N}^{\delta})$ whose each boundary arc $(x_{2r-1}^{\delta} \, x_{2r}^{\delta})$ with $1\le r\le N$ is wired and these $N$ wired arcs are isolated outside of $\Omega^{\delta}$. 
Following~\cite{KenyonWilsonBoundaryPartitionsTreesDimers},
we define a \emph{grove} to be a spanning forest of $\Omega^{\delta}$ such that each of its component trees contains at least one wired arc. 
Note that every grove induces a non-crossing partition on the $N$ wired arcs, 
which we view as a non-crossing partition $\pi$ of $\{1,2,\ldots, N\}$.
Every grove also induces $N$ non-intersecting Peano curves connecting the points $\{x_1^{\delta, \diamond}, \ldots, x_{2N}^{\delta, \diamond}\}$ pairwise. 
The endpoints of these $N$ Peano curves form a planar link pattern $\alpha \in \LP_N$, 
and these two are in bijection 
(via the correspondence illustrated in Figure~\ref{fig::8pointsE_meander}):
we write $\alpha \leftrightarrow \pi(\alpha)$.

Crossing probabilities of the UST Peano curves can be expressed in terms of more general crossing probabilities for groves. 
We consider a random grove $\LG$ sampled from the uniform distribution $\rm{Pr}^{\delta}$ of all groves on $\Omega^{\delta}$ and denote the random induced non-crossing partition by $\grconn^{\delta}$.
Then, 
\cite[Theorem~1.2]{KenyonWilsonBoundaryPartitionsTreesDimers} gives the following probabilities: 
\begin{align} \label{eq: grove crossing general}
\frac{\rm{Pr}^\delta[\grconn^{\delta}=\pi]}{\rm{Pr}^\delta[\LG \textnormal{ has $N$ components}]} =\sum_{\varpi}P_{\pi,\varpi} \, L^\delta_\varpi 
\quad \quad \textnormal{for any non-crossing partition $\pi$ of $\{1,2,\ldots, N\}$},
\end{align}
where the sum is taken over all (possibly crossing) partitions $\varpi$ of $\{1,2, \ldots, N\}$, 
the factor $P_{\pi,\varpi}$ is a constant which only depends on $\pi$ and $\varpi$, and the combinatorial factor $L^\delta_\varpi$ is 
\begin{align*}
L^\delta_\varpi := \sum_{\forest(\varpi)} \prod_{\langle n, r\rangle \in \forest(\varpi)} L^\delta_{n,r} ,
\end{align*}
where the sum is taken over all those spanning forests $\forest(\varpi)$ of the $N$-complete graph whose trees induce the partition $\varpi$, 
the product over edges $\langle n, r\rangle$ in the forest $\forest(\varpi)$, 
and the summands are currents $L^\delta_{n,r}$ between the vertices $n$ and $r$, 
that is, negatives of the entries of the Dirichlet-to-Neumann matrix of the $N$-complete graph (see~\cite[Appendix~A]{KenyonWilsonBoundaryPartitionsTreesDimers}). 
Concretely, $L^\delta_{n,r}$ can be written in terms of differences of boundary values of discrete holomorphic functions as follows.
For each $n \in \{1,2,\ldots,N\}$, let 
$\phi^{\delta}_{n}$ 
be the discrete holomorphic function with boundary data\footnote{Throughout, we use the cyclic indexing convention $x_{2N+1}^{\delta} := x_{1}^{\delta}$ etc.} 
(it is not hard to see that $\phi^{\delta}_{n}$ is uniquely determined)
\begin{align} \label{eqn::boundarydata_phin_new_discrete}
\begin{cases}
\Re \phi_{n}^{\delta} \equiv 1 \; \textnormal{on } (x_{2n-1}^{\delta} \, x_{2n}^{\delta}) ; \\[.3em]  
\Re \phi_{n}^{\delta} \equiv 0 \; \textnormal{on the other arcs } (x_{2r-1}^{\delta} \, x_{2r}^{\delta}) , \textnormal{ for } r \neq n ; \\[.3em]   
\Im \phi_{n}^{\delta} \equiv 0 \;\textnormal{on }(x_{2n-2}^{\delta, *} \, x_{2n-1}^{\delta, *}); \textnormal{ and } \\[.3em]  
\Im \phi_{n}^{\delta} \; \textnormal{ is a constant on the other arcs }
(x_{2r}^{\delta, *} \, x_{2r+1}^{\delta, *}) , \textnormal{ for }  
r \neq n-1 .
\end{cases}
\end{align}
Write
$\Im \phi_{n}^{\delta} \equiv c_{n,r}^\delta$ for the constant values on the arcs 
$(x_{2r}^{\delta, *} \, x_{2r+1}^{\delta, *})$, so $c_{n,n-1}^\delta=0$. 
Then, 
\begin{align*}
L^\delta_{n,r} 
= \; & c_{n,r}^\delta - c_{n,r-1}^\delta .
\end{align*}
Similarly as in the proof of Proposition~\ref{prop::holo_cvg}, 
we see that if 
the primal polygons $(\Omega^{\delta}; x_1^{\delta}, \ldots, x_{2N}^{\delta})$ converge to a polygon $(\Omega; x_1, \ldots, x_{2N})$ in the Carath\'{e}odory sense as $\delta\to 0$, 
then $\phi^{\delta}_{n}$ converges locally uniformly as $\delta\to 0$ to the unique holomorphic function $\phi_{n}$ on $\Omega$ with boundary data analogous to~\eqref{eqn::boundarydata_phin_new_discrete}. 
Note that $\phi_{n}$ is the unique conformal map from $\Omega$ onto the rectangle of unit width with horizontal slits such that $\phi_{n}$ maps $(x_{2n-2}, x_{2n-1}, x_{2n}, x_{2n+1})$ to the four corners of the rectangle with $\phi_{n}(x_{2n-2}) = 0$, 
and it maps each $(x_{2r} \, x_{2r+1})$ to horizontal slits, for $r \in \{n+1, n+2, \ldots, N, N+1, \ldots, N+n-2\}$.

\smallbreak

From these observables we readily obtain the convergence of the crossing probabilities of the UST Peano curves (Proposition~\ref{prop::proba_limit} in Section~\ref{subsec::ust_crossingproba}).
In Section~\ref{subsec::ppf_asy}, we derive the asymptotic properties~(\ref{eqn::ppf_asy1},~\ref{eqn::ppf_asy2}) of the pure partition functions $\PartF_\alpha$, thereby verifying a part \textnormal{(ASY)} of Theorem~\ref{thm::ppf}. 
We complete the proof of Theorems~\ref{thm::coulombgasintegral} and~\ref{thm::ppf} in Section~\ref{subsec::ppf_concluding}.

\subsection{Crossing probabilities --- proof of Theorem~\ref{thm::ust_crossing_proba}}
\label{subsec::ust_crossingproba}

Let us now consider the $N$ UST Peano curves. 
For each $\delta > 0$, the endpoints of these curves
give rise to a random planar link pattern $\conn^{\delta}$ in $\LP_N$. 
From Equation~\eqref{eq: grove crossing general}, we obtain a formula for the crossing probabilities: 
\begin{align*}
\PP_{\beta}^{\delta}[\conn^{\delta}=\alpha] 
\; =  \;\; & \LM_{\alpha, \beta}\; \frac{\rm{Pr}^\delta[\grconn^{\delta}=\pi(\alpha)]}{\sum_{\gamma} \LM_{\gamma,\beta} \, \rm{Pr}^\delta[\grconn^{\delta}=\pi(\gamma)]} 
\; = \; \LM_{\alpha, \beta} \; \frac{\sum_{\varpi}P_{\pi(\alpha),\varpi} \, L^\delta_\varpi}{\sum_\varpi \LE_{\beta,\varpi} \, L^\delta_\varpi} , \\[.5em]
\textnormal{where} \quad
\LE_{\beta, \varpi} = \; & \sum_{\gamma}\LM_{\gamma,\beta} \, P_{\pi(\gamma),\varpi} ,
\end{align*}
and where the sum $\sum_{\varpi}$ is taken over all partitions $\varpi$ of $\{1,2, \ldots, N\}$ and the sum $\sum_{\gamma}$ is taken over all planar link patterns $\gamma \in \LP_N$. 
The results of~\cite{KenyonWilsonBoundaryPartitionsTreesDimers} imply that $\LE$ is a matrix whose elements equal $0$ or $1$, 
and each row of $\LE$ contains at least one non-zero element 
(see more details in~\cite[Section~2.3]{KenyonWilsonBoundaryPartitionsTreesDimers}).
In particular, we readily obtain that the UST crossing probabilities 
in Theorem~\ref{thm::ust_crossing_proba} have conformally invariant scaling limits. 
We emphasize that, for this 
we only require the convergence of polygons in the Carath\'{e}odory sense, 
that is weaker than~\eqref{eqn::polygon_cvg}.
Also, we do not require any additional regularity of $\partial\Omega$ here ---
we only need that $\partial\Omega$ is locally connected.

\begin{proposition} \label{prop::proba_limit}
Fix $N\ge 1$ and a polygon $(\Omega; x_1, \ldots, x_{2N})$. 
Suppose that a sequence $(\Omega^{\delta}; x_1^{\delta}, \ldots, x_{2N}^{\delta})$ of primal polygons converges to $(\Omega; x_1, \ldots, x_{2N})$ 
in the Carath\'{e}odory sense as $\delta\to 0$. 
For any $\alpha, \beta\in\LP_N$ with $\LM_{\alpha, \beta}=1$, the following limit exists and is positive: 
\begin{align} \label{eqn::proba_cvg}
p_{\beta}^{\alpha} =
p_{\beta}^{\alpha}(\Omega; x_1, \ldots, x_{2N}) 
:= \lim_{\delta\to 0} \PP_{\beta}^{\delta}[\conn^{\delta}=\alpha] \; > \; 0 .
\end{align}
Furthermore, the limit 
satisfies the following properties. 
\begin{enumerate}
\item \textnormal{\bf Conformal invariance:} 
\label{item:cross_proba_CI}
For any conformal map $\varphi$ on $\Omega$,~we~have 
\begin{align}\label{eqn::proba_confinv}
p_{\beta}^{\alpha}(\varphi(\Omega); \varphi(x_1), \ldots, \varphi(x_{2N})) 
= p_{\beta}^{\alpha}(\Omega; x_1, \ldots, x_{2N}). 
\end{align}

\smallbreak

\item \textnormal{\bf Smoothness:} 
\label{item:cross_proba_smooth}
$p_{\beta}^{\alpha}(x_1, \ldots, x_{2N}) := p_{\beta}^{\alpha}(\HH; x_1, \ldots, x_{2N})$ 
is a smooth function on $\chamber_{2N}$. 
\end{enumerate}
\end{proposition}

\begin{proof}
Since $L^{\delta}_{n,r} \to L_{n,r}>0$ and $L_{\varpi}^{\delta} \to L_{\varpi}>0$ as $\delta\to 0$, 
and both $L_{n,r}$ and $L_{\varpi}$ are conformally invariant functions, and continuous in $(x_1, \ldots, x_{2N})$, we obtain~\eqref{eqn::proba_cvg}: 
\begin{align}\label{eqn::KW_limit}
p_{\beta}^{\alpha}(\Omega; x_1, \ldots, x_{2N}) 
:= \lim_{\delta\to 0}\PP_{\beta}^{\delta}[\conn^{\delta}=\alpha] 
= \frac{\sum_{\varpi}P_{\pi(\alpha),\varpi} \, L_\varpi}{\sum_\varpi \LE_{\beta,\varpi} \, L_\varpi}  \; > \; 0 .
\end{align}
The conformal invariance~\eqref{eqn::proba_confinv} follows from properties of $L_{\varpi}$.
The smoothness of $L_{n,r}$ follows from an explicit formula for $\phi_n$ given by an integral
of type~\eqref{eq: observable_limit_detA-DetC}. 
Thus, $L_{\varpi}$ and $p^{\alpha}_{\beta}$ are also smooth.
\end{proof}

\ustcrossingproba*

\begin{proof}[Proof of Theorem~\ref{thm::ust_crossing_proba}]
By~\eqref{eqn::KW_limit} it suffices to prove that 
\begin{align*}
p_{\beta}^{\alpha}(\Omega; x_1, \ldots, x_{2N})  = \frac{\PartF_{\alpha}(\Omega; x_1, \ldots, x_{2N})}{\LF_{\beta}(\Omega; x_1, \ldots, x_{2N})} .
\end{align*}
This follows from  Proposition~\ref{prop::indeptbc2} in Appendix~\ref{app: marginal law}, considering the conditional law of the Peano curve and its scaling limit given the event $\{\conn^{\delta}=\alpha\}$.
\end{proof}

\subsection{Asymptotics of pure partition functions}
\label{subsec::ppf_asy}

The purpose of this section is to verify the asymptotic properties \textnormal{(ASY)}~(\ref{eqn::ppf_asy1},~\ref{eqn::ppf_asy2}) 
of the pure partition functions $\PartF_\alpha$ in Theorem~\ref{thm::ppf}.

\smallbreak

Throughout, fix $N\ge 2$ and $j\in\{1,2,\ldots, 2N-1\}$. We will use the following two operations on planar link patterns.
\begin{itemize}
\item \textnormal{\bf Link removal:} 
$\hat{\wr}_j(\alpha) := \alpha/\{j,j+1\}$, a map
$\hat{\wr}_j \colon \{\alpha\in\LP_N \colon \{j,j+1\}\in\alpha\}\to \LP_{N-1}$, 
where $\alpha/\{j,j+1\} \in \LP_{N-1}$ is the link pattern obtained from $\alpha$ by removing $\{j,j+1\}$ and relabeling the remaining indices by $1, 2, \ldots, 2N-2$.
Note that $\hat{\wr}_j$ is a bijection.

\smallbreak

\item \textnormal{\bf Link tying:} 
$\hat{\wp}_j(\alpha) := \wp_j(\alpha)/\{j,j+1\}$, a map 
$\hat{\wp}_j \colon \{\alpha\in\LP_N \colon \{j,j+1\}\not\in\alpha\}\to \LP_{N-1}$, 
where $\wp_j \colon \LP_N\to \LP_N$ 
is the ``tying operation'' defined by 
\begin{align} \label{eq::tying_operation}
\wp_j \colon \LP_N\to \LP_N , \quad
\wp_j(\alpha) = 
\big(\alpha\setminus(\{j,\ell_1\}, \{j+1, \ell_2\})\big)\cup \{j,j+1\}\cup \{\ell_1, \ell_2\} , 
\end{align}  
where $\ell_1$ \textnormal{(}resp.~$\ell_2$\textnormal{)} 
is the pair of $j$ \textnormal{(}resp.~$j+1$\textnormal{)} in $\alpha$ \textnormal{(}and $\{j,\ell_1\}, \{j+1, \ell_2\}, \{\ell_1, \ell_2\}$ are unordered\textnormal{)} --- see Figure~\ref{fig::tying}.
Note that $\hat{\wp}_j$ is a surjection, but not a bijection.
\end{itemize}

\begin{figure}[ht!]
\begin{center}
\includegraphics[scale=0.25]{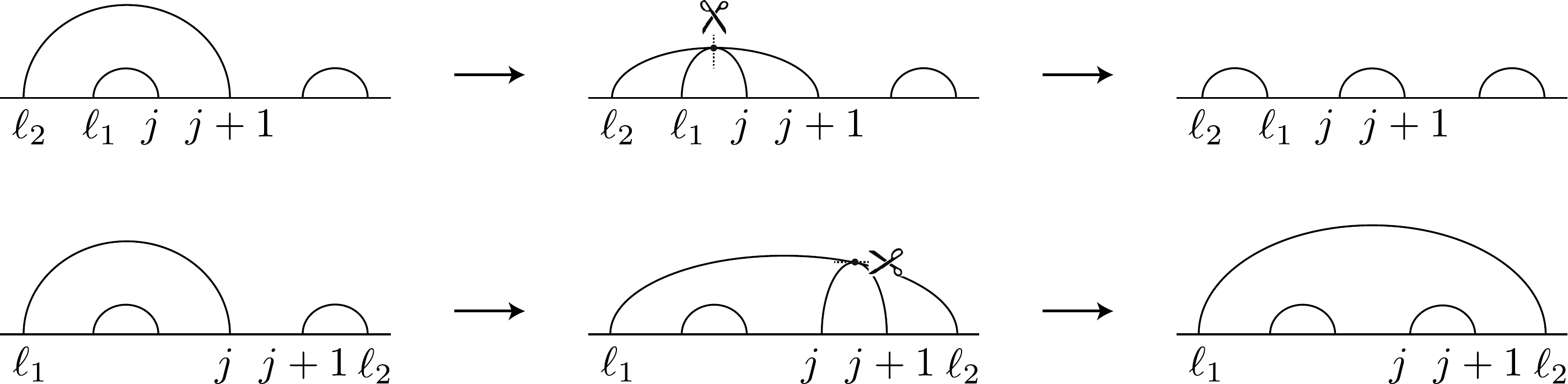}
\end{center}
\caption{\label{fig::tying} 
Illustration of the tying operation~\eqref{eq::tying_operation} on link patterns.}
\end{figure}

\begin{lemma}\label{lem::meander_lemma}
Fix $N\ge 2$. 
The following properties hold for the symmetric matrix~\eqref{eqn::renormalized_meander_matrix}. 
\begin{enumerate}
\item \label{item:meander_removal}
If $\LM_{\beta,\alpha}=1$ and $\{j,j+1\}\in\alpha$, then 
we have 
$\{j,j+1\}\not\in\beta$ and $\LM_{\hat{\wr}_j(\alpha), \hat{\wp}_j(\beta)}=1$.

\smallbreak

\item \label{item:meander_tying}
If $\{j,j+1\}\in\alpha$ and $\{j,j+1\}\not\in\beta$, then 
we have 
$\LM_{\hat{\wr}_j(\alpha), \hat{\wp}_j(\beta)} = \LM_{\alpha, \beta}$.

\smallbreak

\item \label{item::alpha_nicebeta_item1}
If $\{j,j+1\}\not\in\alpha$, then there exists $\beta\in\LP_N$ such that $\{j,j+1\}\in\beta$ and $\LM_{\alpha,\beta}=1$. 
\end{enumerate}
\end{lemma}

\begin{proof} 
Properties~\ref{item:meander_removal}~\&~\ref{item:meander_tying}  are straightforward, see Figure~\ref{fig::meander_removal}. 
Property~\ref{item::alpha_nicebeta_item1} can be verified by induction: the case of $N=2$ is trivial. 
When $N\ge 3$, for each $\alpha\in\LP_N$ with $\{j,j+1\}\not\in\alpha$, we find
$\LM_{\alpha, \hat{\wr}_j^{-1}(\gamma)}=1$ by taking 
$\gamma\in\LP_{N-1}$ such that $\LM_{\hat{\wp}_j(\alpha), \gamma}=1$, and $\beta := \hat{\wr}_j^{-1}(\gamma)$.
\end{proof}

\begin{figure}[ht!]
\begin{center}
\includegraphics[width=0.7\textwidth]{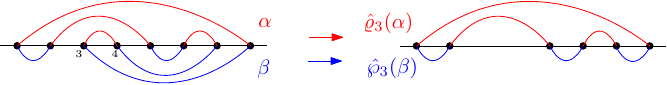}
\end{center}
\caption{\label{fig::meander_removal}
Consider $\alpha=\{\{1,8\}, \{2,5\}, \{3,4\}, \{6,7\}\}$ and $\beta=\{\{1,2\}, \{3,8\}, \{4,7\}, \{5,6\}\}$, which satisfy $\LM_{\alpha,\beta}=1$. 
Also, we see that $\{3,4\}\in\alpha$ and $\hat{\wr}_3(\alpha)=\{\{1,6\}, \{2,3\}, \{4,5\}\}$ 
as well as $\{3,4\}\not\in\beta$ and $\hat{\wp}_3(\beta)=\{\{1,2\}, \{3,4\}, \{5,6\}\}$. 
Lastly, note that $\LM_{\hat{\wr}_3(\alpha), \hat{\wp}_3(\beta)}=1$. }
\end{figure}

To prove the asymptotic properties~(\ref{eqn::ppf_asy1}~\ref{eqn::ppf_asy2})
of the pure partition functions $\PartF_{\alpha}$ defined in Equation~\eqref{eqn::ppf_def}, 
we make use of the corresponding
properties~(\ref{eqn::USTASY1}~\ref{eqn::USTASY2})
of the Coulomb gas functions 
\begin{align} \label{eqn::LF_in_terms_of_ppf}
\LF_{\beta} = \sum_{\alpha \in \LP_N} \LM_{\beta, \alpha} \,\PartF_{\alpha} . 
\end{align}

\begin{lemma}\label{lem::ppf_asy2_help}
Let $\alpha \in \LP_N$. 
Fix $j \in \{1, 2, \ldots, 2N-1 \}$ and suppose $\{j,j+1\}\not\in\alpha$. 
Then, for all $\xi \in (x_{j-1}, x_{j+2})$, the following asymptotic property holds: 
\begin{align}\label{eq::ppf_asy2_help}
\; & \lim_{x_j , x_{j+1} \to \xi} 
\frac{\PartF_{\alpha}(\boldsymbol{x})}{(x_{j+1} - x_j)^{1/4} |\log(x_{j+1}-x_j)|} 
= 0 ,
\end{align}
where  $\boldsymbol{x} = (x_1, \ldots,x_{2N})$ 
and $\boldsymbol{\ddot{x}}_j = (x_1, \ldots, x_{j-1}, x_{j+2}, \ldots, x_{2N})$. 
\end{lemma}

\begin{proof}
Pick $\beta \in \LP_N$ 
as in Item~\ref{item::alpha_nicebeta_item1} of Lemma~\ref{lem::meander_lemma}. 
From the proof of Theorem~\ref{thm::ust_crossing_proba}, we see that the relation $\PartF_{\alpha}(\boldsymbol{x}) = \LM_{\alpha, \beta} \, p_{\beta}^{\alpha}(\boldsymbol{x}) \, \LF_{\beta}(\boldsymbol{x})$ holds, where $p_{\beta}^{\alpha}(\boldsymbol{x}) \in [0,1]$. 
Hence, by Proposition~\ref{prop: Asymptotics F}, 
\begin{align*}
\lim_{x_j , x_{j+1} \to \xi} 
\frac{\PartF_{\alpha}(\boldsymbol{x})}{(x_{j+1} - x_j)^{1/4} |\log(x_{j+1}-x_j)|} 
= \; & \lim_{x_j , x_{j+1} \to \xi} 
\frac{\LM_{\alpha, \beta} \; p_{\beta}^{\alpha}(\boldsymbol{x}) \; \LF_{\beta}(\boldsymbol{x})}{(x_{j+1} - x_j)^{1/4} |\log(x_{j+1}-x_j)|} 
= 0 ,
\end{align*} 
since $\LF_{\beta}$ has the subleading asymptotics~\eqref{eqn::USTASY1}. 
\end{proof}

\begin{lemma}\label{lem::ppf_asy1}
Let $\alpha \in \LP_N$. 
Fix $j \in \{1,2, \ldots, 2N-1 \}$ and suppose $\{j, j+1\}\in\alpha$. 
Then, for all $\xi \in (x_{j-1}, x_{j+2})$, the asymptotic property~\eqref{eqn::ppf_asy1} holds: 
\begin{align} \label{eq: eqn::ppf_asy1_again}
\; & \lim_{x_j , x_{j+1} \to \xi} 
\frac{\PartF_{\alpha}(\boldsymbol{x})}{(x_{j+1} - x_j)^{1/4} |\log(x_{j+1}-x_j)|} 
= \PartF_{\hat{\wr}_j(\alpha)}(\boldsymbol{\ddot{x}}_j) ,
\end{align}
where  $\boldsymbol{x} = (x_1, \ldots,x_{2N})$ 
and $\boldsymbol{\ddot{x}}_j = (x_1, \ldots, x_{j-1}, x_{j+2}, \ldots, x_{2N})$. 
\end{lemma}

\begin{proof}
Pick $\beta \in \LP_N$ satisfying $\{j,j+1\}\not\in\beta$. 
By~\eqref{eqn::LF_in_terms_of_ppf}, we have
\begin{align*}
\frac{\LF_{\beta}(\boldsymbol{x})}{(x_{j+1} - x_j)^{1/4} |\log(x_{j+1}-x_j)|} 
= \sum_{\gamma \in \LP_N} \LM_{\beta, \gamma} \,
\frac{\PartF_{\gamma}(\boldsymbol{x})}{(x_{j+1} - x_j)^{1/4} |\log(x_{j+1}-x_j)|} .
\end{align*}
Taking the limit $x_j , x_{j+1} \to \xi$ of both sides\footnote{Note that we already know that the limits exist by Proposition~\ref{prop: Asymptotics F} and Equation~\eqref{eqn::ppf_def}.}, we see by~\eqref{eqn::USTASY2} (from Proposition~\ref{prop: Asymptotics F}) that 
\begin{align*}
\LF_{\hat{\wp}_j(\beta)}(\boldsymbol{\ddot{x}}_j) 
= \; & \sum_{\substack{\gamma \in \LP_N \colon \\ \{j,j+1\}\in\gamma}} \LM_{\beta, \gamma} \,
\lim_{x_j , x_{j+1} \to \xi} 
\frac{\PartF_{\gamma}(\boldsymbol{x})}{(x_{j+1} - x_j)^{1/4} |\log(x_{j+1}-x_j)|} ,
\end{align*}
also noting that the terms with $\{j,j+1\}\not\in\gamma$ vanish in this limit by the property~\eqref{eq::ppf_asy2_help} from 
Lemma~\ref{lem::ppf_asy2_help} for $\PartF_{\gamma}$. 
Invoking again the relation~\eqref{eqn::LF_in_terms_of_ppf} (applied to $\hat{\wp}_j(\beta)$), we obtain
\begin{align*}
\sum_{\hat{\gamma} \in \LP_{N-1}} \LM_{\hat{\wp}_j(\beta), \hat{\gamma}} \,\PartF_{\hat{\gamma}} (\boldsymbol{\ddot{x}}_j)
= \; & \sum_{\substack{\gamma \in \LP_N \colon \\ \{j,j+1\}\in\gamma}} \LM_{\hat{\wp}_j(\beta), \hat{\wr}_j(\gamma)} \,
\lim_{x_j , x_{j+1} \to \xi} 
\frac{\PartF_{\gamma}(\boldsymbol{x})}{(x_{j+1} - x_j)^{1/4} |\log(x_{j+1}-x_j)|} ,
\end{align*}
also using using Item~\ref{item:meander_tying} of Lemma~\ref{lem::meander_lemma} to write $\LM_{\beta,\gamma} = \LM_{\hat{\wp}_j(\beta),\hat{\wr}_j(\gamma)}$. 
After re-indexing the sum via the bijection $\hat{\wr}_j$, 
recalling that $\hat{\wp}_j$ is a surjection, and that the meander matrix is invertible, we obtain~\eqref{eq: eqn::ppf_asy1_again}.
\end{proof}

\begin{lemma}\label{lem::ppf_asy2}
Let $\alpha \in \LP_N$. 
Fix $j \in \{1, 2, \ldots, 2N-1 \}$ and suppose $\{j,j+1\}\not\in\alpha$. 
Then, for all $\xi \in (x_{j-1}, x_{j+2})$, the asymptotic property~\eqref{eqn::ppf_asy2} holds: 
\begin{align} \label{eq: eqn::ppf_asy2_again}
\; & \lim_{x_j , x_{j+1} \to \xi} 
\frac{\PartF_{\alpha}(\boldsymbol{x})}{(x_{j+1} - x_j)^{1/4}} 
= \pi \, \PartF_{\hat{\wp}_j(\alpha)}(\boldsymbol{\ddot{x}}_j) ,
\end{align}
where  $\boldsymbol{x} = (x_1, \ldots,x_{2N})$ 
and $\boldsymbol{\ddot{x}}_j = (x_1, \ldots, x_{j-1}, x_{j+2}, \ldots, x_{2N})$. 
\end{lemma}

\begin{proof}
Pick $\beta \in \LP_N$ satisfying $\{j,j+1\}\in\beta$. 
By~\eqref{eqn::LF_in_terms_of_ppf} and Item~\ref{item:meander_removal} of Lemma~\ref{lem::meander_lemma}, we have
\begin{align*}
\frac{\LF_{\beta}(\boldsymbol{x})}{(x_{j+1} - x_j)^{1/4}} 
= \sum_{\gamma \in \LP_N} \LM_{\gamma, \beta} \,
\frac{\PartF_{\gamma}(\boldsymbol{x})}{(x_{j+1} - x_j)^{1/4}} 
= \sum_{\substack{\gamma \in \LP_N \\ \{j,j+1\}\not\in\gamma}} \LM_{\hat{\wp}_j(\gamma), \hat{\wr}_j(\beta)} \,
\frac{\PartF_{\gamma}(\boldsymbol{x})}{(x_{j+1} - x_j)^{1/4}} .
\end{align*}
Taking the limit $x_j , x_{j+1} \to \xi$ of both sides\footnote{Note that from Theorem~\ref{thm::ust_crossing_proba}, we know that the summands are positive.}, we see by~\eqref{eqn::USTASY1} (from Proposition~\ref{prop: Asymptotics F}) that 
\begin{align*}
\pi \, \LF_{\hat{\wr}_j(\beta)}(\boldsymbol{\ddot{x}}_j) 
= \; & \sum_{\substack{\gamma \in \LP_N \colon \\ \{j,j+1\}\not\in\gamma}} \LM_{\hat{\wp}_j(\gamma), \hat{\wr}_j(\beta)} \,
\lim_{x_j , x_{j+1} \to \xi} \frac{\PartF_{\gamma}(\boldsymbol{x})}{(x_{j+1} - x_j)^{1/4}} .
\end{align*}
Invoking again the relation~\eqref{eqn::LF_in_terms_of_ppf} (applied to $\hat{\wr}_j(\beta)$), we obtain
\begin{align*}
\sum_{\hat{\mu} \in \LP_{N-1}} \LM_{\hat{\wr}_j(\beta), \hat{\mu}} \, \PartF_{\hat{\mu}}(\boldsymbol{\ddot{x}}_j) 
= \; & \frac{1}{\pi} 
\sum_{\substack{\gamma \in \LP_N \colon \\ \{j,j+1\}\not\in\gamma}} \LM_{\hat{\wp}_j(\gamma), \hat{\wr}_j(\beta)} \,
\lim_{x_j , x_{j+1} \to \xi} \frac{\PartF_{\gamma}(\boldsymbol{x})}{(x_{j+1} - x_j)^{1/4}} \\
= \; & \frac{1}{\pi} 
\sum_{\hat{\mu} \in \LP_{N-1}} \sum_{\gamma \in \hat{\wp}_j^{-1}(\hat{\mu})}
\LM_{\hat{\mu}, \hat{\wr}_j(\beta)} \,
\lim_{x_j , x_{j+1} \to \xi} \frac{\PartF_{\gamma}(\boldsymbol{x})}{(x_{j+1} - x_j)^{1/4}} .
\end{align*}
Using the bijection $\hat{\wr}_j(\beta) = \hat{\beta} \leftrightarrow \beta$ and the inverse meander matrix, we obtain~\eqref{eq: eqn::ppf_asy2_again}: 
\begin{align*}
\PartF_{\hat{\wp}_j(\alpha)}(\boldsymbol{\ddot{x}}_j) 
= \; & \sum_{\hat{\beta} \in \LP_{N-1}} \LM_{\hat{\wp}_j(\alpha), \hat{\beta}}^{-1} \,
\sum_{\hat{\mu} \in \LP_{N-1}} \LM_{\hat{\beta}, \hat{\mu}} \, \PartF_{\hat{\mu}}(\boldsymbol{\ddot{x}}_j) \\
= \; & \frac{1}{\pi} 
\sum_{\hat{\beta} \in \LP_{N-1}} \LM_{\hat{\wp}_j(\alpha), \hat{\beta}}^{-1} \,
\sum_{\hat{\mu} \in \LP_{N-1}} \sum_{\gamma \in \hat{\wp}_j^{-1}(\hat{\mu})}
\LM_{\hat{\mu}, \hat{\beta}} \,
\lim_{x_j , x_{j+1} \to \xi} \frac{\PartF_{\gamma}(\boldsymbol{x})}{(x_{j+1} - x_j)^{1/4}} \\
= \; & \frac{1}{\pi} \sum_{\gamma \in \hat{\wp}_j^{-1}(\hat{\wp}_j(\alpha))}
\lim_{x_j , x_{j+1} \to \xi} \frac{\PartF_{\gamma}(\boldsymbol{x})}{(x_{j+1} - x_j)^{1/4}} \\
= \; & \frac{1}{\pi} \lim_{x_j , x_{j+1} \to \xi} \frac{\PartF_{\alpha}(\boldsymbol{x})}{(x_{j+1} - x_j)^{1/4}} .
\end{align*}
This shows the asserted asymptotics property~\eqref{eq: eqn::ppf_asy2_again}.
\end{proof}

\subsection{Finishing the proofs of Theorems~\ref{thm::coulombgasintegral} and~\ref{thm::ppf}}
\label{subsec::ppf_concluding}

\begin{itemize}
\item The function $\LF_{\beta}$ satisfies PDEs~\eqref{eqn::USTPDE} due to Proposition~\ref{prop: PDEs F}, or Corollary~\ref{cor::coulombgasintegralPDE}. 

\smallbreak

\item The function $\LF_{\beta}$ satisfies M\"obius covariance~\eqref{eqn::USTCOV} due to Proposition~\ref{prop: full Mobius covariance F}. 

\smallbreak

\item Since $\LF_{\beta}$'s satisfy PDEs~\eqref{eqn::USTPDE} and M\"obius covariance~\eqref{eqn::USTCOV}, and $\PartF_{\alpha}$ is a linear combination~\eqref{eqn::ppf_def} of $\LF_{\beta}$'s, we see that $\PartF_{\alpha}$ also satisfies PDEs~\eqref{eqn::USTPDE} and M\"obius covariance~\eqref{eqn::USTCOV}. 

\smallbreak

\item The asymptotics~\eqref{eqn::USTASY1} and~\eqref{eqn::USTASY2} of $\LF_{\beta}$ are proved in Proposition~\ref{prop: Asymptotics F}.

\smallbreak

\item The asymptotics~\eqref{eqn::ppf_asy1} and~\eqref{eqn::ppf_asy2} of $\PartF_{\alpha}$ are proved in Lemmas~\ref{lem::ppf_asy1} and~\ref{lem::ppf_asy2}.
 
\smallbreak
 
\item We have $\LF_{\beta}>0$ due to Proposition~\ref{prop: positivity}.  

\smallbreak

\item The positivity of $\PartF_{\alpha}$ follows because $\PartF_{\alpha}=p^{\alpha}_{\beta} \, \LF_{\beta}$ and $p^{\alpha}_{\beta}, \,\LF_{\beta} > 0$ when $\LM_{\alpha, \beta}=1$.
\end{itemize}

\medbreak

It remains to show the linear independence of both collections
$\{\LF_{\beta} \colon \beta\in\LP_N\}$ and $\{\PartF_{\alpha} \colon \alpha \in\LP_N\}$. 
We give an argument based on the asymptotic properties~\eqref{eqn::USTASY1} and~\eqref{eqn::USTASY2} of $\LF_{\beta}$. 
For $\alpha=\{\{a_1, b_1\}, \ldots, \{a_N, b_N\}\}\in\LP_N$ with $a_r < b_r$ for $1\le r\le N$
(but not with link endpoints ordered as in~\eqref{eq: link pattern ordering}),
we say that the ordering $\{a_1, b_1\}, \ldots, \{a_N, b_N\}$ of links is \emph{allowable} 
if all links of $\alpha$ can be removed in the order $\{a_1, b_1\}, \ldots, \{a_N, b_N\}$ 
in such a way that at each step, the link to be removed connects two consecutive indices 
(cf.~\cite{FloresKlebanPDE, PeltolaWuGlobalMultipleSLEs}).
Note that each $\alpha$ has at least one allowable ordering. 
For such a choice, the iterated limit 
\begin{align*}
&\mathrm{Lim}_{\alpha}(F)\\
:= &\lim_{x_{a_N}, x_{b_N}\to \xi_N} \cdots \lim_{x_{a_1}, x_{b_1}\to \xi_1}|x_{b_N}-x_{a_N}|^{-1/4}\frac{|x_{b_{N-1}}-x_{a_{N-1}}|^{-1/4}}{|\log(x_{b_{N-1}}-x_{a_{N-1}})|}\cdots\frac{|x_{b_1}-x_{a_1}|^{-1/4}}{|\log(x_{b_1}-x_{a_1})|} \; F (\boldsymbol{x}) ,
\end{align*} 
at $\xi_1<\cdots<\xi_N$, 
with $\boldsymbol{x} = (x_1, \ldots, x_{2N})$, 
is well-defined for any function $F \colon \chamber_{2N}\to \C$ in the span of $\{\LF_{\beta} \colon \beta\in\LP_N\}$ or $\{\PartF_{\alpha} \colon \alpha \in\LP_N\}$. 
Thus, it defines a linear operator on either span. Furthermore, (\ref{eqn::ppf_asy1},~\ref{eqn::ppf_asy2}) yield
\begin{align*}
\mathrm{Lim}_{\alpha}(\PartF_{\beta}) 
= \begin{cases}
0, &\textnormal{if }\beta\neq\alpha , \\
\pi, &\textnormal{if }\beta=\alpha. 
\end{cases}
\end{align*}
This implies that the collection $\{\PartF_{\alpha} \colon \alpha \in\LP_N\}$ is linearly independent. 
As the meander matrix~\eqref{eqn::renormalized_meander_matrix} is invertible, 
by definition~\eqref{eqn::ppf_def} of $\PartF_\alpha$, 
it gives rise to a change of basis between
$\{\PartF_{\alpha} \colon \alpha \in\LP_N\}$ and $\{\LF_{\beta} \colon \beta\in\LP_N\}$,
so the latter collection is linearly independent too.  
\qed

\appendix

\smallskip{}
\section{Period matrices for cycles and intervals}
\label{app::matrices}

We record here a useful identity relating the two matrices (\ref{eq::P_matrices},~\ref{eq::A_matrices}). 
As a warm-up, note that by deforming and decomposing 
the integration of $\omega_0$ along the loop $\acycle^{\vcenter{\hbox{\includegraphics[scale=0.2]{figures/link-0.pdf}}}}_1$ into 
the integration along the interval $[x_{1}, x_{2}]$
and $\exp(\frac{2 \pi \ii}{2} )$ times the same negatively oriented interval, we have
\begin{align} \label{eq::relate_loop_and_interval}
\ointclockwise_{\acycle^{\vcenter{\hbox{\includegraphics[scale=0.2]{figures/link-0.pdf}}}}_1} \omega_0
= 2 \int_{x_{1}}^{x_{2}} \omega_0 .
\end{align}

In general, we have the following upper-triangular transformation.

\begin{lemma} \label{lem: relation_of_Pcirc_and_P_matrix}
We have $O_\beta \, P_{\beta} = P_{\beta}^{\circ}$, 
where $O_\beta$ is an upper-triangular matrix with entries
\begin{align} \label{eq::Mbeta_explicit}
(O_\beta)_{r,s} = 
\begin{cases}
0 , & r > s , \\
2 , & r = s , \\
4 \, (-\ii)^{a_s - a_r} , & r < s \textnormal{ and } a_s < b_r , \\
0 , & r < s \textnormal{ and } a_s > b_r ,
\end{cases}
\qquad \textnormal{for } r, s \in \{1,2\ldots,N\} .
\end{align}
\end{lemma}

\begin{proof}
From the decompositions of the integrals of $\omega_{s-1}$ along $\acycle^\beta_{r}$ on the one hand, 
and along $[x_{a_r}, x_{b_r}]$ on the other hand, 
into linear combinations 
\begin{align}
\label{eq:loop_expand}
\ointclockwise_{\acycle^\beta_{r}} \omega_{s-1} 
= \; & 2 \, \sum_{k = a_r}^{b_r-1} (-\ii)^{k - a_r} \,  \landupint_{x_{k}}^{x_{k+1}} \omega_{s-1} , \\
\label{eq:interval_expand}
\landupint_{x_{a_r}}^{x_{b_r}} \omega_{s-1} 
= \; & \sum_{k = a_r}^{b_r-1} \ii^{k - a_r} \, \landupint_{x_{k}}^{x_{k+1}} \omega_{s-1} ,
\end{align}
it is clear that $(O_\beta)_{r,s} = 0$ for all $r > s$ 
and $(O_\beta)_{r,s} = 0$ for all $r < s$ such that 
$a_s > b_r$. 
To find the non-zero entries $(O_\beta)_{1,s}$ with $s \in \{1,2\ldots,N\}$, note that by~(\ref{eq:loop_expand},~\ref{eq:interval_expand}), 
\begin{align*}
\ointclockwise_{\acycle^\beta_{1}} \omega_{s-1} 
= \; & 2 \, \landupint_{x_{a_1}}^{x_{b_1}} \omega_{s-1}
\; + \; 4 \sum_{t=2}^N \one\{a_t < b_1\} (-\ii)^{a_t - a_1} \, \landupint_{x_{a_t}}^{x_{b_t}} \omega_{s-1} \\
= \; & \sum_{t=1}^N (O_\beta)_{1,t} \, \landupint_{x_{a_t}}^{x_{b_t}} \omega_{s-1} ,
&& \quad \textnormal{[as in~\eqref{eq::Mbeta_explicit}]}  
\end{align*}
where we used the fact that the parity of $a_t$ and $b_t$ is always different. 
This shows that $(O_\beta)_{1,t}$ have the claimed form~\eqref{eq::Mbeta_explicit}.
The general formula follows inductively.
\end{proof}

\smallskip{}

\section{Examples of partition functions $\LF_\beta$}
\label{app::examples}

As examples, let us discuss $\LF_{\beta}$ for the special cases of $N=1$ and $N=2$.

\begin{lemma} \label{lem: N=1}
Take $N=1$ in~\eqref{eq: Def of F beta}. 
For $\vcenter{\hbox{\includegraphics[scale=0.3]{figures/link-0.pdf}}} = \{ \{1, 2\} \}$ and $x_1<x_2$, we have
\begin{align*}
\LF_{\vcenter{\hbox{\includegraphics[scale=0.2]{figures/link-0.pdf}}}}(x_1,x_2)
= \pi \, (x_2-x_1)^{1/4} . 
\end{align*}
\end{lemma}

\begin{proof}
By the branch choice for $f_{\beta}$, setting $v = (u-x_1) / (x_{2}-x_{1})$, we have 
\begin{align*}
\LF_{\vcenter{\hbox{\includegraphics[scale=0.2]{figures/link-0.pdf}}}}(x_1,x_2) 
= \; & 
(x_2 - x_1)^{1/4} \int_{x_1}^{x_2} 
\frac{\ud u}{|u-x_1|^{1/2} |u-x_2|^{1/2}} \\
=\; &  (x_2 - x_1)^{1/4} 
\int_0^1\frac{\ud v}{\sqrt{v (1-v)}} 
= \pi \, (x_2 - x_1)^{1/4} , 
\end{align*}
where $\sqrt{\cdot}$ denotes the principal branch of the square root.
\end{proof}

\begin{lemma} \label{lem: N=2}
Take $N=2$ in~\eqref{eq: Def of F beta}. 
For $\vcenter{\hbox{\includegraphics[scale=0.3]{figures/link-1.pdf}}} = \{\{1, 2\}, \{3, 4\} \}$ 
and $\vcenter{\hbox{\includegraphics[scale=0.3]{figures/link-2.pdf}}} = \{\{1, 4\}, \{2, 3\} \}$, 
and $x_1<x_2<x_3<x_4$, we have
\begin{align}
\LF_{\vcenter{\hbox{\includegraphics[scale=0.2]{figures/link-1.pdf}}}} (x_1,x_2,x_3,x_4)
= \; & \pi^2 \, (x_4-x_1)^{1/4}(x_3-x_2)^{1/4}z^{1/4} 
\; \hF\big(1/2, 1/2, 1; z \big),
\label{eqn::hSLE8_case1} \\
\LF_{\vcenter{\hbox{\includegraphics[scale=0.2]{figures/link-2.pdf}}}} (x_1,x_2,x_3,x_4)
= \; & \pi^2 \, (x_2-x_1)^{1/4}(x_4-x_3)^{1/4}(1-z)^{1/4} 
\; \hF \big(1/2, 1/2, 1; 1-z \big), \label{eqn::hSLE8_case2}
\end{align}where $\hF$ is the hypergeometric function~\textnormal{\cite[Eq.~(15.1.1)]{AbramowitzHandbook}} 
and 
\begin{align*}
z=\frac{(x_2-x_1)(x_4-x_3)}{(x_3-x_1)(x_4-x_2)}. 
\end{align*}
\end{lemma}

Note also that $\frac{\pi}{2} \, \hF \big(1/2, 1/2, 1; \cdot \big)$ is the elliptic integral of the first kind, which is not a surprise given the relation~\eqref{eqn::PartFcirc_detAcirc} of $\LF_{\vcenter{\hbox{\includegraphics[scale=0.2]{figures/link-1.pdf}}}}$ and $\LF_{\vcenter{\hbox{\includegraphics[scale=0.2]{figures/link-2.pdf}}}}$ to the $a$-cycles in Section~\ref{subsec::Coulomb_gas_Fbeta}.

\begin{proof}
We first show~\eqref{eqn::hSLE8_case1}. 
Proposition~\ref{prop::two_Fs_as_determinants}
gives 
\begin{align*}
\LF_{\vcenter{\hbox{\includegraphics[scale=0.2]{figures/link-1.pdf}}}} (x_1,x_2,x_3,x_4)
= \frac{1}{4} \,
\LF^\circ_{\vcenter{\hbox{\includegraphics[scale=0.2]{figures/link-1.pdf}}}} (x_1,x_2,x_3,x_4)
= \frac{\pi}{2} \, \ointclockwise_{\acycle^{\vcenter{\hbox{\includegraphics[scale=0.2]{figures/link-1.pdf}}}}_1}  \ud u_1 
\; f^\circ_{\vcenter{\hbox{\includegraphics[scale=0.2]{figures/link-1.pdf}}}}(x_1,x_2,x_3,x_4;u_1) ,
\end{align*}
where $\acycle^{\vcenter{\hbox{\includegraphics[scale=0.2]{figures/link-1.pdf}}}}_1$ is the loop surrounding
$x_1$ and $x_2$. 
Using~\eqref{eq::relate_loop_and_interval} and the branch choice of
$f^\circ_{\vcenter{\hbox{\includegraphics[scale=0.2]{figures/link-1.pdf}}}}$, we obtain 
\begin{align*}
\frac{\LF_{\vcenter{\hbox{\includegraphics[scale=0.2]{figures/link-1.pdf}}}} (x_1,x_2,x_3,x_4)}{f^{(0)}(x_1, x_2, x_3, x_4)}
= \; & 
\frac{\pi}{f^{(0)}(x_1, x_2, x_3, x_4)}
\int_{x_1}^{x_2} \ud u_1 
\; f^\circ_{\vcenter{\hbox{\includegraphics[scale=0.2]{figures/link-1.pdf}}}}(x_1,x_2,x_3,x_4;u_1) 
\\
= \; & \pi \, \int_{x_1}^{x_2} 
\frac{\ud u}{|u-x_1|^{1/2} |u-x_2|^{1/2} |u-x_3|^{1/2} |u-x_4|^{1/2}} ,
\end{align*}
where $f^{(0)}$ is defined in~\eqref{eqn::def_fnod}.
To simplify this, writing $x_{j i} = x_j - x_i$ and setting 
$w = \big(\frac{u-x_1}{x_{21}}\big)\big(\frac{x_{42}}{x_4-u}\big)$, we~have $u = x_4 - x_{41} (1 + w x_{21}/x_{42})^{-1}$, and using also~\cite[Eq.~(17.2.6), Eq.~(17.3.9)]{AbramowitzHandbook}, we obtain
\begin{align*}
\frac{\LF_{\vcenter{\hbox{\includegraphics[scale=0.2]{figures/link-1.pdf}}}} (x_1,x_2,x_3,x_4)}{f^{(0)}(x_1, x_2, x_3, x_4)} 
= & \; 
\frac{\pi }{x_{31}^{1/2} x_{42}^{1/2}}\int_0^1\frac{\ud w}{\sqrt{w (1-w) (1-zw)}} 
\\
= & \; 
\frac{2\pi }{x_{31}^{1/2} x_{42}^{1/2}}\int_0^{\pi/2}\frac{\ud\theta}{\sqrt{1-z\sin^2\theta}} 
= \frac{\pi^2 }{x_{31}^{1/2} x_{42}^{1/2}}
\; \hF \big(1/2, 1/2, 1; z \big) , 
\end{align*}
which gives~\eqref{eqn::hSLE8_case1}.
The identity~\eqref{eqn::hSLE8_case2} then follows from~\eqref{eqn::hSLE8_case1} 
and Propositions~\ref{prop: full Mobius covariance F} and~\ref{prop: positivity}. 
\end{proof}

\smallskip{}

\section{Schwarz-Christoffel type conformal mappings}
\label{app::SC_mappings}

The goal of this appendix is to derive another explicit expression 
for the observable $\phi_{\beta}$ on $\Omega=\HH$, 
compared to the one~\eqref{eq: observable_limit_detA-DetC} found in the proof of Lemma~\ref{lem::conformal_general_expansion}.   
For definiteness and without loss of generality 
(by the full M\"obius covariance 
from Proposition~\ref{prop: full Mobius covariance F}
and duality for the UST model), 
we consider the case where $\{2N-1, 2N\}\in\beta$.

\begin{proposition} \label{prop: SC map}
On $\Omega = \HH$, the holomorphic function $\phi_{\beta}$ with boundary data~\eqref{eqn::boundarydata} 
\textnormal{(}cf.~Proposition~\ref{prop::holo_cvg}\textnormal{)}
has the explicit formula\footnote{Note that~\eqref{eq: Schwarz-Christoffel formula app} is the same for any branch choice, 
since the multiplicative phase factors from the numerator and denominator cancel out.} 
\begin{align} \label{eq: Schwarz-Christoffel formula app}
\phi_{\beta}(z) 
= \phi_{\beta}(z; \boldsymbol{x}) 
= \landupint_{x_1}^z \frac{\tilde{Q}_{\beta}(u) \, \ud u}{\wfunc(u; \boldsymbol{x})} 
\bigg(\landupint_{x_1}^{x_{b_1}} \frac{\tilde{Q}_{\beta}(u) \, \ud u}{\wfunc(u; \boldsymbol{x})} \bigg)^{-1} , \qquad z \in \overline{\HH} ,
\end{align}
where $\wfunc(u; \boldsymbol{x})$ is the meromorphic function~\eqref{eqn::def_hypercurve} on $\Sigma = \Sigma_{x_1, \ldots, x_{2N}}$, and 
\begin{align*}
\tilde{Q}_{\beta}(u) 
= \tilde{Q}_{\beta}(u; \boldsymbol{x}) 
:= \prod_{\ell=1}^{N-2}(u-\mu_\ell)
= u^{N-2} + \sum_{s=0}^{N-3} \tilde{\nu}_s \, u^s 
\end{align*} 
is a monic polynomial with roots $\mu_1, \ldots, \mu_{N-2} \in \R$ 
and coefficients $\tilde{\nu}_0, \ldots, \tilde{\nu}_{N-3} \in \R$ 
determined as the unique solution $\boldsymbol{\tilde{\nu}} = (\tilde{\nu}_0,\tilde{\nu}_1,\ldots,\tilde{\nu}_{N-3},1) \in \R^{N-1}$ to the linear system
\begin{align}\label{eqn::newequationR}
M_{\beta} \boldsymbol{\tilde{\nu}}^t = (0,0,\ldots,0)^t ,
\qquad \textnormal{where} \qquad 
M_{\beta}:=\Big( \landupint_{x_{a_{r+1}}}^{x_{b_{r+1}}} \omega_{s-1} \Big)_{\substack{r \in \{1,2,\ldots,N-2\} \\ s \in \{1,2,\ldots,N-1\}}} 
\end{align}
is a line integral counterpart of a matrix involving $a$-periods of holomorphic one-forms~\eqref{eqn::parametera_polynomial}.
Here, 
$M_\beta := (\hat{A}_\beta)_{N-1,\emptyset}$ 
is the submatrix of $A_\beta$ obtained by removing the last row. 
\end{proposition}

Note that $\phi_{\beta}$ is a Schwarz-Christoffel map~\cite[Chapter~6, Section~2.2]{AhlforsComplexAnalysis},  
conformal from $\HH$ onto a slit rectangle. 
The accessory parameters $\mu_1, \ldots, \mu_{N-2} \in \R$ are those points $x_{a_{r+1}} < \mu_r < x_{b_{r+1}}$ on the real line which are mapped to the tips of the slits in the image of~$\phi_{\beta}$. 
See Figure~\ref{fig::slitrectangle} for an illustration.

\begin{proof}
Let $R_\beta := (\hat{A}_\beta)_{N-1,N-1}$ 
be the principal submatrix of $A_\beta$ obtained by removing the last row and the last column. 
Expanding $\det A_{\beta}$ according to the cofactors along the last row, we have
\begin{align*}
\det A_{\beta}
= \; & \sum_{s=1}^{N-1} (-1)^{N-1+s} \, ( \det (\hat{A}_{\beta})_{N-1,s} ) \,
\landupint_{x_{2N-1}}^{x_{2N}} \omega_{s-1} ,
\end{align*}
where 
$\det (\hat{A}_{\beta})_{N-1,s}$ is the minor obtained from $A_{\beta}$ by removing the last row and $s$:th column. Thus, we see that
\begin{align*}
\det A_{\beta} = \left( \landupint_{x_{2N-1}}^{x_{2N}} \frac{\tilde{Q}_{\beta}(u) \, \ud u}{\wfunc(u; \boldsymbol{x})} \right) \times \det R_{\beta} ,
\end{align*}
which implies by Lemma~\ref{lem: matrices_invertible} that $R_{\beta}$ is invertible and 
the linear system~\eqref{eqn::newequationR} has a unique solution.
It remains to note by~\eqref{eqn::newequationR} that the function~\eqref{eq: Schwarz-Christoffel formula app} satisfies the boundary data~\eqref{eqn::boundarydata} from Proposition~\ref{prop::holo_cvg}, which determines $\phi_{\beta}$ uniquely. 
\end{proof}

\smallskip{}

\section{Identifying Peano curves with given connectivity}
\label{app: marginal law}

The main purpose of this appendix is to verify the following property, needed in the proof of Theorem~\ref{thm::ust_crossing_proba} in Section~\ref{subsec::ust_crossingproba}. We shall use the notations from Section~\ref{sec::ppf}.
\begin{proposition}\label{prop::indeptbc2}
For all $\alpha, \beta, \gamma \in \LP_N$ such that
$\LM_{\alpha, \beta} = \LM_{\alpha, \gamma} = 1$, we have
\begin{align*}
\PartF_\alpha 
= p^{\alpha}_{\beta} \;\LF_{\beta}  = p^{\alpha}_{\gamma} \;\LF_{\gamma} .
\end{align*}
\end{proposition}

This quantity also describes the conditional laws of the scaling limit curves for each connectivity $\alpha \in \LP_N$ (in the spirit of Doob's transform), as stated in the following result. 
As in the proof of Theorem~\ref{thm::ust_general}, we will fix a sequence of conformal maps $\varphi_{\delta} \colon \Omega^{\delta, \diamond} \to \HH$ and $\varphi \colon \Omega \to \HH$
such that $\varphi(x_1)<\cdots<\varphi(x_{2N})$ and,
as $\delta \to 0$, the maps $\varphi_{\delta}^{-1}$ converge to $\varphi^{-1}$ locally uniformly on $\HH$, and $\smash{\varphi_{\delta}(x_j^{\delta, \diamond}) \to \varphi(x_j)}$ for all $1\le j\le 2N$. 
We consider the Peano curve $\eta_i^{\delta}$ started from $\smash{x_i^{\delta, \diamond}}$ in the scaling limit.

\begin{proposition} \label{prop::conditionallaw_ppf}
Assume the same setup as in Theorems~\ref{thm::ust_general} and~\ref{thm::ust_crossing_proba}. Fix $\alpha, \beta\in\LP_N$ such that $\LM_{\alpha,\beta}=1$. 
The conditional law 
of $\eta_i^{\delta}$ given $\{\conn^{\delta}=\alpha\}$ 
converges weakly to the image under $\varphi^{-1}$ of the Loewner chain with driving function solving the following SDEs, up to 
the first time when $\varphi(x_{i-1})$ or $\varphi(x_{i+1})$ is swallowed\textnormal{:} 
\begin{align*}
\begin{cases}
\ud W_t = \sqrt{8} \, \ud B_t + 8 \,  (\partial_{i}\log (p^{\alpha}_{\beta} \; \LF_{\beta}))(V_t^{1}, \ldots, V_t^{i-1}, W_t, V_t^{i+1}, \ldots, V_t^{2N}) \, \ud t, \\
\ud V_t^j =\frac{2 \, \ud t}{V_t^j-W_t},\\ 
W_0 = \varphi(x_{i}) ,\\
V_0^j=\varphi(x_j), \quad j\in\{1, \ldots, i-1, i+1, \ldots, 2N\} .
\end{cases}
\end{align*}
\end{proposition}

In particular, combining Propositions~\ref{prop::indeptbc2}~\&~\ref{prop::conditionallaw_ppf}, we see that the marginal law in the latter is given by 
the Loewner chain associated to the pure partition function $\PartF_\alpha$.

\smallbreak

We prove 
Proposition~\ref{prop::conditionallaw_ppf} in
Section~\ref{app::conditionallaw_ppf} and 
Proposition~\ref{prop::indeptbc2} in 
Section~\ref{app::indeptbc2}.

\subsection{Marginal law --- proof of Proposition~\ref{prop::conditionallaw_ppf}}
\label{app::conditionallaw_ppf}

When $\LM_{\alpha, \beta}=1$, 
we denote by 
\begin{align*}
\LL_{\beta}(\eta_i^{\delta}\cond \conn^{\delta}=\alpha)
\end{align*}
the conditional law of $\eta_i^{\delta}$ given the event that the planar link pattern $\conn^{\delta}$ induced by the Peano curves equals $\alpha$.

\begin{lemma}\label{lem::conditionallaw}
The following properties hold for the conditional law $\LL_{\beta}(\eta^{\delta}_i\cond \conn^{\delta}=\alpha)$. 
\begin{enumerate}
\item \label{item::Peanocurve_tight_cond}
The family $\{\LL_{\beta}(\eta^{\delta}_i\cond \conn^{\delta}=\alpha)\}_{\delta>0}$ of laws is 
precompact in the curve space~\eqref{eqn::metric_curvesspace}.

\smallbreak

\item \label{item::conditionallaw_indeptbeta}
The law $\LL_{\beta}(\eta_i^{\delta}\cond\conn^{\delta}=\alpha)$ does not depend on the choice of the b.c. $\beta$. 
\end{enumerate}
\end{lemma}

\begin{proof}
Property~\ref{item::Peanocurve_tight_cond} 
is a consequence of Lemma~\ref{lem::Peanocurve_tight} and~\eqref{eqn::proba_cvg}. 
For Property~\ref{item::conditionallaw_indeptbeta},  
recall that each grove induces $N$ Peano curves whose endpoints form a planar link pattern, 
and note that the conditional law of the grove Peano curve started from $x_{i}^{\delta,\diamond}$ 
given the event $\{\grconn^{\delta}=\pi(\alpha)\}$ is the same as $\LL_{\beta}(\eta_i^{\delta}\cond\conn^{\delta}=\alpha)$.
The former is independent of $\beta$, as claimed.
\end{proof}

Next, we focus on the Peano curve $\eta_1^{\delta}$ started from $x_1^{\delta, \diamond}$ and study its scaling limit ($i=1$). 
In the continuum,
we can relate the setup on $\Omega$ to the setup on $\HH$ via conformal invariance.
In $\HH = \varphi(\Omega)$, 
we denote $p_{\beta}^{\alpha}(\HH; \cdot) = p_{\beta}^{\alpha}(\cdot)$.  
From Theorem~\ref{thm::ust_general}, we already know that the scaling limit of the Peano curve $\eta_1^{\delta}$ is given by 
$\varphi^{-1}(\tilde{\eta}_1)$, where 
$\tilde{\eta}_1$ is the Loewner chain associated to the partition function $\LF_{\beta}$ started from $\varphi(x_1)$, i.e., the driving function of $\tilde{\eta}_1$ is the solution to the following system of SDEs: 
\begin{align} \label{eqn::hateta1_def}
\begin{cases}
\ud W_t = \sqrt{8} \, \ud B_t + 8 \, (\partial_{1}  \log \LF_{\beta})(W_t, V_t^2, \ldots, V_t^{2N}) \, \ud t, \\
\ud V_t^j =\frac{2 \, \ud t}{V_t^j-W_t},\\ 
W_0 = \varphi(x_{1}) ,\\
V_0^j=\varphi(x_j), \quad j\in\{2, \ldots, 2N\} .
\end{cases}
\end{align}
Let $T = T_{\varphi(x_2)}$ be the first time when $\varphi(x_2)$ is swallowed.

\begin{lemma}\label{lem::probatomart}
Assume the same setup as in Theorems~\ref{thm::ust_general} and~\ref{thm::ust_crossing_proba}.
Fix $\alpha, \beta\in\LP_N$ such that $\LM_{\alpha,\beta}=1$.
The following process is a positive
martingale 
with respect to the filtration generated by $\tilde{\eta}_1$\textnormal{:}
\begin{align}\label{eqn::proba_mart}
M_t := p_{\beta}^{\alpha}(W_t, V_t^2, \ldots, V_t^{2N}), \qquad t < T .
\end{align}
Moreover, the law $\LL_{\beta}(\varphi_{\delta}(\eta_1^{\delta})\cond \conn^{\delta}=\alpha)$ converges weakly to the law of $\tilde{\eta}_1$ weighted by $M_t$.
In particular, 
it converges weakly to the Loewner chain associated to $p_{\beta}^{\alpha} \, \LF_{\beta}$ started from $\varphi(x_1)$ up to time $T$. 
\end{lemma}

The proof follows a routine argument, which we outline below.

\begin{proof}
Let $\tau^{\delta}$ be the first time when $\eta_1^{\delta}$ hits $(x_2^{\delta, \diamond} \, x_{2N}^{\delta, \diamond})$. 
For every $t<\tau^{\delta}$, define $\smash{\Omega^{\delta,\diamond}(t)}$ to be the component of $\smash{\Omega^{\delta,\diamond}\setminus\eta_1^{\delta}[0,t]}$ with $\smash{x_2^{\delta,\diamond}}$ and $\smash{x_{2N}^{\delta, \diamond}}$ on its boundary. 
Define also $\Omega^{\delta}(t)$ to be the primal graph associated to $\Omega^{\delta, \diamond}(t)$ and define $x_1^{\delta}(t)$ to be the primal vertex of $\Omega^{\delta}(t)$ nearest to $\eta_1^{\delta}(t)$. 
Thanks to the domain Markov property of our model,
the conditional crossing probability associated to the interface $\eta_1^{\delta}$ gives a tautological martingale (with respect to the filtration generated by $\eta_1^{\delta}$):
\begin{align*}
M_t^\delta := \; &
\PP_{\beta}^{\delta}(\Omega^{\delta}(t); x_1^{\delta}(t), x_2^{\delta}, \ldots, x_{2N}^{\delta})[\conn^{\delta}=\alpha] \\
= \; & \PP_{\beta}^{\delta}(\Omega^{\delta}; x_1^{\delta}, x_2^{\delta}, \ldots, x_{2N}^{\delta})[\conn^{\delta}=\alpha\cond \eta_1^{\delta}[0,t]] , \quad t < \tau^{\delta} ,
\end{align*}
where we denote\footnote{We also denote by 
$\E_{\beta}^{\delta} = \E_{\beta}^{\delta}(\Omega^{\delta}; x_1^{\delta}, \ldots, x_{2N}^{\delta})$ the corresponding expectation.}  
by 
$\PP_{\beta}^{\delta} = \PP_{\beta}^{\delta}(\Omega^{\delta}; x_1^{\delta}, \ldots, x_{2N}^{\delta})$ the law of the UST on the primal polygon $(\Omega^{\delta}; x_1^{\delta}, \ldots, x_{2N}^{\delta})$ with b.c. $\beta \in \LP_N$. 
We define the stopping times 
\begin{align*}
\tau_\eps^\delta 
:= \inf \Big\{t \geq 0 \colon \min_{2\le j\le 2N} |\eta_1^\delta(t) - x_j^{\delta,\diamond}| = \eps \Big\} , \qquad \eps>0 ,
\end{align*}
and we similarly define the stopping times $\tau_\eps$ for $\varphi^{-1}(\tilde{\eta}_1)$. 
We may assume that $\tau_\eps^\delta\to\tau_\eps$ 
almost surely, by considering continuous modifications 
(see more details in~\cite[Appendix~B]{Karrila19}). 
Now, we find
\begin{align} \label{eqn::probamart_aux1}
\E^\delta_\beta \Big[ f \big( \varphi_{\delta}(\eta_1^\delta[0,t\wedge\tau_\eps^\delta]) \big) \, 
\one\{\conn^{\delta}=\alpha\} \Big]
\; = & \;\;
\E_\beta^\delta \Big[ f \big( \varphi_{\delta}(\eta_1^\delta[0,t\wedge\tau^\delta_\eps]) \big)  \, M_{t\wedge\tau_\eps^\delta}^\delta \Big]
\end{align}
for any bounded continuous function $f$. 
Let us consider the two sides of~\eqref{eqn::probamart_aux1} separately. 
\begin{itemize}[leftmargin=*]
\item 
LHS~of~\eqref{eqn::probamart_aux1}: 
By the precompactness from 
Item~\ref{item::Peanocurve_tight_cond} of Lemma~\ref{lem::conditionallaw}, 
we find a subsequential limit 
$\LL_{\beta}(\eta_1^{\delta_n}\cond \conn^{\delta_n}=\alpha)$ converging weakly to the law of some $\gamma_1$ as $\delta_n\to 0$.
Combining this  with~(\ref{eqn::proba_cvg},~\ref{eqn::proba_confinv}) from Proposition~\ref{prop::proba_limit}, we obtain
\begin{align*}
\E^{\delta_n}_\beta \big[ f \big(\varphi_{\delta_n}(\eta_1^{\delta_n}[0,t\wedge\tau_\eps^{\delta_n}]) \big) \, \one\{\conn^{\delta_n}=\alpha\} \big] 
\; \overset{\delta_n\to 0}{\longrightarrow} \;  & \; 
\E\big[ f \big(\varphi(\gamma_1[0,t\wedge \tau_{\eps}]) \big) \big] \;
p_{\beta}^{\alpha}(\Omega; x_1, \ldots, x_{2N})  \\
= & \; 
\E\big[ f \big(\varphi(\gamma_1[0,t\wedge \tau_{\eps}]) \big) \big] \;  M_{0}  .
\end{align*}

\smallbreak

\item 
RHS~of~\eqref{eqn::probamart_aux1}:
As $\varphi_{\delta}(\eta_1^{\delta})$ converges weakly to $\tilde{\eta}_1 $ by Theorem~\ref{thm::ust_general}, it follows that 
the discrete polygon $(\Omega^{\delta}(t\wedge \tau_{\eps}^{\delta}); x_1^{\delta}(t\wedge \tau_{\eps}), x_2^{\delta}, \ldots, x_{2N}^{\delta})$ is convergent in the Carath\'{e}odory sense  (see~\cite[Proof of Theorem~4.2]{HanLiuWuUST}). 
Combining with~(\ref{eqn::proba_cvg},~\ref{eqn::proba_confinv}) from Proposition~\ref{prop::proba_limit}, we find
\begin{align*}
\E_\beta^\delta \big[ f \big( \varphi_{\delta}(\eta_1^\delta[0,t\wedge\tau^\delta_\eps]) \big) \, M_{t\wedge\tau_\eps^\delta}^\delta\big] 
\; \overset{\delta\to 0}{\longrightarrow} \; & \; 
\E\big[ f \big(\tilde{\eta}_1[0,t\wedge\tau_\eps] \big) \, 
p_{\beta}^{\alpha}(W_{t\wedge\tau_\eps}, V_{t\wedge\tau_\eps}^2, \ldots, V_{t\wedge\tau_\eps}^{2N})\big] \\
= & \; 
\E\big[ f \big(\tilde{\eta}_1[0,t\wedge\tau_\eps] \big) \, 
M_{t\wedge\tau_\eps} \big] .  
\end{align*}
\end{itemize}
In conclusion, for any $\eps > 0$, we have
\begin{align*}
\; & \E\big[ f \big(\varphi(\gamma_1[0,t\wedge \tau_{\eps}]) \big) \big] 
= \E \Big[ f \big(\tilde{\eta}_1[0,t\wedge\tau_\eps] \big) \, 
\frac{M_{t\wedge\tau_\eps}}{M_{0}} \Big] \\
\quad \overset{\eps \to 0}{\Longrightarrow} \quad 
\; &  \E\big[ f \big(\varphi(\gamma_1[0,t\wedge T]) \big) \big] 
= \E \Big[ f \big(\tilde{\eta}_1[0,t\wedge T] \big) \, 
\frac{M_{t\wedge T}}{M_{0}} \Big] ,  
\quad  \textnormal{where} \quad \underset{\eps\to 0}{\lim}\, \tau_{\eps} \, \ge T .
\end{align*}
Indeed, a similar argument\footnote{Here, it is important that the limit of $\eta_1^{\delta}$ does not hit any points in $\{x_2, \ldots, x_{2N}\}$ except at its endpoint (Lemma~\ref{lem::Peanocurve_tight}).} 
as in~\cite[Proof of Theorem~4.2]{HanLiuWuUST}
shows that $\tau_{\eps}$ converge almost surely to a time at least $T$.
By the Radon-Nikodym theorem, this implies that up to time $T$, 
the process $M_t$ defined in~\eqref{eqn::proba_mart} is a martingale for $\tilde{\eta}_1$, 
and the limit of $\LL_{\beta}(\varphi_{\delta}(\eta_1^{\delta})
\cond \conn^{\delta}=\alpha)$ is the same as $\tilde{\eta}_1$ weighted by $M_t$. 
Furthermore, as the driving function of $\tilde{\eta}_1$ satisfies~\eqref{eqn::hateta1_def}, Girsanov's theorem yields 
\begin{align*}
\ud B_t = \ud G_t + \sqrt{8} \, \big(\partial_1\log p_{\beta}^{\alpha}\big)(W_t, V_t^2, \ldots, V_t^{2N})\ud t ,
\end{align*}
where $G_t$ is a Brownian motion under the law of $\gamma_1$.
Combining this with~\eqref{eqn::hateta1_def}, we see that $\varphi(\gamma_1)$ has the same law as the Loewner chain associated to $p^{\alpha}_{\beta} \, \LF_{\beta}$. 
\end{proof}

\subsection{Asymptotics of crossing probabilities}
\label{app::crossingproba_ASY}

In the proof of Proposition~\ref{prop::indeptbc2}, we will need the following asymptotics property of the limit crossing probabilities.

\begin{lemma}\label{lem::crossingproba_ASY}
Fix $N\ge 2$ and a polygon $(\Omega; x_1, \ldots, x_{2N})$ whose boundary is a $C^1$-Jordan curve. 
Fix $j \in \{1,2, \ldots, 2N-1 \}$ and suppose $\{j, j+1\}\in\alpha$ and $\LM_{\alpha, \beta}=1$.
Then, the function $p^{\alpha}_{\beta}$ defined in Proposition~\ref{prop::proba_limit} has the following asymptotics: 
for all $j\in\{1,2, \ldots, 2N\}$ and $\xi\in(x_j \, x_{j+1})$, we have 
\begin{align}
\label{eqn::crossingproba_ASY_alpha}
\lim_{x_j, x_{j+1}\to \xi} \;
p^{\alpha}_{\beta}(\Omega; \boldsymbol{x}) = \; & p^{\hat{\wr}_j(\alpha)}_{\hat{\wp}_j(\beta)}(\Omega; \boldsymbol{\ddot{x}}_j) ,
\end{align}
where $\boldsymbol{x} = (x_1, \ldots, x_{2N})$ 
and $\boldsymbol{\ddot{x}}_j = (x_1, \ldots, x_{j-1}, x_{j+2}, \ldots, x_{2N})$.
\end{lemma}

\begin{proof}
To facilitate notation, we take $j=2N-1$ without loss of generality. So we assume $\{2N-1, 2N\}\in\alpha$ and write $\ddot{\boldsymbol{x}} = (x_1, \ldots, x_{2N-2})$. To prove~\eqref{eqn::crossingproba_ASY_alpha}, we approximate $(\Omega; x_1, \ldots, x_{2N})$ by discrete polygons in the same setup as in Theorem~\ref{thm::ust_general}, i.e.,
in the sense of~\eqref{eqn::polygon_cvg}. 
Note that at this point, the results of Theorem~\ref{thm::ust_general} and 
Proposition~\ref{prop::proba_limit} are already at our disposal. 
In particular,~\eqref{eqn::proba_cvg} immediately gives 
\begin{align*}
\lim_{\delta\to 0} \PP_{\beta}^{\delta}[\conn^{\delta}=\alpha]
= p^{\alpha}_{\beta}(\Omega; \boldsymbol{x}) . 
\end{align*}
To prove~\eqref{eqn::crossingproba_ASY_alpha}, 
we shall consider the double-limit of 
$\PP_{\beta}^{\delta}[\conn^{\delta}=\alpha]$ as $\delta\to 0$ and $x_{2N-1}, x_{2N}\to \xi$.

Let $\eta^{\delta}_N$ denote the Peano curve starting from $x_{2N-1}^{\delta, \diamond}$.
Pick $\eps>0$ much larger than the diameter of the arc $(x_{2N-1}x_{2N})$ but much smaller than the distance between $(x_{2N-1} \, x_{2N})$ and $(x_1 \, x_{2N-2})$. 
With $\eps$ fixed, write
\begin{align} \label{eqn::proba_asy_aux2}
\PP_{\beta}^{\delta}[\conn^{\delta}=\alpha]
= \; & \PP_{\beta}^{\delta}\big[\conn^{\delta}=\alpha, \; \eta^{\delta}_N \not\subset B(\xi, \eps)\big]
\, + \,  \PP_{\beta}^{\delta}\big[\conn^{\delta}=\alpha, \; \eta^{\delta}_N \subset B(\xi, \eps)\big] .
\end{align}
We consider the two terms on the right-hand side separately. 
\begin{itemize}
\item First term on RHS of~\eqref{eqn::proba_asy_aux2}:
A sample of the UST can be generated using loop-erased random walks (LERW) via Wilson's algorithm~\cite{WilsonUSTLERW} 
(see also~\cite[Section~2]{SchrammScalinglimitsLERWUST}).
Note that, on the event $\{\conn^{\delta}=\alpha\}$,
if $\eta^{\delta}_N \not\subset B(\xi,\eps)$, then the LERW branch in the dual tree connecting the arcs 
$(x_{2N-2}^{\delta, *} \, x_{2N-1}^{\delta, *})$ and $(x_{2N}^{\delta, *} \, x_1^{\delta, *})$ goes outside of $B(\xi,\eps)$.
Thus, it follows from the Beurling estimate for simple random walk (see, e.g.,~\cite[Lemma~2.1]{SchrammScalinglimitsLERWUST}) that 
\begin{align*}
\PP_{\beta}^{\delta} \big[ \conn^{\delta}=\alpha, \;  \eta_N^{\delta}\not\subset B(\xi, \eps)\big]
\le C_0 \, \Big( \frac{\diam(x_{2N-1} \, x_{2N})}{\eps}\Big)^{C_1} , \quad\textnormal{for $\delta>0$ small enough},
\end{align*}
where $C_0, C_1>0$ are universal constants and the right-hand side is uniform in $\delta$. 

\smallbreak

\item Second term on RHS of~\eqref{eqn::proba_asy_aux2}:
By our choice of $\eps$, we have
\begin{align*}
\PP_{\beta}^{\delta}\big[\conn^{\delta}=\alpha, \; \eta^{\delta}_N \subset B(\xi, \eps)\big] 
= \; & 
\E_{\beta}^{\delta}\Big[\one\{\eta^{\delta}_N \subset B(\xi, \eps)\} \,
\PP_{\beta}^{\delta}\big[\conn^{\delta}=\alpha\cond \eta^{\delta}_N \big]\Big] .
\end{align*}
By precompactness (Lemma~\ref{lem::Peanocurve_tight}), $\eta^{\delta}_N$ has 
a subsequential weak scaling limit $\eta$, and we may couple the convergent subsequence so that it converges almost surely.
Note that the connected component 
of $\Omega^{\delta}\setminus\eta_N^{\delta}$ with $x_1^{\delta}, \ldots, x_{2N-2}^{\delta}$ on its boundary converges in the Carath\'{e}odory sense to 
the connected component
of $\Omega\setminus\eta$ with $x_1,\ldots, x_{2N-2}$ on its boundary 
(by a standard argument, see, e.g.~\cite{GarbanWuFKIsing}). 
Therefore, Proposition~\ref{prop::proba_limit} shows that, almost surely, 
\begin{align*}
\lim_{\delta\to 0} \PP_{\beta}^{\delta}\big[\conn^{\delta}=\alpha\cond \eta^{\delta}_N \big]
= p^{\hat{\wr}_{2N-1}(\alpha)}_{\hat{\wp}_{2N-1}(\beta)}(\Omega\setminus\eta; \ddot{\boldsymbol{x}}) 
= p^{\hat{\wr}_{2N-1}(\alpha)}_{\hat{\wp}_{2N-1}(\beta)}(\Omega; \psi_{\eta}(\ddot{\boldsymbol{x}})),
\end{align*}
where $\psi_{\eta}$ is the conformal map
from the connected component of $\Omega\setminus\eta$ with $x_1, \ldots, x_{2N-2}$ on its boundary onto $\Omega$
fixing $x_1, x_2$ and $x_{2N-2}$,
and we write $\psi_{\eta}(\ddot{\boldsymbol{x}}) = (\psi_{\eta}(x_1), \ldots, \psi_{\eta}(x_{2N-2}))$. 
In conclusion, as $\delta \to 0$ we obtain the following bounds
for the second term on the RHS of~\eqref{eqn::proba_asy_aux2}:
\begin{align*}
\E \Big[\one\{\eta \subset B(\xi, \eps)\} \,
p^{\hat{\wr}_{2N-1}(\alpha)}_{\hat{\wp}_{2N-1}(\beta)}(\Omega; \psi_{\eta}(\ddot{\boldsymbol{x}})) \Big] 
\leq \; & 
\liminf_{\delta\to 0} 
\PP_{\beta}^{\delta}\big[\conn^{\delta}=\alpha, \; \eta^{\delta}_N \subset B(\xi, \eps)\big] \\
\leq \; & 
\limsup_{\delta\to 0}
\PP_{\beta}^{\delta}\big[\conn^{\delta}=\alpha, \; \eta^{\delta}_N \subset B(\xi, \eps)\big] \\
\leq \;&  \E \Big[\one\{\eta \subset \overline{B(\xi, \eps)}\} \,
p^{\hat{\wr}_{2N-1}(\alpha)}_{\hat{\wp}_{2N-1}(\beta)}(\Omega; \psi_{\eta}(\ddot{\boldsymbol{x}}))\Big] .
\end{align*}
\end{itemize}
Collecting these estimates 
and taking ``$\underset{\delta\to 0}{\liminf}$'',
resp.~``$\underset{\delta\to 0}{\limsup}$'', we obtain 
\begin{align*}
p^{\alpha}_{\beta}(\Omega; \boldsymbol{x})
\geq \; & \E \Big[\one\{\eta \subset B(\xi, \eps)\} \,
p^{\hat{\wr}_{2N-1}(\alpha)}_{\hat{\wp}_{2N-1}(\beta)}(\Omega; \psi_{\eta}(\ddot{\boldsymbol{x}}))\Big] , \\
p^{\alpha}_{\beta}(\Omega; \boldsymbol{x})
\leq \; & \E \Big[\one\{\eta \subset \overline{B(\xi, \eps)}\} \,
p^{\hat{\wr}_{2N-1}(\alpha)}_{\hat{\wp}_{2N-1}(\beta)}(\Omega; \psi_{\eta}(\ddot{\boldsymbol{x}}))\Big]
\, + \, 
C_0 \, \Big( \frac{\diam(x_{2N-1} \, x_{2N})}{\eps}\Big)^{C_1} .
\end{align*}
To finish, after taking first the limit $x_{2N-1}, x_{2N}\to \xi$
(note that $p^{\alpha}_{\beta}(\boldsymbol{x})$ is a smooth function of the marked points by Item~\ref{item:cross_proba_smooth} of
Proposition~\ref{prop::proba_limit}) 
and then the limit $\eps \to 0$, we obtain
\begin{align*}
\lim_{x_{2N-1}, x_{2N}\to \xi} p^{\alpha}_{\beta}(\Omega; \boldsymbol{x}) 
= & \;  \lim_{\eps\to 0}\lim_{x_{2N-1}, x_{2N}\to \xi}\E \Big[\one\{\eta\subset B(\xi, \eps)\} \, p^{\hat{\wr}_{2N-1}(\alpha)}_{\hat{\wp}_{2N-1}(\beta)}(\Omega; \psi_{\eta}(\ddot{\boldsymbol{x}}))\Big] \\
= & \; \lim_{\eps\to 0}\lim_{x_{2N-1}, x_{2N}\to \xi} 
\E \Big[\one\{\eta\subset \overline{B}(\xi, \eps)\} \, p^{\hat{\wr}_{2N-1}(\alpha)}_{\hat{\wp}_{2N-1}(\beta)}(\Omega; \psi_{\eta}(\ddot{\boldsymbol{x}}))\Big] \\
= & \;  p^{\hat{\wr}_{2N-1}(\alpha)}_{\hat{\wp}_{2N-1}(\beta)}(\Omega; \ddot{\boldsymbol{x}}) ,
\end{align*}
which gives~\eqref{eqn::crossingproba_ASY_alpha} and concludes the proof.
\end{proof}

\begin{lemma}\label{lem::product_ASY1}
Fix $N\ge 2$ and $j \in \{1,2, \ldots, 2N-1 \}$  
and suppose $\{j, j+1\}\in\alpha$ and $\LM_{\alpha, \beta}=1$. 
Then, for all $\xi \in (x_{j-1}, x_{j+2})$, we have
\begin{align} \label{eqn::product_ASY1}
\; & \lim_{x_j , x_{j+1} \to \xi} 
\frac{\LM_{\alpha, \beta} \; p^{\alpha}_{\beta}(\boldsymbol{x}) \; \LF_{\beta}(\boldsymbol{x})}{(x_{j+1} - x_j)^{1/4} |\log(x_{j+1}-x_j)|} 
= \LM_{\hat{\wr}_j(\alpha), \hat{\wp}_j(\beta)} \; p^{\hat{\wr}_j(\alpha)}_{\hat{\wp}_j(\beta)}(\boldsymbol{\ddot{x}}_j) \; \LF_{\hat{\wp}_j(\beta)}(\boldsymbol{\ddot{x}}_j) ,
\end{align}
where $\boldsymbol{x} = (x_1, \ldots, x_{2N})$ 
and 
$\boldsymbol{\ddot{x}}_j = (x_1, \ldots, x_{j-1}, x_{j+2}, \ldots, x_{2N})$. 
\end{lemma}

\begin{proof}
By Item~\ref{item:meander_removal} of Lemma~\ref{lem::meander_lemma} we have $\LM_{\alpha, \beta} = \LM_{\hat{\wr}_j(\alpha), \hat{\wp}_j(\beta)}$. 
Thus, the claim follows using the asymptotics properties~\eqref{eqn::USTASY2} and~\eqref{eqn::crossingproba_ASY_alpha}.
\end{proof}

\subsection{Proof of Proposition~\ref{prop::indeptbc2}}
\label{app::indeptbc2}

\begin{proof}[Proof of Proposition~\ref{prop::indeptbc2}]
We denote the quantity of interest as
\begin{align*}
X^{\alpha}_{\beta}(\boldsymbol{x}) 
:= p^{\alpha}_{\beta}(\boldsymbol{x})\;\LF_{\beta}(\boldsymbol{x}) , \qquad \boldsymbol{x}\in\chamber_{2N} .
\end{align*}
On the one hand, by Lemma~\ref{lem::conditionallaw} the limit of $\LL_{\beta}(\varphi_{\delta}(\eta_1^{\delta})\cond \conn^{\delta}=\alpha)$ 
is independent of $\beta$, while 
on the other hand, by Lemma~\ref{lem::probatomart} it is the same as the law of the Loewner chain associated to $p^{\alpha}_{\beta}\LF_{\beta}$.  
This implies that
$\partial_1\log(p^{\alpha}_{\beta}\LF_{\beta})$ is independent of $\beta$ (note that $p^{\alpha}_{\beta}(\boldsymbol{x})$ is a smooth function of the marked points by Item~\ref{item:cross_proba_smooth} of
Proposition~\ref{prop::proba_limit}).
Moreover, by rotation symmetry of the UST model, we see that $\partial_i\log(p^{\alpha}_{\beta}\LF_{\beta})$ is independent of $\beta$ for all $i\in\{1, 2, \ldots, 2N\}$. 
Thus, for any $\boldsymbol{x},\boldsymbol{y}\in\chamber_{2N}$, the quantity $\log (p^{\alpha}_{\beta}(\boldsymbol{x})\LF_{\beta}(\boldsymbol{x}))-\log (p^{\alpha}_{\beta}(\boldsymbol{y})\LF_{\beta}(\boldsymbol{y}))$ is independent of $\beta$. 
This shows that the following ratio is a constant: 
\begin{align*}
C(\alpha, \beta, \gamma):=\frac{X^{\alpha}_{\beta}(\boldsymbol{x})}{X^{\alpha}_{\gamma}(\boldsymbol{x})} , \qquad \boldsymbol{x}\in\chamber_{2N} .
\end{align*}
It suffices to show that $C(\alpha, \beta, \gamma)=1$. We prove this by induction on $N \geq 2$. The case of $N=2$ being trivial, we assume that $N \geq 3$,  and
pick $j$ such that $\{j, j+1\}\in\alpha$. 
Then, for all $\xi \in (x_{j-1}, x_{j+2})$, from the asymptotics~\eqref{eqn::product_ASY1} in Lemma~\ref{lem::product_ASY1} and the induction hypothesis, we find that 
\begin{align*}
C(\alpha, \beta, \gamma)=\lim_{x_j, x_{j+1}\to \xi}\frac{X^{\alpha}_{\beta}(\boldsymbol{x})}{X^{\alpha}_{\gamma}(\boldsymbol{x})}=\frac{X^{\hat{\wr}_j(\alpha)}_{\hat{\wp}_j(\beta)}(\boldsymbol{\ddot{x}}_j)}{X^{\hat{\wr}_j(\alpha)}_{\hat{\wp}_j(\gamma)}(\boldsymbol{\ddot{x}}_j)}=C(\hat{\wr}_j(\alpha), \hat{\wp}_j(\beta), \hat{\wp}_j(\gamma))=1 .
\end{align*}
This shows that the quantity 
$X^{\alpha}_{\beta} = p^{\alpha}_{\beta} \; \LF_{\beta} = X^{\alpha}$ 
is independent of $\beta$ as long as $\LM_{\alpha, \beta} = 1$. It remains to identify it with the pure partition function $\PartF_{\alpha}$ defined in~\eqref{eqn::ppf_def}.
By Proposition~\ref{prop::proba_limit}, 
as a sum of total probability, we have
$\sum_{\alpha} \LM_{\alpha, \beta} \, p_{\beta}^{\alpha} = 1$, so 
\begin{align*}
\LF_{\beta} 
= \; & \sum_{\alpha \in \LP_N} \LM_{\alpha, \beta} \, p_{\beta}^{\alpha} \, \LF_{\beta} 
= \sum_{\alpha \in \LP_N} \LM_{\alpha, \beta} \, X^{\alpha} .
\end{align*} 
Inverting the matrix $\LM$ and recalling the definition~\eqref{eqn::ppf_def}, we obtain $\PartF_{\alpha} = p_{\beta}^{\alpha}\LF_{\beta} = X^{\alpha}$.
\end{proof}


\bigskip{}


\newcommand{\etalchar}[1]{$^{#1}$}


\begin{thebibliography}{DCKK{\etalchar{+}}20}


\bibitem[Ahl78]{AhlforsComplexAnalysis}
Lars~V. Ahlfors.
 {\em Complex analysis}.
 McGraw-Hill Book Co., New York, third edition, 1978.

\bibitem[AS92]{AbramowitzHandbook}
Milton Abramowitz and Irene~A. Stegun.
 {\em Handbook of mathematical functions with formulas, graphs, and
  mathematical tables}.
 Dover Publications, Inc., New York, 1992. Reprint of the 1972 edition.

\bibitem[BB03]{Bauer-Bernard:Conformal_field_theories_of_SLEs}
 Michel Bauer and Denis Bernard.
 Conformal field theories of stochastic {L}oewner evolutions.
 {\em Comm. Math. Phys.}, 239(3):493--521, 2003.

\bibitem[BBK05]{BBK:Multiple_SLEs_and_statistical_mechanics_martingales}
Michel Bauer, Denis Bernard, and Kalle Kyt{\"o}l{\"a}.
 Multiple {S}chramm-{L}oewner evolutions and statistical mechanics martingales.
 {\em J. Stat. Phys.}, 120(5-6):1125--1163, 2005.

\bibitem[BPZ84]{BPZ:Infinite_conformal_symmetry_in_2D_QFT}
Alexander~A. Belavin, Alexander~M. Polyakov, and Alexander~B. Zamolodchikov.
 Infinite conformal symmetry in two-dimensional quantum field theory.
 {\em Nucl. Phys. B}, 241(2):333--380, 1984.

\bibitem[BLR20]{BerestyckiLaslierRayDimersIG}
Nathana{\"e}l Berestycki, Beno\^{\i}t Laslier, and Gourab Ray.
 Dimers and imaginary geometry.
 {\em Ann. Probab.}, 48(1):1--52, 2020.
 
\bibitem[Car84]{Cardy:Conformal_invariance_and_surface_critical_behavior}
John~L. Cardy.
 Conformal invariance and surface critical behavior.
 {\em Nucl. Phys. B}, 240(4):514--532, 1984

\bibitem[Car99]{Cardy:Logarithmic_correlations}
John~L. Cardy,
 Logarithmic correlations in quenched random magnets and polymers.
 {\em Preprint in} \url{arXiv:cond-mat/9911024}, 1999.

\bibitem[Car03]{Cardy:SLE_and_Dyson_circular_ensembles}
John~L. Cardy.
 Stochastic {L}oewner evolution and {D}yson's circular ensembles.
 {\em J. Phys. A}, 36(24):L379--L386, 2003.
 
\bibitem[CHI21]{CHI:Ising_CFT}
Dmitry Chelkak, Cl{\'e}ment Hongler, and Konstantin Izyurov. 
 Correlations of primary fields in the critical planar {I}sing model. 
 {\em Preprint in} \url{arXiv:2103.10263}, 2021. 

\bibitem[CR13]{Creutzig_Ridout:log_CFT_survey}
Thomas Creutzig and David Ridout.
 Logarithmic conformal field theory: Beyond an introduction.
 {\em J.~Phys.~A}, 46(49):494006, 2013. 

\bibitem[DC13]{DuminilCopinParafermionic}
Hugo Duminil-Copin.
 {\em Parafermionic observables and their applications to planar statistical physics models}.
 Volume~25 of {\em Ensaios Matem\'{a}ticos}.
 Sociedade Brasileira de Matem\'{a}tica, Rio de Janeiro, 2013.
 
\bibitem[DF84]{DF-multipoint_correlation_functions}
Vladimir~S. Dotsenko and Vladimir~A. Fateev.
 Conformal algebra and multipoint correlation functions in 2{D} statistical models.
 {\em Nucl. Phys. B}, 240(3):312--348, 1984.

\bibitem[DFGG97]{FrancescoGolinelliGuitterMeanders}
Philippe Di~Francesco, Olivier Golinelli, and Emmanuel Guitter.
 Meanders and the {T}emperley-{L}ieb algebra.
 {\em Comm. Math. Phys.}, 186(1):1--59, 1997.

\bibitem[DFMS97]{DMS:CFT}
Philippe Di~Francesco, Pierre Mathieu, and David S{\'e}n{\'e}chal.
 {\em Conformal field theory}.
 Graduate Texts in Contemporary Physics. Springer-Verlag, New York, 1997.
  
\bibitem[Dub06]{DubedatEulerIntegralsCommutingSLEs}
Julien Dub{\'e}dat.
 Euler integrals for commuting {SLE}s.
 {\em J. Stat. Phys.}, 123(6):1183--1218, 2006.

\bibitem[Dub07]{DubedatCommutationSLE}
Julien Dub{\'e}dat.
 Commutation relations for {S}chramm-{L}oewner evolutions.
 {\em Comm. Pure Appl. Math.}, 60(12):1792--1847, 2007.

\bibitem[Dub09]{DubedatSLEFreefield}
Julien Dub{\'e}dat.
 S{LE} and the free field: partition functions and couplings.
 {\em J. Amer. Math. Soc.}, 22(4):995--1054, 2009.

\bibitem[Dub15]{Dubedat:SLE_and_Virasoro_representations_localization}
Julien Dub{\'e}dat.
 $\mathrm{SLE}$ and {V}irasoro representations: localization.
 {\em Comm. Math. Phys.}, 336(2):695--760, 2015.

\bibitem[Dup04]{Duplantier:Conformal_fractal_geometry_and_boundary_quantum_gravity}
Bertrand Duplantier.
 Conformal fractal geometry and boundary quantum gravity.
 In {\em Fractal Geometry and Applications: A~Jubilee of Beno{\^{i}}t Mandelbrot}, Vol.~72
 of {\em Proceedings of Symposia in Pure Mathematics}, p.~365-482.  
 American Mathematical Society, Providence, R.I., 2004.
  
\bibitem[FK92]{RiemannSurfacesFarkasKra}
Hershel~M. Farkas and Irwin Kra.
 {\em Riemann surfaces}.
 Vol.~71 of {\em Graduate Texts in Mathematics}.
 Springer-Verlag, New York, second edition, 1992.

\bibitem[FF84]{Feigin-Fuchs:Verma_modules_over_Virasoro_book}
Boris~L. Fe{\u\i}gin and Dmitry~B. Fuchs.
 Verma modules over the {V}irasoro algebra.
 In {\em Topology (Leningrad 1982)}. Vol.~1060 of {\em Lecture Notes in Mathematics}, pp.~230--245. Springer-Verlag, Berlin Heidelberg, 1984. 

\bibitem[FPW24]{FPW22} 
Yu~Feng, Eveliina Peltola, and Hao Wu.
\newblock Connection probabilities of multiple {FK}-{I}sing interfaces.
\newblock {\em Probab. Theory Related Fields}, 189(1-2):281--367, 2024.

\bibitem[FK15]{FloresKlebanPDE}
Steven~M. Flores and Peter Kleban.
 A solution space for a system of null-state partial differential equations: {P}arts 1-4.
 {\em Comm. Math. Phys.}, 333(1-2):389--715, 2015.

\bibitem[FSKZ17]{FloresSimmonsKlebanZiffCrossingProba}
Steven~M. Flores, Jacob~J.~H. Simmons, Peter~Kleban, and Robert~M. Ziff.
 A formula for crossing probabilities of critical systems inside polygons.
 {\em J.~Phys.~A}, 50(6):064005, 2017.

\bibitem[GK96]{GK:Indecomposable_fusion_products}
Matthias R.~Gaberdiel and Horst~G. Kausch.
 Indecomposable fusion products. 
 {\em Nucl. Phys.~B}, 477(1):293--318, 1996.

\bibitem[GW20]{GarbanWuFKIsing}
Christophe Garban and Hao Wu.
 On the convergence of {FK}-{I}sing percolation to {SLE}$(16/3,16/3-6)$.
 {\em J. Theor. Probab.}, 33(2):828--865, 2020.

\bibitem[Gur93]{Gurarie:log_CFT}
Victor Gurarie.
 Logarithmic operators in conformal field theory. 
 {\em Nucl. Phys.~B}, 410(3):535--549, 1993.

\bibitem[HLW24]{HanLiuWuUST}
Yong Han, Mingchang Liu, and Hao Wu.
 Hypergeometric {SLE} with $\kappa=8$: Convergence of {UST} and {LERW} in topological rectangles.
 {\em Ann. Inst. H. Poincar\'{e} Probab. Statist.} to appear, 2024.
 {\em Preprint in} \url{arXiv:2008.00403}. 

\bibitem[Izy15]{IzyurovObservableFree}
Konstantin Izyurov.
 Smirnov's observable for free boundary conditions, interfaces and crossing probabilities.
 {\em Comm. Math. Phys.}, 337(1):225--252, 2015.
  
\bibitem[Izy22]{IzyurovMultipleFKIsing}
Konstantin Izyurov.
 On multiple {SLE} for the {FK}-{I}sing model.
 {\em Ann. Probab.}, 50(2):771--790, 2022.

\bibitem[Kar19]{Karrila19}
Alex Karrila.
 Multiple {SLE} type scaling limits: from local to global.
 {\em Preprint in} \url{arXiv:1903.10354}, 2019.

\bibitem[Kar20]{KarrilaUSTBranches}
Alex Karrila.
 U{ST} branches, martingales, and multiple {$\rm SLE(2)$}.
 {\em Electron. J. Probab.}, 25(83):1--37, 2020.

\bibitem[KKP20]{KarrilaKytolaPeltolaCorrelationsLERWUST}
Alex Karrila, Kalle Kyt{\"o}l{\"a}, and Eveliina Peltola.
 Boundary correlations in planar {LERW} and {UST}.
 {\em Comm. Math. Phys.}, 376(3):2065--2145, 2020.

\bibitem[Kau00]{Kausch:Symplectic_fermions}
Horst~G. Kausch.
 Symplectic fermions.
 {\em Nucl. Phys.~B}, 583(3):513--541, 2000. 

\bibitem[Ken01]{KenyonDominosGFF}
Richard Kenyon.
 Dominos and the {G}aussian free field.
 {\em Ann. Probab.}, 29(3):1128--1137, 2001.

\bibitem[KL07]{Kozdron-Lawler:Configurational_measure_on_mutually_avoiding_SLEs}
Michael~J. Kozdron and Gregory~F. Lawler.
 The configurational measure on mutually avoiding {$\mathrm{SLE}$} paths.
 {\em Fields Inst. Commun.}, 50:199--224, 2007.

\bibitem[KRV20]{KRV1}
Antti Kupiainen, R{\'e}mi Rhodes, and Vincent Vargas.
 Integrability of Liouville theory: proof of the DOZZ Formula.
 {\em Ann. of Math. (2)}, 191(1):81--166, 2020. 

\bibitem[GKR23]{GKR:Imaginary_Liouville}
Colin Guillarmou, Antti Kupiainen, and R{\'e}mi Rhodes.
 Compactified imaginary {L}iouville theory. 
 {\em Preprint in} \url{arXiv:2310.18226}, 2023.
 
\bibitem[KP16]{KytolaPeltolaPurePartitionFunctions}
Kalle Kyt{\"o}l{\"a} and Eveliina Peltola.
 Pure partition functions of multiple {SLE}s.
 {\em Comm. Math. Phys.}, 346(1):237--292, 2016.

\bibitem[KP20]{KytolaPeltolaConformalCovBoundaryCorrelation}
Kalle Kyt{\"o}l{\"a} and Eveliina Peltola.
 Conformally covariant boundary correlation functions with a quantum group.
 {\em J. Eur. Math. Soc.}, 22(1):55--118, 2020.

\bibitem[KW11]{KenyonWilsonBoundaryPartitionsTreesDimers}
Richard~W. Kenyon and David~B. Wilson.
 Boundary partitions in trees and dimers.
 {\em Trans. Amer. Math. Soc.}, 363(3):1325--1364, 2011.

\bibitem[Kyt09]{Kytola:SLE_local_martingales_in_logarithmic_representations}
Kalle Kyt{\"o}l{\"a}.
 SLE local martingales in logarithmic representations.
 {\em J.~Stat.~Mech.}, 0908:P08005, 2009.

\bibitem[KR09]{KR:Staggered}
Kalle Kyt{\"o}l{\"a} and David Ridout.
 On staggered indecomposable Virasoro modules. 
 {\em J.~Math. Phys.}, 50:123503, 2009.

\bibitem[Law05]{LawlerConformallyInvariantProcesses}
Gregory~F. Lawler. Conformally invariant processes in the plane. 
 Volume 114 of {\em Mathematical Surveys and Monographs}.
 American Mathematical Society, Providence, RI, 2005.

\bibitem[Law09]{LawlerPartitionFunctionsSLE}
Gregory~F. Lawler.
 Partition functions, loop measure, and versions of {SLE}.
 {\em J.~Stat. Phys.}, 134(5-6):813--837, 2009.

\bibitem[LSW01]{LSW:Brownian_intersection_exponents1}
Gregory~F. Lawler, Oded Schramm, and Wendelin Werner.
 Values of {B}rownian intersection exponents {I}: {H}alf-plane exponents.
 {\em Acta Math.}, 187(2):237--273, 2001.

\bibitem[LSW04]{LawlerSchrammWernerLERWUST}
Gregory~F. Lawler, Oded Schramm, and Wendelin Werner.
 Conformal invariance of planar loop-erased random walks and uniform spanning trees.
 {\em Ann. Probab.}, 32(1B):939--995, 2004.
 
\bibitem[LW23]{LiuWuLERW}
Mingchang Liu and Hao Wu.
Loop-erased random walk branch of uniform spanning tree in
  topological polygons.
 {\em Bernoulli}, 29(2):1555--1577, 2023.


\bibitem[MR07]{MR:Percolation_LCFT}
Pierre Mathieu and David Ridout.
 From percolation to logarithmic conformal field theory.
 {\em Phys. Lett.~B}, 657:120--129, 2007.

\bibitem[MS17]{MillerSheffieldIG4}
Jason Miller and Scott Sheffield.
 Imaginary geometry {IV}: interior rays, whole-plane reversibility, and space-filling trees.
 {\em Probab. Theory Related Fields}, 169(3-4):729--869, 2017.

\bibitem[Nie87]{Nienhuis:Coulomb_gas_formulation_of_2D_phase_transitions}
Bernard Nienhuis.
 Coulomb gas formulation of two-dimensional phase transitions.
 In {\em Phase Transitions and Critical Phenomena}, Vol.~11, pp.~1-53. Academic Press, London, 1987.

\bibitem[PR07]{Pearce_Rasmussen:Solvable_critical_dense_polymers}
Paul~A. Pearce and J{\o}rgen Rasmussen.
 Solvable critical dense polymers.
 {\em J. Stat. Mech.}, 0702:P02015, 2007.

\bibitem[Pel19]{PeltolaCFTSLE}
Eveliina Peltola.
 Towards a conformal field theory for {S}chramm-{L}oewner evolutions.
 {\em J. Math. Phys.}, 60(10):103305, 2019.

\bibitem[PW19]{PeltolaWuGlobalMultipleSLEs}
Eveliina Peltola and Hao Wu.
Global and local multiple {SLE}s for $\kappa \leq 4$ and connection probabilities for level lines of {GFF}.
 {\em Comm. Math. Phys.}, 366(2):469-536, 2019.

\bibitem[PW23]{PeltolaWuCrossingProbaIsing}
Eveliina Peltola and Hao Wu.
 Crossing probabilities of multiple {I}sing interfaces.
 {\em Ann. Appl. Probab.}, 33(4):3169--3206, 2023. 

\bibitem[Pol70]{Polyakov:Conformal_symmetry_of_critical_fluctuations}
Alexander~M. Polyakov.
 Conformal symmetry of critical fluctuations.
 {\em JETP Lett.}, 12(12):381--383, 1970.

\bibitem[RS07]{Read_Saleur:boundary_log_CFT}
Nicholas Read and Hubert Saleur.
 Associative-algebraic approach to logarithmic conformal field theories.
 {\em Nucl. Phys.~B}, 777(3):316--351, 2007.

\bibitem[RS05]{RohdeSchrammSLEBasicProperty}
Steffen Rohde and Oded Schramm.
 Basic properties of {SLE}.
 {\em Ann. of Math. (2)}, 161(2):883--924, 2005.

\bibitem[Roh96]{Rohsiepe:Reducible_but_indecomposable}
Falk Rohsiepe.
 On reducible but indecomposable representations of the {V}irasoro algebra.
 {\em BONN-TH-96-17}, 1996.

\bibitem[Rue13]{Ruelle:Abelian_sandpile_CFT_survey}
Philippe Ruelle.
 Logarithmic conformal invariance in the {A}belian sandpile model.
 {\em J.~Phys.~A}, 46(49):494014, 2013.

\bibitem[SV14]{SantachiaraViti}
Raoul Santachiara and Jacopo Viti.
 Local logarithmic correlators as limits of {C}oulomb gas integrals. 
 {\em Nucl. Phys.~B}, 882(1):229--262, 2014.

\bibitem[Sch08]{Schottenloher:Mathematical_introduction_to_CFT}
Martin Schottenloher.
 {\em A mathematical introduction to conformal field theory}. 
 Vol.~759 of {\em Lecture Notes in Physics}.
 Springer-Verlag, Berlin Heidelberg, 2nd edition, 2008.

\bibitem[Sch00]{SchrammScalinglimitsLERWUST}
Oded Schramm.
 Scaling limits of loop-erased random walks and uniform spanning trees.
 {\em Israel J. Math.}, 118(1):221--288, 2000.

\bibitem[Sch06]{SchrammICM}
Oded Schramm.
 Conformally invariant scaling limits: an overview and a collection of problems.
 In {\em International {Congress} of {Mathematicians}.}  
{Vol}. {I}, pp.~513--543. Eur. Math. Soc., Z{\"u}rich, 2006.

\bibitem[Smi01]{SmirnovPercolationConformalInvariance}
Stanislav Smirnov.
 Critical percolation in the plane: conformal invariance, {C}ardy's formula, scaling limits.
 {\em C. R. Acad. Sci. Paris S\'er. I Math.}, 333(3):239--244, 2001.

\bibitem[Smi06]{SmirnovConformalInvariance}
Stanislav Smirnov.
 Towards conformal invariance of 2D lattice models.
 In {\em International {Congress} of {Mathematicians}.} 
{Vol}. {II},  pp.~1421--1451. Eur. Math. Soc., Z{\"u}rich, 2006.

\bibitem[Wil96]{WilsonUSTLERW}
David~B. Wilson.
 Generating random spanning trees more quickly than the cover time.
 In {\em Proceedings of the {T}wenty-eighth {A}nnual {ACM} {S}ymposium on the {T}heory of {C}omputing ({P}hiladelphia, {PA}, 1996)}, pp.~296--303.
  ACM, New York, 1996.

\bibitem[Wu20]{WuHyperSLE}
Hao Wu.
 Hypergeometric {SLE}: conformal {M}arkov characterization and applications.
 {\em Comm. Math. Phys.}, 374(2):433-484, 2020.

\bibitem[WZ17]{WuZhanSLEBoundaryArmExponents}
Hao Wu and Dapeng Zhan.
 Boundary arm exponents for {SLE}.
 {\em Electron. J. Probab.}, 22(89):1--26, 2017.


\end{thebibliography}
\end{document}